\documentclass{amsart}
\usepackage[utf8]{inputenc}
\usepackage{bbold}
\usepackage{amssymb}
\usepackage{amsmath}
\usepackage{amsthm}
\usepackage{mathtools}
\usepackage{bbm}
\theoremstyle{definition}
\newtheorem{theorem}{Theorem}[section]
\newtheorem{corollary}[theorem]{Corollary}
\newtheorem{lemma}[theorem]{Lemma}

\newtheorem{convention}[theorem]{Convention}
\newtheorem{case}{Case}
\newtheorem{definition}[theorem]{Definition}
\newtheorem{remark}[theorem]{Remark}

\newtheorem{note}[theorem]{Note}
\newtheorem{notation}[theorem]{Notation}

\title{Sieve Methods in Random Graph Theory}

\author[Y. R. Liu]{Yu-Ru Liu}
\address{Department of Pure Mathematics\\ University of Waterloo\\ 200 University Avenue West\\ Waterloo, Ontario\\ Canada\\ N2L 3G1}
\email{yrliu@uwaterloo.ca}

\author[J. C. Saunders]{J.C. Saunders}
\address{Department of Mathematics \& Statistics\\ University of Calgary\\ 2500 University Drive NW\\ Calgary, Alberta\\ T2N 1N4\\ Canada}
\email{john.saunders1@ucalgary.ca}

\date{}

\begin{document}

\begin{abstract}
In this paper, we apply the Tur\'an sieve and the simple sieve developed 
by R. Murty and the first author to study problems in random graph
theory. In particular, we obtain upper and lower bounds on the probability of a graph on $n$ vertices having diameter $d$ for some $d\geq 2$ with edge probability $p$ where the edges are chosen independently. An interesting feature revealed in these results is that the Tur\'an sieve
and the simple sieve ``almost completely'' complement each other. As a corollary to our result, we note that the probability of a random graph having diameter $2$ approaches $1$ as $n\rightarrow\infty$ for constant edge probability $p=1/2$. This is an appendix of a shorter version of this paper.
\end{abstract}


\maketitle

\section{Introduction}
For the purpose of analyzing the random graphs in this paper, we first introduce two sieves known as the simple sieve and the Turan sieve, which were introduced in \cite{comb}. These sieves can be described in terms of a bipartite graph. Let $X$ be a bipartite graph with finite partite sets $A$ and $B$. For $a \in A$ and $b \in B$, we denote by $a\sim b$ if there is an
edge that joins $a$ and $b$.  Define 
\[ \deg b =\#\{a\in A:a\sim b\} \quad \textup{and} \quad  \omega(a)=\#\{b\in B:a\sim b\}.\]
For $b_1,b_2\in B$, we define 
\[ n(b_1,b_2)=\#\{a\in A:a\sim b_1,a\sim b_2\}.\]
In \cite{comb}, R. Murty and the first author derived an elementary
sieve method, called \textit{the simple sieve}, which states that  
\[ \#\{a\in A:\omega(a)=0\}\geq |A|-\sum_{b\in B}\deg b.\]
In the same paper, they also adopted Tur\'an's proof about 
the normal order of distinct
prime factors of a natural number \cite{turan} to prove that
\[ \#\{a\in A:\omega(a)=0\}\leq|A|^2\cdot\frac{\sum_{b_1,b_2\in B}n(b_1,b_2)}{(\sum_{b\in B}\deg b)^2}-|A|.\]
The above result is called \textit{the Tur\'an sieve}.

In this paper, we apply
both the simple sieve and the Tur\'an sieve to study problems about
random graph theory. First, we need the following definition.
\begin{definition}
The \emph{diameter} of a graph $G$ is defined as the maximum number of edges in $G$ that are needed to traverse from one vertex to another in $G$ where we exclude paths that backtrack, detour, and loop.
\end{definition}
Let $G(n,p)$ denote the set of all simple graphs on $n$ vertices where each edge is chosen independently with probability $p$. In 1981, Bollob\'{a}s \cite{bollobas} obtained sharp asymptotic results for the probability of a random graph from $G(n,p)$ vertices having diameter $d$ for any fixed $d\geq 2$ with $n\rightarrow\infty$. Here we extend his results and obtain concrete upper and lower bounds on the probability of a random graph from $G(n,p)$ having diameter at most $d$ where $n$, $p$, and $d$ are fixed. The results of Bollob\'{a}s's follow if we let $n\rightarrow\infty$. We also study analogous questions for random $k$-partite graphs having diameter $d$ with $k\geq 2$. Although our approaches work for general diameter $d$, to better illustrate the methods, Sections $2$, $3$, and $4$ will be dedicated to stating and proving our results for diameter $2$ or diameter $3$ in the case of random bipartite graphs. The rest of the sections will then be devoted to proving generalised results for any $d\geq 2$. In those later sections, for the three types of graphs we consider (graphs in general, $k$-partite graphs for any fixed $k\geq 3$, bipartite graphs) we first impose some restrictions on the values of $n$ and $p$ and then for clarity impose further restrictions on the values of $n$ and $p$ to make our results more meaningful. Here is one of the main theorems of the paper.
\begin{theorem}
\label{theorem:3}
Let $G(n,p)$ denote the set of all simple graphs on $n$ vertices where each edge is chosen independently with probability $p$. Also, let $P(G(n,p))$ be the probability of a graph from $G(n,p)$ having diameter $2$. Then
\begin{equation*}
P(G(n,p))\geq 1-\frac{n^2(1-p^2)^{n-2}(1-p)}{2}
\end{equation*}
and
\begin{equation*}
P(G(n,p))\leq\frac{2}{(n-1)^2(1-p^2)^n(1-p)}+\frac{8}{n}\left(1+\frac{p^3}{(1-p)^2}\right)^n.
\end{equation*}
\end{theorem}
\begin{corollary}
\label{cor:r1}
Let $P(G(n,p))$ be defined as in Theorem \ref{theorem:3}. If $p=\frac{1}{2}$, then we have
\begin{equation*}
P(G(n,1/2))\geq 1-\frac{4n^2(3/4)^n}{9}.
\end{equation*}
\end{corollary}
In the case $p=\frac{1}{2}$, Gilbert \cite{gilbert} showed that `almost all' graphs are connected. Since a graph with diameter $2$ is connected, the above result provides an explicit bound for Gilbert's result.\par
In the situation where the edge probability $p\rightarrow 0$ as $n\rightarrow\infty$, we will show the following corollary.
\begin{corollary}
\label{cor:r2}
Let $P(G(n,p))$ be defined as in Theorem \ref{theorem:3}. Let $\lim_{n\rightarrow\infty}p=0$. We have
\begin{equation}
P(G(n,p))\geq 1-(1+o(1))\frac{n^2}{2}e^{-np^2}\label{prop121}
\end{equation}
and
\begin{equation}
P(G(n,p))\leq(1+o(1))\left(\frac{2}{n^2}e^{np^2}\right)\left(1+4ne^{np^2\left(p^2-1\right)}\right).\label{prop122}
\end{equation}
Suppose further that
\begin{equation*}
\lim_{n\rightarrow\infty}(2\log n-np^2-\log 2)=c
\end{equation*}
for some $c\in\mathbb{R}\backslash\{0\}$.
\newline
1) If $c>0$, we have
\begin{equation*}
P(G(n,p))\leq(1+o(1))e^{-c}.
\end{equation*}
2) If $c<0$, we have
\begin{equation*}
P(G(n,p))\geq(1+o(1))(1-e^c).
\end{equation*}
\end{corollary}
We will also study analogous problems for random directed graphs in the appropriate sections of this paper. As we noted in Corollary \ref{cor:r2}, the upper bound we obtained through the Tu\'an sieve works effectively for $c>0$, while the lower bound we obtained through the simple sieve gives a non-trivial result for $c<0$. It is interesting to see that the Tur\'an sieve and the simple sieve ``almost completely" complement each other in this way.

\section{Graphs with Diameter 2 with the Sieves}
In this section, we use the Tur\'an sieve and the simple sieve to prove Theorem \ref{theorem:3}.
\begin{proof}
For a fixed $n\in\mathbb{N}$, let $G(n,p)$ denote the set of all graphs on $n$ vertices with edge probability $p$, and let $P(G(n,p))$ be the probability of a graph from $G(n,p)$ having diameter $2$. Consider the function $g_n:[0,1]\rightarrow[0,1]$ defined as $g_n(x):=P(G(n,p),x)$. There are $2^{\frac{n(n-1)}{2}}$ graphs in total in $G(n,p)$. Let us say $M$ of these have diameter $2$ and label these as $G_1$, $G_2$,\ldots,$G_M$. For $1\leq i\leq M$, let $k_i$ denote the number of edges in $G_i$. Then the probability of selecting the graph $G_i$ from $G(n,p)$ according to the edge probability $x$ is $x^{k_i}(1-x)^{\frac{n(n-1)}{2}-k_i}$. Therefore,
\begin{equation*}
g_n(x)=x^{k_1}(1-x)^{\frac{n(n-1)}{2}-k_1}+x^{k_2}(1-x)^{\frac{n(n-1)}{2}-k_2}+\cdots+x^{k_M}(1-x)^{\frac{n(n-1)}{2}-k_M}.
\end{equation*}
Thus, for each $n\in\mathbb{N}$ the function $g_n$ is continuous. Therefore, we may assume that $p\in\mathbb{Q}\cap(0,1)$ since $\mathbb{Q}\cap(0,1)$ is dense in $[0,1]$.

Let $p=\frac{r}{s}$ where $r=r(n), s=s(n)\in\mathbb{N}$. We let $A$ be the set of all graphs in $G(n,p)$, allowing for a number of duplicates of each possible graph to accommodate the edge probability $p$. We accomplish this by letting there be $r^{\binom{n}{2}}$ copies of the complete graph, $r^{\binom{n}{2}}\left(\frac{s}{r}-1\right)$ copies of each graph with $\binom{n}{2}-1$ edges, $r^{\binom{n}{2}}\left(\frac{s}{r}-1\right)^2$ copies of each graph with $\binom{n}{2}-2$ edges, and so on. By the binomial theorem we have
\begin{equation*}
|A|=\sum_{k=0}^{\binom{n}{2}}{\binom{\binom{n}{2}}{k}}r^k(s-r)^{\binom{n}{2}-k}=s^{\binom{n}{2}}.
\end{equation*}
We let $B$ be all pairs of vertices so $|B|=\binom{n}{2}$. For a graph $a\in A$ and a pair of vertices $b\in B$, we say $a\sim b$ if the pair of vertices $b$ in $a$ do not share a common neighbouring vertex and are not neighbours themselves. Thus, we will have $\omega(a)=0$ if and only if $a$ is connected with diameter at most $2$.

Pick a pair of vertices $b\in B$ and call them $v_1$ and $v_2$. To calculate $\deg b$, we need to calculate the number of graphs in $A$ such that the pair of vertices do not have a common neighbouring vertex and are not neighbours themselves. For each of the potential $(n-2)$ neighbouring vertices, we need to consider two edges, making sure at least one of them is not in the graph. Since each potential edge contributes a factor of $r$ or $(s-r)$ depending on whether it is in a specified graph, we have
\begin{align*}
D(r,s,n):=\deg b&=((s-r)^2+2r(s-r))^{n-2}(s-r)(s^{\binom{n}{2}-2(n-2)-1})\\
&=(s^2-r^2)^{n-2}(s-r)s^{\binom{n}{2}-2(n-2)-1}.
\end{align*}
It follows that
\begin{equation*}
\sum_{b\in B}\deg b=\frac{s^{\binom{n}{2}}n(n-1)(1-p^2)^{n-2}(1-p)}{2}.
\end{equation*}
By the simple sieve, we obtain
\begin{align}
P(G(n,p))&\geq 1-\frac{n(n-1)(1-p^2)^{n-2}(1-p)}{2}\nonumber\\
&>1-\frac{n^2(1-p^2)^{n-2}(1-p)}{2}.\label{simplesieve}
\end{align}

We now try to get an upper bound for $P(G(n,p))$, in which we need to estimate $\sum_{b_1,b_2\in B}n(b_1,b_2)$. In the following, we calculate $n(b_1,b_2)$, depending on how many vertices $b_1$ and $b_2$ have in common.
\begin{case}
\normalfont
Suppose that $b_1$ and $b_2$ are two pairs of vertices that have no vertices in common, i.e., $b_1$ and $b_2$ consist of $4$ distinct vertices. For each of $b_1$ and $b_2$, the probability that the pair of vertices in question are not connected by an edge nor have any common neighbouring vertices is
\begin{equation*}
\frac{D(r,s,n)}{s^{\binom{n}{2}}}.
\end{equation*}
As is the case for calculating $\deg b$, for each of the pair of vertices $b_1$ and $b_2$, we need to consider pairs of edges for each potential neighbouring vertex. If the potential neighbouring vertex is among the remaining $n-4$ vertices, then the pair of edges to consider with respect to $b_1$ will be disjoint from the pair of edges to consider with respect to $b_2$. The only real problem to consider is when the potential neighbouring vertex is among the pair of vertices $b_1$ and $b_2$ where we have four possible edges to consider. These observations give rise to
\begin{align*}
n(b_1,b_2)&=\frac{D(r,s,n)^2}{s^{\binom{n}{2}}}\cdot\frac{s^4((s-r)^4+4r(s-r)^3+2r^2(s-r)^2)}{(s^2-r^2)^4},
\end{align*}
and thus
\begin{align*}
&\quad\sum_{b_1,b_2\in B,\text{ 4 vertices}}n(b_1,b_2)\\
&<{\binom{n}{2}}^2\frac{D(r,s,n)^2}{s^{\binom{n}{2}}}\\
&\qquad\cdot\frac{p^{-4}((p^{-1}-1)^4+4(p^{-1}-1)^3+2(p^{-1}-1)^2)}{(p^{-2}-1)^4}\\
&<{\binom{n}{2}}^2\frac{D(r,s,n)^2}{s^{\binom{n}{2}}}\cdot\left(1+\frac{4p^3}{(1-p)^2}\right).
\end{align*}
\end{case}
\begin{case}
\normalfont
Take two pairs of vertices $b_1$ and $b_2$ that have exactly one vertex in common, i.e., $b_1$ and $b_2$ consist of $3$ distinct vertices. We can do a similar kind of analysis of edge selection as in Case $1$ to calculate
\begin{align*}
&\quad\sum_{b_1,b_2\in B,\text{ 3 vertices}}n(b_1,b_2)\\
&=\frac{D(r,s,n)^2n(n-1)(n-2)}{s^{\binom{n}{2}}}\left(1+\frac{1}{p^{-3}+p^{-2}-p^{-1}-1}\right)^{n-3}\\
&\leq\frac{D(r,s,n)^2n(n-1)(n-2)}{s^{\binom{n}{2}}}\left(1+\frac{p^3}{(1-p)}\right)^{n-3}.
\end{align*}
\end{case}
\begin{case}
\normalfont
Suppose $b_1$ and $b_2$ have two vertices in common. Then the two pairs are identical, and we have
\begin{equation*}
n(b_1,b_2)=\deg b.
\end{equation*}
It follows that
\begin{align*}
\sum_{b_1,b_2\in B,\text{ 2 vertices}}n(b_1,b_2)&=\sum_{b\in B}\deg b=\frac{s^{\binom{n}{2}}n(n-1)(1-p^2)^{n-2}(1-p)}{2}.
\end{align*}
\end{case}
\noindent Combining Cases $1-3$, we get
\begin{align*}
\sum_{b_1,b_2\in B}n(b_1,b_2)&<{\binom{n}{2}}^2\frac{D(r,s,n)^2}{s^{\binom{n}{2}}}\cdot\left(1+\frac{4p^3}{(1-p)^2}\right)\\
&\quad +\frac{D(r,s,n)^2n(n-1)(n-2)}{s^{\binom{n}{2}}}\left(1+\frac{p^3}{(1-p)}\right)^{n-3}\\
&\quad+\frac{s^{n_1n_2}n(n-1)(1-p^2)^{n-2}(1-p)}{2}.
\end{align*}
By the Tur\'an sieve, we deduce
\begin{align*}
&\quad P(G(n,p))\\
&\leq\frac{2}{n(n-1)(1-p^2)^{n-2}(1-p)}+\frac{4}{n}\left(1+\frac{p^3}{(1-p)}\right)^{n-3}\\
&\quad+\frac{4p^3}{(1-p)^2}.
\end{align*}
Notice that
\begin{equation*}
\frac{p^3}{(1-p)^2}<\frac{1}{n}\left(1+n\frac{p^3}{(1-p)^2}\right)<\frac{1}{n}\left(1+\frac{p^3}{(1-p)^2}\right)^n.
\end{equation*}
It follows that
\begin{align}
P(G(n,p)) &\leq\frac{2}{(n-1)^2(1-p^2)^n(1-p)}+\frac{8}{n}\left(1+\frac{p^3}{(1-p)^2}\right)^n\label{turansieve}.
\end{align}
By \eqref{simplesieve} and \eqref{turansieve} Theorem \ref{theorem:3} follows.
\end{proof}
We now prove Corollary \ref{cor:r2}.
\begin{proof}
By Theorem \ref{theorem:3} we have
\begin{equation*}
P(G(n,p))>1-\frac{n^2\left(1-p^2\right)^{p^{-2}\cdot np^2}\left(1-p^2\right)^{-2}(1-p)}{2}.
\end{equation*}
Since $p^{-2}\geq 1$, we have
\begin{equation}\label{estimate1}
e^{-np^2}\left(1-p^2\right)^{np^2}<\left(1-p^2\right)^{p^{-2}np^2}<e^{-np^2}.
\end{equation}
Since $\lim_{n\rightarrow\infty}p=0$, we have that
\begin{equation*}
\lim_{n\rightarrow\infty}\left(1-p^2\right)^{-2}=1\quad\text{and}\quad\lim_{n\rightarrow\infty}(1-p)=1,
\end{equation*}
from which we get
\begin{equation}
P(G(n,p))\geq 1-\frac{n^2}{2}e^{-np^2}(1+o(1))\label{eqn8}.
\end{equation}
For the upper bound, first note that
\begin{equation*}
\frac{8}{n}\left(1+\frac{p^3}{(1-p)^2}\right)^n<\frac{2e^{np^2}}{n^2}\cdot 4ne^{\left(\frac{np^3}{\left(1-p\right)^2}-np^2\right)}.
\end{equation*}
Combining this with Equations \eqref{turansieve} and \eqref{estimate1}, we get
\begin{align*}
P(G(n,p))&<\frac{2e^{np^2}}{(n-1)^2\left(1-p^2\right)^{np^2}\left(1-p\right)}\\
&\quad+\frac{2e^{np^2}}{n^2}\cdot 4ne^{\left(\frac{np^3}{\left(1-p\right)^2}-np^2\right)}.
\end{align*}
Note that for $n\in\mathbb{N}$ with $\frac{2}{n^2}e^{np^2}\geq 1$, we have
\begin{equation}
\left(\frac{2}{n^2}e^{{np}^2}\right)\left(1+4ne^{np^2\left(p-1\right)}\right)>1\label{trivial}.
\end{equation}
In particular for those $n$, the bound in Theorem \ref{theorem:3} is trivial. Thus, it suffices to consider $n\in\mathbb{N}$ such that
\begin{equation*}
\frac{2}{n^2}e^{np^2}<1.
\end{equation*}
Label all such $n\in\mathbb{N}$ as $n_1,n_2,\ldots,n_j,\ldots$ such that $n_1<n_2<\cdots$ If there are only finitely many, then for sufficiently large $n$, we will have Equation \eqref{trivial} and so the bound in Theorem \ref{theorem:3} is trivial. Thus, we may assume that $n_1,n_2,...,n_j,...$ is  an infinite list. Since $p$ depends on $n$, at least in this proof, we will sometimes denote $p$ by $p(n)$ in the rest of this proof. Then for all $j\in\mathbb{N}$, we have
\begin{equation*}
n_jp(n_j)^2<2\log n_j-\log 2,
\end{equation*}
and so
\begin{equation*}
\lim_{j\rightarrow\infty}n_jp(n_j)^3=\lim_{j\rightarrow\infty}(n_jp(n_j)^2)^{3/2}n_j^{-1/2}=0.
\end{equation*}
We also have
\begin{equation}
\lim_{j\rightarrow\infty}n_jp(n_j)^4=0\quad\text{ and }\quad\lim_{j\rightarrow\infty}p(n_j)^2=0\label{jlimit}.
\end{equation}
Note that if $0\leq x\leq 1$ and $y\geq 1$, then $(1-x)^y\geq 1-xy$. Thus, if $n_jp(n_j)^2\geq 1$, then
\begin{equation*}
\left(1-p(n_j)^2\right)^{n_jp(n_j)^2}\geq 1-n_jp(n_j)^4.
\end{equation*}
Suppose that $n_jp(n_j)^2<1$. Then we have
\begin{equation*}
\left(1-p(n_j)^2\right)^{n_jp(n_j)^2}\geq 1-p(n_j)^2.
\end{equation*}
Thus, by Equation \eqref{jlimit}, we have
\begin{equation*}
\lim_{j\rightarrow\infty}\left(1-p(n_j)^2\right)^{n_jp(n_j)^2}=1
\end{equation*}
and
\begin{equation*}
\lim_{j\rightarrow\infty}n_jp(n_j)^3\left(1-\frac{1}{\left(1-p(n_j)\right)^2}\right)=0.
\end{equation*}
Also, notice that
\begin{align*}
&\quad n_jp(n_j)^3\left(1-\frac{1}{\left(1-p(n_j)\right)^2}\right)\\
&=n_jp(n_j)^2\left(p(n_j)-1\right)-\left(\frac{n_jp(n_j)^3}{(1-p(n_j))^2}-n_jp(n_j)^2\right).
\end{align*}
We thus obtain
\begin{equation}
 P(G(n,p))\leq(1+o(1))\left(\frac{2}{n^2}e^{np^2}\right)\left(1+4ne^{np^2\left(p-1\right)}\right)\label{eqn10}.
\end{equation}
Now we suppose further that
\begin{equation}
\lim_{n\rightarrow\infty}(2\log n-np^2-\log 2)=c\label{eqn11}
\end{equation}
for some $c\in\mathbb{R}\backslash\{0\}$. Then we have
\begin{equation*}
\lim_{n\rightarrow\infty}\left(\log n-\frac{np^2}{2}\right)=\tilde{c}
\end{equation*}
for some $\tilde{c}\in\mathbb{R}$. Since $\lim_{n\rightarrow\infty}p=0$, it follows that
\begin{align*}
&\quad\lim_{n\rightarrow\infty}\left(\log n+np^3-np^2\right)\\
&=\lim_{n\rightarrow\infty}\left(\left(\log n-\frac{np^2}{2}\right)+\left(np^2-\frac{np^2}{2}\right)\right)\\
&=-\infty.
\end{align*}
Thus, we have
\begin{equation}
ne^{np^2\left(p^2-1\right)}=o(1)\label{eqn39}.
\end{equation}
Also, by Equation \eqref{eqn11}, we have 
\begin{equation}
\frac{2}{n^2}e^{np^2}=e^{-c}(1+o(1))\label{eqn57}
\end{equation}
and
\begin{equation}
\frac{n^2}{2}e^{-np^2}=e^c(1+o(1))\label{eqn104}.
\end{equation}
By Equations \eqref{eqn8} and \eqref{eqn104}, we obtain
\begin{equation*}
P(G(n,p))\geq1-(1+o(1))e^c.
\end{equation*}
Also, by Equations \eqref{eqn10}, \eqref{eqn39}, and \eqref{eqn57}, we obtain
\begin{equation*}
P(G(n,p))\leq(1+o(1))e^{-c}.
\end{equation*}
This finishes the proof of Corollary \ref{cor:r2}.
\end{proof}
\begin{remark}\label{rk1}
Assume that $n\geq 200$ and $p\leq 1/2$. The $o(1)$ in \eqref{prop121} can be made explicit as $4p^2$ and the $o(1)$ in \eqref{prop122} can be made explicit as $\frac{4(\log n)^2+2}{n}+p+\frac{3e^8(2\log n)^{3/2}}{n^{1/2}}$.
\end{remark}
\begin{remark}
Using the above methods, we can obtain similar results about the probability of a random directed graph on $n$ vertices having diameter $2$ where each directed edge is chosen independently with probability $p$. Furthermore, for any two vertices, say $v_1$ and $v_2$, the existence of the edge from $v_1$ to $v_2$ has probability $p$, while the existence of the edge from $v_2$ to $v_1$ also occurs with probability $p$, and these two edges occur independently. More precisely, in Theorem \ref{theorem:3}, Corollary \ref{cor:r1}, and Corollary \ref{cor:r2}, we multiply the second term of the lower bound by $2$, divide the upper bound by $2$, and we add $\log 2$ to our expressions for $c$. Everything else is left unchanged.
\end{remark}	
\section{Analysis of $k$-partite graphs for Diameter $2$}
Here we apply our analysis to $k$-partite graph sets for $k\geq 3$. First, we present a definition.
\begin{definition}
Let $k\geq 2$. A \textit{simple $k$-partite graph} is an undirected graph whose vertices can be divided into $k$ sets, such that there are no edges between two vertices in the same set.
\end{definition}
We exclude the bipartite case ($k=2$) because the only bipartite graph that has diameter $2$ is the complete bipartite graph; we analyze that case by itself in the next section.
\begin{convention}
\rm For each $k$-partite graph, we label the $k$ partite sets of the graph in a non-decreasing order in terms of the number of vertices each set contains. Thus, the $i$th set is a set containing $n_i$ vertices.
\end{convention}
\begin{theorem}\label{theorem:5}
Fix $k\geq 3$ and for each $n\in\mathbb{N}$, $n\geq k+2$, pick $n_1,n_2,...,n_k\in\mathbb{N}$ such that $n_1\leq n_2\leq\ldots\leq n_k$, $n_{k-1}\geq 2$, and $n_1+n_2+\cdots+n_k=n$. Let $\mathbf{n^{(k)}}=(n_1,n_2,\ldots,n_k)$ and let $G(\mathbf{n^{(k)}},p)$ denote the set of all $k$-partite graphs with the partite sets having $n_1,n_2,\ldots,n_k$ vertices respectively where each edge is chosen independently with probability $p$. Also, let $P(G(\mathbf{n^{(k)}},p))$ be the probability of a graph from $G(\mathbf{n^{(k)}},p)$ having diameter $2$. Then
\begin{align*}
&\quad P(G(\mathbf{n^{(k)}},p))\\
&\geq 1-\frac{n_k^2(1-p^2)^{n-n_k}}{2}\\
&\quad\cdot\left(1+\frac{2n_{k-1}(1-p^2)^{-n_{k-1}}}{n_k}+\frac{7k^2n_{k-1}^2(1-p^2)^{n_k-n_{k-1}-n_{k-2}}}{3n_k^2}\right)
\end{align*}
and
\begin{align*}
&\quad P(G(\mathbf{n^{(k)}},p))\\
&\leq\frac{2}{n_k(n_k-1)(1-p^2)^{n-n_k}}\left(1+\frac{2n_{k-1}(1-p^2)^{-n_{k-1}}(1-p)}{(n_k-1)}\right)^{-1}\\
&\qquad +\frac{3k^3\left(1+\frac{p^3}{(1-p)}\right)^{n-n_k}(1-p^2)^{-2}}{(n_{k-1}-1)}.
\end{align*}
\end{theorem}
\begin{proof}
As in the proof of Theorem \ref{theorem:3}, we may assume that $p\in\mathbb{Q}\cap(0,1)$ for all $n\in\mathbb{N}$.

Let $p=\frac{r}{s}$ where $r,s\in\mathbb{N}$. As in the proof of Theorem \ref{theorem:3}, we let $A$ be the set of all graphs in $G(\mathbf{n^{(k)}},p)$, allowing for a number of duplicates of each possible graph to accommodate the edge probability $p$. Since the complete $k$-partite graph has $t:=\sum_{1\leq i<j\leq k}n_in_j$ edges, we have $r^t$ copies of the complete bipartite graph and $|A|=s^t$.

We let $B$ be all pairs of vertices. Thus, $|B|=\frac{n(n-1)}{2}$. For a graph $a\in A$ and a pair of vertices $b\in B$, we say $a\sim b$ if the pair of vertices $b$ in $a$ do not share a common neighbouring vertex and are not connected by a single edge. Thus, we will have $\omega(a)=0$ if and only if $a$ is connected with diameter at most $2$. For each pair of vertices $b\in B$ that are in the $i$th partite set for some $1\leq i\leq k$, we will have
\begin{align*}
\quad D(r, s, n, n_i)&:=\deg b\\
&=((s-r)^2+2r(s-r))^{n-n_i}((s-r)+r)^{t-2n+2n_i}\\
&=(1-p^2)^{n-n_i}s^t.
\end{align*}
For each pair of vertices $b\in B$ with one vertex being in the $i$th partite set and the other in the $j$th partite set where $i<j$, we have
\begin{align*}
&\quad D(r, s, n, n_i, n_j)\\
&:=\deg b\\
&=((s-r)^2+2r(s-r))^{n-n_i-n_j}((s-r)+r)^{t-2n+2n_i+2n_j}(1-p)\\
&=(1-p^2)^{n-n_i-n_j}(1-p)s^t.
\end{align*}
It follows that
\begin{align*}
&\quad\sum_{b\in B}\deg b\\
&=s^t\sum_{i=1}^k\binom{n_i}{2}(1-p^2)^{n-n_i}+s^t\sum_{1\leq i<j\leq k}n_in_j(1-p^2)^{n-n_i-n_j}(1-p).
\end{align*}
By the simple sieve, we obtain
\begin{align*}
&\quad P(G(\mathbf{n^{(k)}},p))\\
&\geq 1-\frac{n_k^2(1-p^2)^{n-n_k}}{2}\\
&\quad\cdot\left(1+\frac{2n_{k-1}(1-p^2)^{-n_{k-1}}}{n_k}+\frac{7k^2n_{k-1}^2(1-p^2)^{n_k-n_{k-1}-n_{k-2}}}{3n_k^2}\right).
\end{align*}
To get an upper bound for $P(G(\mathbf{n^{(k)}},p))$, we need to estimate $\sum_{b_1,b_2\in B}n(b_1,b_2)$. Similar to the proof of Theorem \ref{theorem:3}, by calculating $n(b_1,b_2)$ based on the number of vertices $b_1$ and $b_2$ have in common, we can get
\begin{align*}
&\quad\sum_{b_1,b_2\in B}n(b_1,b_2)\\
&\leq\frac{\left(\sum_{b\in B}\deg b\right)^2}{s^{n_1n_2}}\left(1+\frac{4p^3}{(1-p)^2)^2}\right)\\
&\quad +\frac{n^3D(r,s,n,n_k,n_{k-1})D(r,s,n,n_k,n_{k-2})}{s^t(1-p^2)^{2}}\left(1+\frac{p^3}{(1-p)}\right)^{n-n_k-n_{k-1}-n_{k-2}}\\
&\quad +\binom{k}{2}\frac{n_k^2n_{k-1}D(r,s,n,n_k,n_{k-1})^2}{s^t}\left(1+\frac{p^3}{(1-p)}\right)^{n-n_k-n_{k-1}}\\
&\quad +\frac{k^2\binom{n_k}{2}n_{k-1}D(r,s,n,n_k)D(r,s,n,n_k,n_{k-1})}{s^t(1-p^2)}\left(1+\frac{p^3}{(1-p)}\right)^{n-n_k-n_{k-1}}\\
&\quad +\frac{kn_k^3D(r,s,n,n_k)^2}{s^t}\left(1+\frac{p^3}{(1-p)}\right)^{n-n_k}.
\end{align*}
Then, by the Tur\'an sieve, we get
\begin{align*}
P(G(\mathbf{n^{(k)}},p))&<\frac{2}{n_k(n_k-1)(1-p^2)^{n-n_k}}\\
&\qquad\cdot\left(1+\frac{2n_{k-1}(1-p^2)^{-n_{k-1}}(1-p)}{(n_k-1)}\right)^{-1}\\
&\qquad +\frac{3k^3\left(1+\frac{p^3}{(1-p)}\right)^{n-n_k}(1-p^2)^{-2}}{(n_{k-1}-1)}.
\end{align*}
This completes the proof of Theorem \ref{theorem:5}.
\end{proof}
By substituting $p=\frac{1}{2}$, we deduce from Theorem \ref{theorem:5} the following.
\begin{corollary}
Let $P(G(\mathbf{n^{(k)}},p))$ be defined as in Theorem \ref{theorem:5}. If $p=\frac{1}{2}$, then we have
\begin{align*}
&\quad P(G(\mathbf{n^{(k)}},p),1/2)\\
&\geq 1-\frac{n_k^2(3/4)^{n-n_k}}{2}\left(1+\frac{2n_{k-1}(3/4)^{-n_{k-1}}}{n_k}+\frac{7k^2n_{k-1}^2(3/4)^{n_k-n_{k-1}-n_{k-2}}}{3n_k^2}\right).
\end{align*}
\end{corollary}
In the case when $p\rightarrow 0$ as $n\rightarrow\infty$, we have the following.
\begin{corollary}\label{prop:5}
Let $P(G(\mathbf{n^{(k)}},p))$ be defined as in Theorem \ref{theorem:5}. Let $\lim_{n\rightarrow\infty} p^4(n-n_k)=0$. We have
\begin{align*}
P(G(\mathbf{n^{(k)}},p))&\geq 1-\frac{n_k^2e^{-p^2(n-n_k)}}{2}\\
&\quad\cdot\left(1+\frac{2n_{k-1}}{n_k}e^{p^2n_{k-1}}\left(1+\frac{7k^2n_{k-1}e^{-p^2(n_k-n_{k-2})}}{6n_k}\right)\right)
\end{align*}
and
\begin{align}
&\quad P(G(\mathbf{n^{(k)}},p))\\
&\leq(1+o(1))\frac{2e^{p^2(n-n_k)}}{n_k^2}\left(1+\frac{2n_{k-1}}{n_k}e^{p^2n_{k-1}}\right)^{-1}\nonumber\\
&\quad\cdot\left(1+\frac{3k^3n_k^2e^{\left(p^3-p^2\right)(n-n_k)}}{2(n_{k-1}-1)}+\frac{3k^3n_kn_{k-1}e^{\left(p^3-p^2\right)(n-n_k)+p^2n_{k-1}}}{(n_{k-1}-1)}\right)\nonumber.
\end{align} 
Suppose further that
\begin{equation}
\lim_{n\rightarrow\infty}\left(\log n_{k-1}-\log n-p^2n_{k-1}\right)=-\infty,\label{limit1}
\end{equation}
\begin{equation}
\lim_{n\rightarrow\infty}\left(2\log n+(p^3-p^2)(n-n_k)-\log n_{k-1}\right)=-\infty,\label{limit2}
\end{equation}
\begin{equation}
\lim_{n\rightarrow\infty}\left(\left(p^3-p^2\right)(n-n_k)+p^2n_{k-1}+\log n\right)=-\infty,\label{limit3}
\end{equation}
and that
\begin{equation*}
\lim_{n\rightarrow\infty}\left(2\log n_k-p^2(n-n_k)-\log 2+\log\left(1+\frac{2n_{k-1}}{n_k}e^{p^2n_{k-1}}\right)\right)=c
\end{equation*}
for some $c\in\mathbb{R}$. 
\newline
1) If $c<0$, we have
\begin{equation*}
P(G(\mathbf{n^{(k)}},p))\geq 1-(1+o(1))e^c.
\end{equation*}
2) If $c>0$, we have
\begin{equation*}
P(G(\mathbf{n^{(k)}},p))\leq(1+o(1))e^{-c}.
\end{equation*}
\end{corollary}
\begin{proof}
Since $\lim_{n\rightarrow\infty}p^4(n-n_k)=0$ and $p^{-2}\geq 1$, for $n\rightarrow\infty$, by similar reasoning as in the proof of Theorem \ref{theorem:3}, we have
\begin{align*}
e^{-p^2(n-n_k)}>\left(1-p^2\right)^{p^{-2}\cdot p^2(n-n_k)}&>e^{-p^2(n-n_k)}\left(1-p^2\right)^{p^2(n-n_k)}\\
&=e^{-p^2(n-n_k)}(1+o(1))
\end{align*}
and
\begin{align*}
e^{p^2n_{k-1}}<\left(1-p^2\right)^{-p^{-2}\cdot p^2n_{k-1}}&<e^{p^2n_{k-1}}\left(1-p^2\right)^{-p^2n_{k-1}}\\
&=e^{p^2n_{k-1}}(1+o(1)).
\end{align*}
Thus the first term in the upper bound of $P(G(\mathbf{n^{(k)}},p))$ in Theorem \ref{theorem:5} becomes
\begin{align}
&\quad\frac{2}{n_k(n_k-1)\left(1-p^2\right)^{p^{-2}\cdot p^2(n-n_k)}}\\
&\quad\cdot\left(1+\frac{2n_{k-1}\left(1-p^2\right)^{\frac{-n}{np^2}\cdot p^2n_{k-1}}\left(1-p\right)}{(n_k-1)}\right)^{-1}\nonumber\\
&=(1+o(1))\frac{2e^{p^2(n-n_k)}}{n_k^2}\left(1+\frac{2n_{k-1}}{n_k}e^{p^2n_{k-1}}\right)^{-1}.\label{1st}
\end{align}
For the second term, first note that since $\lim_{n\rightarrow\infty}p^4(n-n_k)=0$, we have
\begin{align*}
\lim_{n\rightarrow\infty}(n-n_k)p^2\left(\frac{p}{\left(1-p\right)}-1\right)-(n-n_k)p^2\left(p-1\right)
=0.
\end{align*}
Thus the second term in the upper bound of $P(G(\mathbf{n^{(k)}},p))$ in Theorem \ref{theorem:5} becomes
\begin{equation}
\frac{3k^3e^{\frac{p^3(n-n_k)}{1-p}}(1-p^2)^{-2}}{n_{k-1}-1}=\frac{3k^3e^{p^3(n-n_k)}}{n_{k-1}-1}(1+o(1)).\label{2nd}
\end{equation}
Combining Equations \eqref{1st} and \eqref{2nd}, the upper bound of Corollary \ref{prop:5} follows. Also, by Equations \eqref{limit1}, \eqref{limit2}, and \eqref{limit3}, Statements (1) and (2) follow as in the proof of Corollary \ref{cor:r2}.
\end{proof}
\begin{remark}
Similar to Remark \ref{rk1}, all $o(1)$ terms in Corollary \ref{prop:5} can be made explicit.
\end{remark}
We consider one more application of the sieves to random $k$-partite graphs.
\begin{definition}
\label{definition:3}
The \textit{$k$-partite Tur\'an graph} (named after the same P\'al Tur\'an) on $n$ vertices is defined as the $k$-paritite graph on $n$ vertices such that the partitioned sets are as equal as possible. In other words, for each $1\leq i\leq k$, we have $n_i=\lfloor\frac{n}{k}\rfloor$ or $n_i=\lceil\frac{n}{k}\rceil$.
\end{definition}
In the case of $k$-partite Tur\'an graphs, we can calculate $\sum_{b\in B}\deg b$ a lot more precisely, using the above methods. Then we can prove the following.
\begin{theorem}\label{theorem:6}
Let $G'(n,k,p)$ denote the set of all Tur\'an $k$-partite graphs where each edge is chosen independently with probability $p$. Also, let $P(G'(n,k,p))$ be the probability of a graph from $G'(n,k,p)$ having diameter $2$. For $n>2k$, we have
\begin{align*}
&\quad P(G'(n,k,p))\\
&\geq 1-\frac{n^2(1-p^2)^{n(1-1/k)-1}}{2k}(1+(k-1)(1-p^2)^{-n/k-1})\left(1+\frac{k}{n}\right)
\end{align*}
and
\begin{align*}
&\quad P(G'(n,k,p))\\
&\leq\frac{2k}{n^2(1-p^2)^{n(1-1/k)+1}}\left(1+(k-1)(1-p^2)^{1-n/k}(1-p)\right)^{-1}\left(1-\frac{2k}{n}\right)^{-1}\\
&\qquad +\frac{4k^3\left(1+\frac{p^3}{(1-p)}\right)^{n(1-1/k)+1}(1-p^2)^{-2}}{n(k-1)}\left(1-\frac{2k}{n}\right)^{-4}.
\end{align*}
\end{theorem}
\begin{corollary}
Let $G'(n,k,p)$ be defined as in Theorem \ref{theorem:6}. If $p=\frac{1}{2}$, we have
\begin{equation*}
P(G'(n,k,p),1/2)\geq 1-\frac{4n^2(3/4)^{n(1-1/k)}}{6k}\left(1+(k-1)(4/3)^{n/k+1}\right)\left(1+\frac{k}{n}\right).
\end{equation*}
\end{corollary}
In the case when $p\rightarrow 0$ as $n\rightarrow\infty$, we can prove the following.
\begin{corollary}\label{prop:6}
Let $G'(n,k,p)$ be as in Theorem \ref{theorem:6}. Let $\lim_{n\rightarrow\infty}p^4n=0$. As $n\rightarrow\infty$, we have
\begin{align*}
P(G'(n,k,p))\geq 1-(1+o(1))\frac{n^2e^{-np^2\left(1-\frac{1}{k}\right)}}{2k}\left(1+(k-1)e^{\frac{np^2}{k}}\right)
\end{align*}
and
\begin{align}
&\quad P(G'(n,k,p))\\
&\leq(1+o(1))\frac{2ke^{np^2\left(1-\frac{1}{k}\right)}}{n^2}\left(1+(k-1)e^{\frac{np^2}{k}}\right)^{-1}\nonumber\\
&\quad\cdot\left(1+\frac{2k^2ne^{\left(np^3-np^2\right)(1-1/k)}}{(k-1)}+2k^2ne^{\left(np^3-np^2\right)(1-1/k)+\frac{np^2}{k}}\right)\nonumber
\end{align}
as $n\rightarrow\infty$. Suppose further that
\begin{equation*}
\lim_{n\rightarrow\infty}\left(2\log n-\log k-np^2\left(1-\frac{1}{k}\right)-\log 2+\log\left(1+(k-1)e^{\frac{np^2}{k}}\right)\right)=c
\end{equation*}
for some $c\in\mathbb{R}\backslash\{0\}$.
\newline
1) If $c<0$, we have
\begin{equation*}
P(G'(n,k,p))\geq 1-(1+o(1))e^c.
\end{equation*}
2) If $c>0$, we have
\begin{equation*}
P(G'(n,k,p))\leq(1+o(1))e^{-c}.
\end{equation*}
\end{corollary}
\begin{remark}
Similar to Remark \ref{rk1}, all $o(1)$ terms in Corollary \ref{prop:6} can be made explicit.
\end{remark}
\begin{remark}
We can similarly derive all of the above results for directed $k$-partite graphs on $n$ vertices where each directed edge is chosen independently with probability $p$. Furthermore, for any two vertices, say $v_1$ and $v_2$, occuring in different partite sets, the existence of the edge from $v_1$ to $v_2$ has probability $p$, while the existence of the edge from $v_2$ to $v_1$ also occurs with probability $p$, and these two edges occur independently. In the appropriate theorems, corollaries, and Corollarys, we multiply the second term of the rebound by $2$, divide the upperbound by $2$, and we add $\log 2$ to our expressions for $c$.. Everything else is left unchanged.
\end{remark}
\section{Bipartite Graphs with Diameter $3$}
Here we analyze bipartite graphs in a similar way to $k$-partite graphs, but instead of considering diameter $2$, we consider diameter $3$ since, except for the complete bipartite graph, all bipartite graphs have diameter at least $3$.
\begin{theorem}\label{theorem11}
For each $n\in\mathbb{N}$, $n\geq 4$, pick $n_1,n_2\in\mathbb{N}$ such that $2\leq n_1\leq n_2$ and $n_1+n_2=n$. Let $G''(n_1,n_2,p)$ denote the set of all bipartite graphs with the partite sets having $n_1$ and $n_2$ vertices respectively where each edge is chosen independently with probability $p$. Also, let $P(G''(n_1,n_2,p))$ be the probability of a graph from $G''(n_1,n_2,p)$ having diameter $3$. Then
\begin{align*}
P(G''(n_1,n_2,p))\geq 1-\frac{n_2^2(1-p^2)^{n_1}}{2}\left(1+\frac{n_1^2(1-p^2)^{n_2-n_1}}{n_2^2}\right)
\end{align*}
and
\begin{align*}
&\quad P(G''(n_1,n_2,p))\\
&\leq\left(\frac{2}{n_2(n_2-1)(1-p^2)^{n_1}}+\frac{\left(1+\frac{p^3}{(1-p)}\right)^{n_1}}{n_2}\left(8+\frac{8}{(1-p)}\right)\right)\\
&\qquad\cdot\left(1+\frac{n_1(n_1-1)(1-p^2)^{n_2-n_1}}{n_2(n_2-1)}\right)^{-1}.
\end{align*}
\end{theorem}
\begin{proof}
As in the proof of Theorem \ref{theorem:3}, we may assume that $p\in\mathbb{Q}\cap(0,1)$ for all $n\in\mathbb{N}$.

Let $p=\frac{r}{s}$ where $r,s\in\mathbb{N}$. As in the proof of Theorem \ref{theorem:3}, we let $A$ be the set of all graphs in $G''(n_1,n_2,p)$, allowing for a number of duplicates of each possible graph to accommodate the edge probability $p$. Since the complete bipartite graph has $n_1n_2$ edges, we have $r^{n_1n_2}$ copies of the complete bipartite graph and $|A|=s^{n_1n_2}$.

We let $B$ be the set of all pairs of vertices such that both vertices of a pair occur in the same partite set. Thus, $|B|=\binom{n_1}{2}\binom{n_2}{2}$. For $a\in A$ and $b\in B$, we write $a\sim b$ if the pair of vertices $b$ in the graph $a$ do not share a common neighbouring vertex. Thus, we will have $\omega(a)=0$ if and only if $a$ is connected with diameter at most $3$. For each pair of vertices $b\in B$ in the set containing $n_1$ vertices, we have
\begin{align*}
D(r,s,n,n_1):=\deg b=((s-r)^2+2r(s-r))^{n_2}((s-r)+r)^{n_1n_2-2n_2}.
\end{align*}
For each pair of vertices $b\in B$ in the set containing $n_2$ vertices, the $n_1$ and $n_2$ are switched in the above equality. It follows that
\begin{equation*}
\sum_{b\in B}\deg b=\frac{s^{n_1n_2}n_1(n_1-1)(1-p^2)^{n_2}}{2}+\frac{s^{n_1n_2}n_2(n_2-1)(1-p^2)^{n_1}}{2}.
\end{equation*}
By the simple sieve, we obtain
\begin{align*}
P(G''(n_1,n_2,p))&> 1-\frac{n_1^2(1-p^2)^{n_2}}{2}-\frac{n_2^2(1-p^2)^{n_1}}{2}\\
&=1-\frac{n_2^2(1-p^2)^{n_1}}{2}\left(1+\frac{n_1^2(1-p^2)^{n_2-n_1}}{n_2^2}\right)
\end{align*}
To get an upper bound for $P(G''(n_1,n_2,p))$, we need to estimate $\sum_{b_1,b_2\in B}n(b_1,b_2)$. Using the same argument as in the proof of Theorem \ref{theorem:3}, we can get
\begin{align*}
&\quad\sum_{b_1,b_2\in B}n(b_1,b_2)\\
&=\binom{n_1}{2}\binom{n_1-2}{2}\frac{D(r,s,n,n_2)^2}{s^{n_1n_2}}+\binom{n_2}{2}\binom{n_2-2}{2}\frac{D(r,s,n,n_1)^2}{s^{n_1n_2}}\\
&\quad +2\binom{n_1}{2}\binom{n_2}{2}\frac{D(r,s,n,n_1)D(r,s,n,n_2)}{s^{n_1n_2}}\cdot\left(1+\frac{4p^3}{(1-p)^2}\right)\\
&\quad +\frac{D(r,s,n,n_2)^2n_2(n_2-1)(n_2-2)}{s^{n_1n_2}}\left(1+\frac{p^3}{(1-p)}\right)^{n_1}\\
&\quad +\frac{s^{n_1n_2}n_2(n_2-1)(1-p^2)^{n_1}}{2}\left(1+\frac{n_1(n_1-1)(1-p^2)^{n_2-n_1}}{n_2(n_2-1)}\right).
\end{align*}
Then, by the Tur\'an sieve, we can get
\begin{align*}
&\quad P(G''(n_1,n_2,p))\cdot\left(1+\frac{n_1(n_1-1)(1-p^2)^{n_2-n_1}}{n_2(n_2-1)}\right)^2.\\
&<\frac{2}{n_2(n_2-1)(1-p^2)^{n_1}}\left(1+\frac{n_1(n_1-1)(1-p^2)^{n_2-n_1}}{n_2(n_2-1)}\right)\\
&\quad+\frac{4n_1^3\left(1+\frac{p^3}{(1-p)}\right)^{n_2}(1-p^2)^{2n_2-2n_1}}{n_2^4}+\frac{4\left(1+\frac{p^3}{(1-p)}\right)^{n_1}}{n_2}\\
&\quad+\frac{8n_1^2\left(\frac{p^3}{(1-p)^2}\right)(1-p^2)^{n_2-n_1}}{n_2^2}.
\end{align*}
Notice that
\begin{equation*}
\frac{p^3}{(1-p)}<\frac{1}{n_1}\left(1+n_1\frac{p^3}{(1-p)}\right)<\frac{1}{n_1}\left(1+\frac{p^3}{(1-p)}\right)^{n_1}
\end{equation*}
and
\begin{align*}
\left(1+\frac{p^3}{(1-p)}\right)(1-p^2)^2
&<1.
\end{align*}
It follows that
\begin{align*}
&\quad P(G''(n_1,n_2,p))\left(1+\frac{n_1(n_1-1)(1-p^2)^{n_2-n_1}}{n_2(n_2-1)}\right)\\
&<\frac{2}{n_2(n_2-1)(1-p^2)^{n_1}}+\frac{4n_1^3\left(1+\frac{p^3}{(1-p)}\right)^{n_2}(1-p^2)^{2n_2-2n_1}}{n_2^4}\\
&\quad +\frac{\left(1+\frac{p^3}{(1-p)}\right)^{n_1}}{n_2}\left(4+\frac{8n_1(1-p^2)^{n_2-n_1}}{n_2(1-p)}\right)\\
&\leq\frac{2}{n_2(n_2-1)(1-p^2)^{n_1}}+\frac{\left(1+\frac{p^3}{(1-p)}\right)^{n_1}}{n_2}\left(8+\frac{8}{(1-p)}\right)
\end{align*}
from which we obtain our upper bound. This completes the proof of Theorem \ref{theorem11}.
\end{proof}
By substituting in $p=\frac{1}{2}$, we deduce from Theorem \ref{theorem11} the following.
\begin{corollary}\label{cor341}
Let $P(G''(n_1,n_2,p))$ be defined as in Theorem \ref{theorem11}. If $p=\frac{1}{2}$, then we have
\begin{equation*}
P(G''(n_1,n_2,p),1/2)\geq 1-\frac{n_2^2(3/4)^{n_1}}{2}\left(1+\frac{n_1^2(3/4)^{n_2-n_1}}{n_2^2}\right)
\end{equation*}
and
\begin{align*}
P(G''(n_1,n_2,p),1/2)\leq\left(\frac{2(4/3)^{n_1}}{n_2(n_2-1)}+\frac{24(5/4)^{n_1}}{n_2}\right)\left(1+\frac{n_1(n_1-1)(3/4)^{n_2-n_1}}{n_2(n_2-1)}\right)^{-1}.
\end{align*}
\end{corollary}
\begin{remark}
The upper bound given for $P(G''(n_1,n_2,p),1/2)$ in Corollary \ref{cor341} will in general only be non-trivial when $n_2$ much larger than $n_1$. For instance, if $n_1<\min\{\frac{2\log n_2-\log 8}{\log(4/3)},\frac{\log n_2-\log 48}{\log(5/4)}\}$, then the upper bound will be less than $1$.
\end{remark}
In the situation where the edge probability $p\rightarrow 0$ as $n\rightarrow\infty$, we will show the following.
\begin{corollary}\label{prop:11}
Let $P(G''(n_1,n_2,p))$ be as in Theorem \ref{theorem11}. Let $\lim_{n\rightarrow\infty}np^4=0$. We have
\begin{align*}
P(G(n,p))\geq 1-\frac{n_2^2e^{-n_1p^2}}{2}\left(1+e^{2\log n_1-2\log n_2-(n_2-n_1)p^2}\right)
\end{align*}
and
\begin{align}
P(G(n,p))&\leq(1+o(1))\left(\frac{2}{n_2^2}e^{n_1p^2}\right)\left(1+e^{2\log n_1-2\log n_2-(n_2-n_1)p^2}\right)^{-1}\\
&\qquad\cdot\left(1+8n_2e^{n_1p^2\left(p-1\right)}\right)\nonumber.
\end{align}
Suppose further that
\begin{equation*}
\lim_{n\rightarrow\infty}\left(2\log n_1-2\log n_2-(n_2-n_1)p^2\right)=-\infty,
\end{equation*}
and
\begin{equation*}
\lim_{n\rightarrow\infty}\left(2\log n_2-n_1p^2-\log 2\right)=c
\end{equation*}
for some $c\in\mathbb{R}$.
\newline
1) If $c<0$, we have
\begin{equation*}
P(G''(n_1,n_2,p))\geq1-(1+o(1))e^c.
\end{equation*}
2) If $c>0$, we have
\begin{equation*}
P(G''(n_1,n_2,p))\leq(1+o(1))e^{-c}.
\end{equation*}
\end{corollary}
\begin{proof}
By the upper bound of $P(G^{(n),\bf{b}},p)$ in Theorem \ref{theorem11}, we can get
\begin{align*}
&\quad P(G''(n_1,n_2,p))\left(1+\frac{n_1(n_1-1)\left(1-p^2\right)^{n_2-n_1}}{n_2(n_2-1)}\right)\\
&<\frac{2}{n_2(n_2-1)\left(1-p^2\right)^{n_1}}+\left(8+\frac{8}{1-p}\right)\frac{\left(1+\frac{p^3}{1-p}\right)^{n_1}}{n_2}.
\end{align*}
Since $\lim_{n\rightarrow\infty}np^4=0$, we have $\lim_{n\rightarrow\infty}n_1p^4=0$ and so
\begin{align*}
&\quad\lim_{n\rightarrow\infty}n_1p^2\left(\frac{p}{1-p}-1\right)-n_1p^2\left(p-1\right)
=0.
\end{align*}
Also, since $p^{-2}\geq 1$, we have
\begin{equation*}
\left(1-p^2\right)^{n_1}>e^{-n_1p^2}\left(1-p^2\right)^{n_1p^2}=e^{-n_1p^2}(1-o(1))
\end{equation*}
and
\begin{align*}
\frac{\left(1+\frac{p^3}{1-p}\right)^{n_1}}{n_2}<\frac{e^{\frac{n_1p^3}{1-p}}}{n_2}=\frac{e^{n_1p^2}}{n_2^2}\cdot n_2e^{n_1p^2\left(\frac{p}{1-p}-1\right)}.
\end{align*}
Also,
\begin{align*}
&\quad\frac{n_1^2\left(1-p^2\right)^{n_2-n_1}}{n_2^2\left(1-p^2\right)}\\
&=e^{2\log n_1-2\log n_2}\left(1-p^2\right)^{p^{-2}\cdot (n_2-n_1)p^2}\left(1-p^2\right)^{-1}\\
&>e^{2\log n_1-2\log n_2-(n_2-n_1)p^2}\left(1-p^2\right)^{(n_2-n_1)p^2}\left(1-p^2\right)^{-1}.
\end{align*}
Since $\lim_{n\rightarrow\infty}np^4=0$, we have
\begin{equation*}
\lim_{n\rightarrow\infty}\left(1-p^2\right)^{(n_2-n_1)p^2}=1.
\end{equation*}
We thus obtain our bounds. Statements (1) and (2) follow as in the proof of Corollary \ref{cor:r2}.
\end{proof}
\begin{remark}
Similar to Remark \ref{rk1}, all $o(1)$ terms in Corollary \ref{prop:11} can be made explicit.
\end{remark}
Substituting in $n_1=n_2=\frac{n}{2}$ or $n_1=\frac{n-1}{2}$ and $n_2=\frac{n+1}{2}$ can lead to similar asymptotics for Tur\'an bipartite graphs.
\begin{theorem}\label{theorem12}
Let $G'''(n,p)$ denote the set of all Tur\'an bipartite graphs where each edge is chosen independently with probability $p$. Also, let $P(G'''(n,p))$ be the probability of a graph from $G'''(n,p)$ having diameter $3$. For $n\geq 4$, we have
\begin{align*}
P(G'''(n,p))\geq 1-\frac{(n+1)^2(1-p^2)^{(n-1)/2}}{8}
\end{align*}
and
\begin{align*}
&\quad P(G'''(n,p))\\
&\leq\left(\frac{8}{n(n-2)(1-p^2)^{n/2}}+\frac{2\left(1+\frac{p^3}{(1-p)}\right)^{n/2}}{n}\left(8+\frac{8}{(1-p)}\right)\right)\\
&\qquad\cdot\left(1+\frac{(n-3)(1-p^2)}{(n+1)}\right)^{-1}.
\end{align*}
\end{theorem}
Substituting $p=\frac{1}{2}$ gives the following.
\begin{corollary}
Let $G'''(n,p)$ be defined as in Corollary \ref{theorem12}. If $p=\frac{1}{2}$, then we have
\begin{equation*}
P(G''(n_1,n_2,p),1/2)\geq 1-\frac{(n+1)^2(3/4)^{(n-1)/2}}{4}.
\end{equation*}
\end{corollary}
In the situation where the edge probability $p\rightarrow 0$ as $n\rightarrow\infty$, we have the following.
\begin{corollary}\label{prop:12}
Let $G'''(n,p)$ be defined as in Corollary \ref{theorem12}. Let $\lim_{n\rightarrow\infty}np^4=0$. We have
\begin{equation*}
P(G'''(n,p))\geq 1-(1+o(1))\frac{n^2e^{-\frac{np^2}{2}}}{4}
\end{equation*}
and
\begin{equation*}
P(G'''(n,p))\leq(1+o(1))\left(\frac{4}{n^2}e^{\frac{np^2}{2}}\right)\left(1+8ne^{\frac{np^2}{2}\left(p^2-1\right)}\right).
\end{equation*}
Suppose further that
\begin{equation*}
\lim_{n\rightarrow\infty}\left(2\log n-\log 4-\frac{np^2}{2}\right)=c
\end{equation*}
for some $c\in\mathbb{R}$.
\newline
1) If $c<0$, we have
\begin{equation*}
P(G'''(n,p))\geq1-(1+o(1))e^c.
\end{equation*}
2) If $c>0$, we have
\begin{equation*}
P(G'''(n,p))\leq(1+o(1))e^{-c}.
\end{equation*}
\end{corollary}
\begin{remark}
Similar to Remark \ref{rk1}, all $o(1)$ terms in Corollary \ref{prop:12} can be made explicit.
\end{remark}
\begin{remark}
Again, we can give analogous results for directed bipartite graphs on $n$ vertices where each directed edge is chosen independently with probability $p$. Furthermore, for any two vertices, say $v_1$ and $v_2$, occuring in different partite sets, the existence of the edge from $v_1$ to $v_2$ has probability $p$, while the existence of the edge from $v_2$ to $v_1$ also occurs with probability $p$, and these two edges occur independently. In the appropriate theorems, corollaries, and Corollarys, we multiply the second term of the lowerbound by $2$, divide the upperbound by $2$, and we add $\log 2$ to our expressions for $c$.  Everything else is left unchanged.
\end{remark}
\section{Initial Results for Graphs with Diameter $d\geq 2$}
We now generalise the above results for any given diameter $d\geq 2$. In this section, we give such a result for a graph from $G(n,p)$ having diameter at most $d$ for some $d\geq 2$ with some restrictions in place for $n$ and $p$. Then in the next section, we refine this result to make it more clear and meaningful by imposing further restrictions on $n$ and $p$. First, a note.
\begin{note}
Throughout this note let
\begin{equation*}
f(n,p,d,i_0):=p\prod_{i=0}^{d-2}\left(\frac{1-(1-p)^{(4np)^ii_0}}{(4np)^ii_0}\right)\prod_{j=1}^{d-2}\left(\frac{1-(1-p^{j+1})^{n^j}}{n^jp^{j+1}}\right),
\end{equation*} 
\begin{equation*}
h(n,p,d,i_0):=\left(1-\frac{4}{5}\left(\frac{e}{3}\right)^{4p\left(n-i_0-4npi_0-\ldots-(4np)^{d-4}i_0\right)}\right)^{2-d}.
\end{equation*}
with the conventions that $h(n,p,2,i_0)=1$ and
\begin{equation*}
h(n,p,3,i_0)=\left(1-\frac{4}{5}\left(\frac{e}{3}\right)^{4np}\right)^{-1},
\end{equation*}
and
\begin{equation*}
g(n,p,d,d',i_0):=\begin{cases}
                      n-1 & d=2\\
                      1+\sum_{l=0}^{d'}\prod_{m=0}^{l}\left(n-1-\sum_{q=0}^m(4np)^qi_0\right) & d\geq 3, d'<d-3\\
                      1+\sum_{l=0}^{d-4}\prod_{m=0}^{l}\left(n-1-\sum_{q=0}^m(4np)^qi_0\right) \\
                      \quad+\left(n-1-\sum_{q=0}^{d-3}(4np)^qi_0\right)\\
                      \qquad\cdot\prod_{m=0}^{d-3}\left(n-1-\sum_{q=0}^m(4np)^qi_0\right)  & d\geq 3, d'\geq d-3.
                       \end{cases}
\end{equation*}
\end{note}
We will prove the following theorem.
\begin{theorem}\label{bigthm}
Fix $d\geq 2$, $d\in\mathbb{N}$. Let $G(n,p)$ denote the set of all simple graphs on $n$ vertices where each edge is chosen independently with probability $p$. Also, let $P(G(n,p),d)$ be the probability of a graph from $G(n,p)$ having diameter at most $d$. Suppose that
\begin{equation*}
2+8np+2(4np)^2+\ldots+2(4np)^{d'}\leq n-2
\end{equation*}
where $d'\geq 0$. We have
\begin{equation*}
P(G(n,p),d)>1-\binom{n}{2}h(n,p,d,1)\left(1-f(n,p,d,1)\right)^{g(n,p,d,d',1)}
\end{equation*}
and
\begin{align*}
P(G(n,p),d)&<{(1-p^d)^{-2\left(n^{d-1}+\frac{n^{d-2}}{p}+\frac{n^{d-3}}{p^2}+\ldots+\frac{1}{p^{d-1}}\right)}}h(n-1,p,d,2)\\
&\quad\cdot\left(1-f(n,p,d,2)\right)^{2\cdot g(n-1,p,d,d',2)}\\
&\quad-1+\frac{1}{\binom{n}{2}(1-p^d)^{\left(n^{d-1}+\frac{n^{d-2}}{p}+\frac{n^{d-3}}{p^2}+\ldots+\frac{1}{p^{d-1}}\right)}}.
\end{align*}
\end{theorem}
To prove Theorem \ref{bigthm}, we first need the following two lemmas.
\begin{lemma}\label{combineq}
For all $n,m\in\mathbb{N}$ with $m<n$, we have
\begin{equation*}
\binom{n}{m}<\left(\frac{n}{m}\right)^m\cdot\frac{e^{1+1/(12n)}}{\sqrt{2\pi m\left(1-\frac{m}{n}\right)}}.
\end{equation*}
\end{lemma}
\begin{proof}
Robbins shows in \cite{robbins} that for all $m\in\mathbb{N}$, we have
\begin{equation*}
\sqrt{2\pi}m^{m+1/2}e^{-m}\cdot e^{1/(12m+1)}.<m!<\sqrt{2\pi}m^{m+1/2}e^{-m}\cdot e^{1/(12m)}.
\end{equation*}
Thus we have
\begin{align*}
\binom{n}{m}&\leq\frac{n^n}{m^m(n-m)^{n-m}}\cdot\frac{\sqrt{2\pi n}}{\sqrt{2\pi m}\cdot\sqrt{2\pi(n-m)}}e^{\frac{1}{12n}}\\
&<\left(\frac{n}{n-m}\right)^n\left(\frac{n-m}{m}\right)^m\cdot\frac{e^{1/(12n)}}{\sqrt{2\pi m\left(1-\frac{m}{n}\right)}}\\
&=\left(\frac{n}{n-m}\right)^{n-m}\left(\frac{n}{m}\right)^m\cdot\frac{e^{1/(12n)}}{\sqrt{2\pi m\left(1-\frac{m}{n}\right)}}\\
&<\left(\frac{n}{m}\right)^m\cdot\frac{e^{1+1/(12n)}}{\sqrt{2\pi m\left(1-\frac{m}{n}\right)}}.
\end{align*}
\end{proof}
\begin{lemma}\label{4nplemma}
Suppose $f:\mathbb{N}\times\mathbb{N}$ satisfies $f(n,i+1)\leq f(n,i)$ for all $i,n\in\mathbb{N}$. Let $r\in\mathbb{R}$, $r>0$ satisfy $\frac{4r}{r+1}<1$. Then for all $n\in\mathbb{N}$ and for all $\frac{4nr}{r+1}\leq t\leq n$ we have
\begin{equation*}
\sum_{i=0}^n\binom{n}{i}r^if(n,i)<\left(1-\frac{4}{5}\left(\frac{e}{3}\right)^t\right)^{-1}\sum_{i=0}^{\lfloor t\rfloor}\binom{n}{i}r^if(n,i).
\end{equation*}
\end{lemma}
\begin{proof}
First assume that $n\geq 2$ and $t<n-1$. For all $i\geq\frac{4nr}{r+1}$, we have the following:
\begin{align*}
\frac{\binom{n}{i+1}r^{i+1}}{\binom{n}{i}r^i}<\frac{(n-i)r}{i+1}<\left(\frac{n}{i}-1\right)r\leq\left(\frac{1}{4}+\frac{1}{4r}-1\right)r<\frac{1}{4}.
\end{align*}
Thus we have
\begin{align*}
\frac{\sum_{i=\lfloor t\rfloor+1}^n\binom{n}{i}r^if(n,i)}{\sum_{i=0}^n\binom{n}{i}r^if(n,i)}&=\left(1+\frac{\sum_{i=0}^{i=\lfloor t\rfloor}\binom{n}{i}r^if(n,i)}{\sum_{i=\lfloor t\rfloor+1}^n\binom{n}{i}r^if(n,i)}\right)^{-1}\\
&\leq\left(1+\frac{\sum_{i=0}^{i=\lfloor t\rfloor}\binom{n}{i}r^if\left(n,\lfloor t\rfloor\right)}{\sum_{i=\lfloor t\rfloor+1}^n\binom{n}{i}r^if(n,\lfloor t\rfloor+1)}\right)^{-1}\\
&\leq\frac{\sum_{i=\lfloor t\rfloor+1}^n\binom{n}{i}r^i}{\sum_{i=0}^n\binom{n}{i}r^i}\\
&<\frac{\binom{n}{\lfloor t\rfloor+1}r^{\lfloor t\rfloor+1}}{(1+r)^{n}}\sum_{i=0}^{\infty}\frac{1}{4^i}\\
&\leq\frac{4\binom{n}{\lfloor t\rfloor+1}r^{\lfloor t\rfloor+1}}{3(1+r)^{n}}.
\end{align*}
By Lemma \ref{combineq}, we have
\begin{align*}
\binom{n}{\lfloor t\rfloor+1}r^{\lfloor t\rfloor+1}&<\left(\frac{en}{\lfloor t\rfloor+1}\right)^{\lfloor t\rfloor+1}\cdot\frac{e^{1/(12n)}r^{\lfloor t\rfloor+1}}{\sqrt{2\pi\left(\lfloor t\rfloor+1\right)\left(1-\frac{\left(\lfloor t\rfloor+1\right)}{n}\right)}}\\
&<\left(\frac{enr}{\lfloor t\rfloor+1}\right)^{\lfloor t\rfloor+1}\cdot\frac{e^{1/(12n)}}{\sqrt{2\pi(1-1/n)}}\\
&<\left(\frac{e(r+1)}{4}\right)^t\cdot\frac{e^{1/(12n)}}{\sqrt{2\pi(1-1/n)}}
\end{align*}
with the second inequality following from $t<n-1$ and the third inequality following from $\frac{e(r+1)}{4}<\frac{e}{3}<1$. Thus we have
\begin{align*}
\frac{\sum_{i=\lfloor t\rfloor+1}^n\binom{n}{i}r^if(n,i)}{\sum_{i=0}^n\binom{n}{i}r^if(n,i)}&\leq\frac{4e^{1/(12n)}\left(\frac{e(r+1)}{4}\right)^t}{3(1+r)^{n}\sqrt{2\pi(1-1/n)}}\\
&<\frac{4e^{1/(12n)}\left(\frac{e}{3}\right)^t}{3\sqrt{\pi}}\\
&<\frac{4}{5}\left(\frac{e}{3}\right)^t.
\end{align*}
Thus
\begin{equation*}
\sum_{i=0}^n\binom{n}{i}r^if(n,i)<\left(1-\frac{4}{5}\left(\frac{e}{3}\right)^t\right)^{-1}\sum_{i=0}^{\lfloor t\rfloor}\binom{n}{i}r^if(n,i).
\end{equation*}
Now assume that $n\geq 1$ and $n-1\leq t<n$. Then we have
\begin{equation*}
\frac{\sum_{i=\lfloor t\rfloor+1}^n\binom{n}{i}r^if(n,i)}{\sum_{i=0}^n\binom{n}{i}r^if(n,i)}\leq\frac{r^n}{(1+r)^n}<\left(\frac{1}{4}\right)^n=\left(\frac{3}{4e}\right)^n\left(\frac{e}{3}\right)^n<\frac{4}{5}\left(\frac{e}{3}\right)^t.
\end{equation*}
Finally, if $t=n$ or $n=0$, then the desired result holds trivially.
\end{proof}
We will now prove Theorem $1$.
\newline
\newline
For each $n\in\mathbb{N}$, let $G(n,p)$ denote the set of all graphs on $n$ vertices with edge probability $p$, and let $P'(G(n,p))$ be the probability of a graph from $G(n,p)$ having diameter at most $d$. Let $p=\frac{r}{s}$ where $r=r(n), s=s(n)\in\mathbb{N}$. We let $A$ be the set of all graphs in $G(n,p)$, allowing for a number of duplicates of each possible graph to accommodate the edge probability $p$, so that
\begin{equation*}
|A|=\sum_{k=0}^{\binom{n}{2}}\binom{\binom{n}{2}}{ k}r^k(s-r)^{\binom{n}{2}-k}=s^{\binom{n}{2}}.
\end{equation*}
We let $B$ be all pairs of vertices so $|B|=\binom{n}{2}$. For a graph $a\in A$ and a pair of vertices $b\in B$, we say $a\sim b$ if there is no path between the pair of vertices $b$ that consists of at most $d$ edges. Thus, we will have $\omega(a)=0$ if and only if $a$ is connected with diameter at most $d$.
\newline
\newline
Pick a pair of vertices $b\in B$ and call them $v_1$ and $v_2$. To calculate $\deg b$, we need to calculate the number of graphs in $A$ such that there is no path from $v_1$ to $v_2$ that consists of at most $d$ edges. To help with this calculation, we will calculate a generalised notion of $\deg b$ as follows. Let $0\leq i_0\leq n-1$. Pick a specific set of $i_0$ vertices out of the $n$ labeled vertices, as well as another vertex, say $v$, out of the $n$ labeled vertices. We will let $C(n,r,s,d,i_0)$ denote the number of graphs in $A$ such that there is no path from any of the $i_0$ vertices to vertex $v$ that consists of at most $d$ edges. We can derive the recursive formula
\begin{align}
C(n,r,s,d+1,i_0)&=(s-r)^{i_0(n-i_0)}s^{\binom{n}{2}-i_0(n-i_0)}\nonumber\\
&\quad+\sum_{i_1=1}^{n-1-i_0}\binom{n-1-i_0}{i_1}\left(s^{i_0}-(s-r)^{i_0}\right)^{i_1}(s-r)^{i_0(n-i_0-i_1)}s^{\binom{i_0}{2}}\nonumber\\
&\qquad\qquad\qquad\cdot C(n-i_0,r,s,d,i_1)\label{eqn2a}
\end{align} 
valid for all $0\leq i_0\leq n-1$ and $d\geq 1$, which can be simplified to
\begin{align}\label{eqn2}
C(n,r,s,d+1,i_0)&=\sum_{i_1=0}^{n-1-i_0}\binom{n-1-i_0}{i_1}\left(s^{i_0}-(s-r)^{i_0}\right)^{i_1}(s-r)^{i_0(n-i_0-i_1)}s^{\binom{i_0}{2}}\nonumber\\
&\qquad\qquad\qquad\cdot C(n-i_0,r,s,d,i_1)
\end{align}
if we assume that $i_0>0$. As well,
\begin{equation*}
C(n,r,s,1,i_0)=(s-r)^{i_0}s^{\binom{n}{2}-i_0}
\end{equation*}
for all $1\leq i_0\leq n-1$, completing the formula. Then we can deduce that $C(n,r,s,d,1)=\deg b$ if we are working with diameter $d$. Let $D(n,p,d,i_0)=\frac{C(n,r,s,1,i_0)}{s^{\binom{n}{2}}}$ so that $D(n,p,d,i_0)$ is the probability that the edge distance between $v$ and any of the $i_0$ vertices is greater than $d$. We will prove that for all $0\leq i_0\leq n-1$, $0<p<1$, $d\geq 1$ that
\begin{equation}\label{lowerdegb}
D(n,p,d,i_0)\geq(1-p^d)^{i_0\left(n^{d-1}+\frac{n^{d-2}}{p}+\frac{n^{d-3}}{p^2}+\ldots+\frac{1}{p^{d-1}}\right)}.
\end{equation}
If we also have the additional constraint $1\leq i_0\leq\frac{n-1}{1+4np+(4np)^2+\ldots+(4np)^{d'}}$ where $d'\geq 0$, then we also have
\begin{equation}
D(n,p,d,i_0)<h(n,p,d,i_0)\left(1-f(n,p,d,i_0)\right)^{i_0g(n,p,d,d,i_0)}.\label{upperdegb}
\end{equation}
We prove by induction on $d$. First, we need a few lemmas.
\begin{lemma}\label{lem3}
Fix $0<r<1$ and let $g(x):=\frac{1-(1-r)^x}{rx}$. Then $g(x)$ is a decreasing function on $\mathbb{R}$. Also, Fix $y\geq 1$ and let $h(x)=\frac{1-(1-x)^y}{yx}$. Then $h(x)$ is a non-increasing function on $[0,1]$.
\end{lemma}
\begin{proof}
Fix $0<r<1$ and let $g(x):=\frac{1-(1-r)^x}{rx}$. Let $g'(x)$ be the derivative of $g$ with respect to $x$. Then
\begin{align*}
g'(x)&=\frac{xr\left((1-r)^x\log(1-r)\right)-\left(1-(1-r)^x\right)r}{x^2r^2}\\
&=\frac{x\left((1-r)^x\log(1-r)\right)-\left(1-(1-r)^x\right)}{rx^2}\\
&<\frac{(1-r)^x-1}{rx^2}\\
&<0.
\end{align*}
Therefore $g(x)$ is decreasing. Let $h(x):=\frac{1-(1-x)^y}{yx}$. Let $h'(x)$ be the derivative of $h$ with respect to $x$. Then
\begin{align*}
h'(x)&=\frac{yx(1-x)^{y-1}-y\left(1-(1-x)^y\right)}{y^2x^2}\\
&=\frac{(1-x)^{y-1}-1}{yx^2}\\
&\leq 0.
\end{align*}
Thus $h(x)$ is a non-increasing function on $[0,1]$.
\end{proof}
\begin{lemma}\label{lem2}
Let $0\leq q,r\leq 1$ and $y\geq 1$. We have
\begin{equation*}
(1-qr)^y\leq 1-q+q(1-r)^y.
\end{equation*}
Also, if $C<\frac{1-(1-r)^M}{Mr}$ where $y<M$, then we also have
\begin{equation*}
1-q+q(1-r)^y\leq(1-Cqr)^y.
\end{equation*}
\end{lemma}
\begin{proof}
We observe that the lemma holds for $r=0$. Fix $0\leq q\leq 1$ and $y\geq 1$ and let $f(r)=1-q+q(1-r)^y-(1-qr)^y$. Let $f'(r)$ be the derivative of $f$ with respect to $r$. Then
\begin{equation*}
f'(r)=-qy(1-r)^{y-1}+qy(1-qr)^{y-1}=qy((1-qr)^{y-1}-(1-r)^{y-1})\geq 0.
\end{equation*}
It follows that $f(r)\geq 0$ for all $0\leq r\leq 1$ and so the first result follows. Let $C<\frac{1-(1-r)^M}{Mr}$ where $y\leq M$. From Lemma \ref{lem3}, we therefore have
\begin{align*}
C&<\frac{1-(1-r)^y}{yr}\\
&=\frac{1-\left(1-\frac{q}{y}+\frac{q}{y}(1-r)^y\right)}{qr}\\
&\leq\frac{1-\left(1-q+q(1-r)^y\right)^{1/y}}{qr}
\end{align*}
Thus
\begin{equation*}
1-Cqr>\left(1-q+q(1-r)^y\right)^{1/y}
\end{equation*}
from which the result follows.
\end{proof}
\begin{lemma}\label{lem1}
Let $i,j\in\mathbb{N}\cup\{0\}$, $0<p<1$, and $t\geq 1$. Then we have
\begin{equation*}
(1-p^{j+1})^{it}\leq(1-p)^i+\left(1-(1-p)^i\right)(1-p^j)^t.
\end{equation*}
Also, let $0<C_1<1$, $C_2<\frac{1-(1-p^j)^M}{Mp^j}$ and $C_3<\frac{1-(1-p)^N}{Np}$ where $M\geq t$ and $N\geq i$. Then we also have
\begin{equation*}
(1-p)^i+\left(1-(1-p)^i\right)(1-C_1p^j)^t\leq(1-C_1C_2C_3p^{j+1})^{it}.
\end{equation*}
\end{lemma}
\begin{proof}
Let $i,j\in\mathbb{N}\cup\{0\}$, $0<p<1$, and $t\geq 1$. Applying Lemma \ref{lem2} twice we have
\begin{align*}
(1-p)^i+\left(1-(1-p)^i\right)(1-p^j)^t&=\left(1-\left(1-(1-p^i)\right)\right)+\left(1-(1-p)^i\right)(1-p^j)^t\\
&\geq\left(1-p^j\left(1-(1-p)^i\right)\right)^t\\
&=\left(1-p^j+p^j(1-p)^i\right)^t\\
&\geq(1-p^{j+1})^{it}.
\end{align*}
By Lemma \ref{lem3} we have
\begin{equation*}
C_2<\frac{1-(1-p^j)^M}{Mp^j}<\frac{1-(1-C_1p^j)^M}{MC_1p^j}.
\end{equation*}
Thus, applying Lemma \ref{lem2} twice, we have
\begin{align*}
(1-p)^i+\left(1-(1-p)^i\right)(1-C_1p^j)^t&=\left(1-\left(1-(1-p^i)\right)\right)+\left(1-(1-p)^i\right)(1-C_1p^j)^t\\
&\leq\left(1-C_2C_1p^j\left(1-(1-p)^i\right)\right)^t\\
&=\left(1-C_2C_1p^j+C_2C_1p^j(1-p)^i\right)^t\\
&\leq(1-C_3C_2C_1p^{j+1})^{it}.
\end{align*}
\end{proof}
For $d=1$, we have $D(n,p,1,i_0)=(1-p)^{i_0}$. Suppose for some $d\geq 1$ \eqref{lowerdegb} holds for all $0\leq i_0\leq n-1$, and $0<p<1$. We will prove it holds for $d+1$. First, we can verify that \eqref{lowerdegb} holds if $i_0=0$ (in which case both sides of \eqref{lowerdegb} are just equal to $1$), so assume that $i_0>0$. From \eqref{eqn2} we have
\begin{align*}
D(n,p,d+1,i_0)&=(1-p)^{i_0(n-i_0)}\sum_{i_1=0}^{n-1-i_0}\binom{n-1-i_0}{i_1}((1-p)^{-i_0}-1)^{i_1}D(n-i_0,p,d,i_1)\\
&>(1-p)^{i_0(n-i_0)}\sum_{i_1=0}^{n-1-i_0}\binom{n-1-i_0}{i_1}((1-p)^{-i_0}-1)^{i_1}(1-p^d)^{i_1\left(n^{d-1}+\frac{n^{d-2}}{p}+\frac{n^{d-3}}{p^2}+\ldots+\frac{1}{p^{d-1}}\right)}\\
&=(1-p)^{i_0(n-i_0)}\left(1+\left((1-p)^{-i_0}-1\right)(1-p^d)^{n^{d-1}+\frac{n^{d-2}}{p}+\frac{n^{d-3}}{p^2}+\ldots+\frac{1}{p^{d-1}}}\right)^{n-1-i_0}\\
&=(1-p)^{i_0}\left((1-p)^{i_0}+\left(1-(1-p)^{i_0}\right)(1-p^d)^{n^{d-1}+\frac{n^{d-2}}{p}+\frac{n^{d-3}}{p^2}+\ldots+\frac{1}{p^{d-1}}}\right)^{n-1-i_0}\\
\end{align*}
Using Lemma \ref{lem1} we thus have
\begin{align*}
D(n,p,d+1,i_0)&>(1-p)^{i_0}(1-p^{d+1})^{i_0(n-1-i_0)\left(n^{d-1}+\frac{n^{d-2}}{p}+\frac{n^{d-3}}{p^2}+\ldots+\frac{1}{p^{d-1}}\right)}\\
&=\frac{(1-p^{d+1})^{i_0}}{(1+p+p^2+\ldots+p^d)^{i_0}}\cdot(1-p^{d+1})^{i_0(n-1-i_0)\left(n^{d-1}+\frac{n^{d-2}}{p}+\frac{n^{d-3}}{p^2}+\ldots+\frac{1}{p^{d-1}}\right)}\\
&>(1-p^{d+1})^{i_0}e^{(-p-p^2-\ldots-p^d)i_0}(1-p^{d+1})^{i_0(n-1-i_0)\left(n^{d-1}+\frac{n^{d-2}}{p}+\frac{n^{d-3}}{p^2}+\ldots+\frac{1}{p^{d-1}}\right)}\\
&>(1-p^{d+1})^{(1+p^{-1}+p^{-2}+p^{-3}+\ldots+p^{-d})i_0}(1-p^{d+1})^{i_0(n-1)\left(n^{d-1}+\frac{n^{d-2}}{p}+\frac{n^{d-3}}{p^2}+\ldots+\frac{1}{p^{d-1}}\right)}\\
&>(1-p^{d+1})^{i_0\left(n^d+\frac{n^{d-1}}{p}+\frac{n^{d-2}}{p^2}+\ldots+\frac{1}{p^d}\right)}.
\end{align*}
Thus \eqref{lowerdegb} is proved. Next we prove \eqref{upperdegb} again by induction on $d$. First, applying Lemma \ref{lem1} we have
\begin{align*}
D(n,p,2,i_0)&=(1-p)^{i_0}(1-p+p(1-p)^{i_0})^{n-1-i_0}\\
&<(1-p+p(1-p)^{i_0})^{n-1}\\
&<\left(1-\frac{p\left(1-(1-p)^{i_0}\right)}{i_0}\right)^{i_0(n-1)}.
\end{align*}
Suppose for some $d\geq 2$ with any $d'\geq 0$ \eqref{upperdegb} holds for all $1+4np+(4np)^2+\ldots+(4np)^{d'}\leq n-2$, and $0<p<1$. We will prove it holds for $d+1$. We have
\begin{equation*}
D(n,p,d+1,i_0)=(1-p)^{i_0(n-i_0)}\sum_{i_1=0}^{n-1-i_0}\binom{n-1-i_0}{i_1}((1-p)^{-i_0}-1)^{i_1}D(n-i_0,p,d,i_1).
\end{equation*}
We divide into three cases.
\begin{case}{$\frac{n-1}{1+4np}<i_0\leq n-1$.}
\normalfont
\newline
\newline
We have the following:
\begin{align*}
D(n,p,d+1,i_0)&=(1-p)^{i_0(n-i_0)}\sum_{i_1=0}^{n-1-i_0}\binom{n-1-i_0}{i_1}((1-p)^{-i_0}-1)^{i_1}D(n-i_0,p,d,i_1)\\
&\leq(1-p)^{i_0(n-i_0)}\left(1+\sum_{i_1=1}^{n-1-i_0}\binom{n-1-i_0}{i_1}((1-p)^{-i_0}-1)^{i_1}\right.\\
&\qquad\qquad\qquad\qquad\qquad\qquad\qquad\left.\left(1-f(n-i_0,p,d,i_1)\right)^{i_1g(n-i_0,p,d,0,i_1)}h(n-i_0,p,d,i_1)\right).
\end{align*}
We can deduce that $h(n-i_0,p,d,i_1)\leq h(n-i_0,p,d,4npi_0)\leq h(n,p,d+1,i_0)$ and from Lemma \ref{lem3}, we can deduce that $f(n,p,d,4npi_0)<f(n-i_0,p,d,i_1)$. As well, $g(n-i_0,p,d,n-1-i_0)\leq g(n-i_0,p,d,i_1)$. Thus we have
\begin{align*}
D(n,p,d+1,i_0)&<h(n,p,d+1,i_0)(1-p)^{i_0(n-i_0)}\left(1+\sum_{i_1=1}^{n-1-i_0}\binom{n-1-i_0}{i_1}((1-p)^{-i_0}-1)^{i_1}\right.\\
&\qquad\qquad\qquad\qquad\qquad\qquad\qquad\qquad\qquad\qquad\left.\left(1-f(n,p,d,4npi_0)\right)^{i_1g(n-i_0,p,d,0,n-1-i_0)}\right)\\
&=h(n,p,d+1,i_0)(1-p)^{i_0(n-i_0)}\\
&\quad\cdot\left(1+((1-p)^{-i_0}-1)\left(1-f(n,p,d,4npi_0)\right)^{g(n-i_0,p,d,0,n-1-i_0)}\right)^{n-1-i_0}\\
&<(1-p)^{i_0}h(n,p,d+1,i_0)\\
&\quad\cdot\left((1-p)^{i_0}+(1-(1-p)^{i_0})\left(1-f(n,p,d,4npi_0)\right)^{g(n-i_0,p,d,0,n-1-i_0)}\right)^{n-1-i_0}.
\end{align*}
We note that $g(n-i_0,p,d,0,n-1-i_0)<n^{d-1}$ and so, using Lemma \ref{lem1}, we thus have
\begin{align*}
&\quad D(n,p,d+1,i_0)\\
&<h(n,p,d+1,i_0)\\
&\quad\cdot\left(1-p\left(\frac{1-(1-p)^{i_0}}{pi_0}\right)\left(\frac{1-(1-p^d)^{n^{d-1}}}{n^{d-1}p^d}\right)f(n,p,d,4npi_0)\right)^{i_0(n-i_0-1)g(n-i_0,p,d,0,n-1-i_0)+i_0}\\
&=h(n,p,d+1,i_0)\left(1-f(n,p,d+1,i_0)\right)^{i_0(n-i_0-1)g(n-i_0,p,d,0,n-1-i_0)+i_0}.
\end{align*}
We can deduce that $(n-i_0-1)g(n-i_0,p,d,0,n-1-i_0)+1\geq g(n,p,d+1,0,i_0)$ and so we have \eqref{upperdegb}. 
\end{case}
\begin{case}{$i_0\leq\frac{n-1}{1+4np}$}
\normalfont
\newline
\newline
Given a set of $i_1$ vertices and one additional vertex, say $v$, in a graph from $G(n-i_0,p)$, we know that $D(n-i_0,p,d,i_1)$ is the probability that the edge distance between $v$ and any of the $i_1$ vertices is greater than $d$. By adding one more vertex to our set of $i_1$ vertices, it therefore follows that $D(n-i_0,p,d,i_1+1)\leq D(n-i_0,p,d,i_1)$. Thus, by Lemma \ref{4nplemma}, we have
\begin{align*}
D(n,p,d+1,i_0)&<\left(1-\frac{4}{5}\left(\frac{e}{3}\right)^{4npi_0}\right)^{-1}(1-p)^{i_0(n-i_0)}\sum_{i_1=0}^{4npi_0}\binom{n-1-i_0}{i_1}((1-p)^{-i_0}-1)^{i_1}D(n-i_0,p,d,i_1)\\
&<\left(1-\frac{4}{5}\left(\frac{e}{3}\right)^{4npi_0}\right)^{-1}(1-p)^{i_0(n-i_0)}\\
&\quad\cdot\left(1+\sum_{i_1=1}^{4npi_0}\binom{n-1-i_0}{i_1}((1-p)^{-i_0}-1)^{i_1}\right.\\
&\qquad\qquad\qquad\qquad\qquad\cdot\left.\left(1-f(n-i_0,p,d,i_1)\right)^{i_1g(n-i_0,p,d,0,i_1)}h(n-i_0,p,d,i_1)\right).
\end{align*}
We can deduce that $h(n-i_0,p,d,i_1)\leq h(n-i_0,p,d,4npi_0)$ and from Lemma \ref{lem3}, we can deduce that $f(n,p,d,4npi_0)<f(n-i_0,p,d,i_1)$. As well, $g(n-i_0,p,d,0,n-1-i_0)\leq g(n-i_0,p,d,0,i_1)$. Thus we have
\begin{align*}
D(n,p,d+1,i_0)&<\left(1-\frac{4}{5}\left(\frac{e}{3}\right)^{4npi_0}\right)^{-1}(1-p)^{i_0(n-i_0)}\\
&\quad\cdot\left(1+\sum_{i_1=1}^{4npi_0}\binom{n-1-i_0}{i_1}((1-p)^{-i_0}-1)^{i_1}\right.\\
&\left.\qquad\qquad\qquad\qquad\qquad\cdot\left(1-f(n,p,d,4npi_0)\right)^{i_1g(n-i_0,p,d,0,n-1-i_0)}h(n-i_0,p,d,4npi_0)\right)\\
&<\left(1-\frac{4}{5}\left(\frac{e}{3}\right)^{4npi_0}\right)^{-1}h(n-i_0,p,d,4npi_0)(1-p)^{i_0(n-i_0)}\\
&\quad\cdot\left(1+\sum_{i_1=1}^{n-1-i_0}\binom{n-1-i_0}{i_1}((1-p)^{-i_0}-1)^{i_1}\left(1-f(n,p,d,4npi_0)\right)^{i_1g(n-i_0,p,d,0,n-1-i_0)}\right)\\
&=h(n,p,d+1,i_0)(1-p)^{i_0(n-i_0)}\\
&\quad\cdot\left(1+\left((1-p)^{-i_0}-1\right)\left(1-f(n,p,d,4npi_0)\right)^{g(n-i_0,p,d,0,n-1-i_0)}\right)^{n-1-i_0}\\
&<h(n,p,d+1,i_0)(1-p)^{i_0}\\
&\quad\cdot\left((1-p)^{i_0}+\left(1-(1-p)^{i_0}\right)\left(1-f(n,p,d,4npi_0)\right)^{g(n-i_0,p,d,0,n-1-i_0)}\right)^{n-1-i_0}.
\end{align*}
We note that $g(n-i_0,p,d,0,n-1-i_0)<n^{d-1}$ and so, using Lemma \ref{lem1}, we thus have
\begin{align*}
&\quad D(n,p,d+1,i_0)\\
&<h(n,p,d+1,i_0)\\
&\quad\cdot\left(1-p\left(\frac{1-(1-p)^{i_0}}{pi_0}\right)\left(\frac{1-(1-p^d)^{n^{d-1}}}{n^{d-1}p^d}\right)f(n,p,d,4npi_0)\right)^{i_0(n-i_0-1)g(n-i_0,p,d,0,n-1-i_0)+i_0}\\
&=h(n,p,d+1,i_0)\left(1-f(n,p,d+1,i_0)\right)^{i_0(n-i_0-1)g(n-i_0,p,d,0,n-1-i_0)+i_0}.
\end{align*}
We can deduce that $(n-i_0-1)g(n-i_0,p,d,0,n-1-i_0)+1\geq g(n,p,d+1,0,i_0)$ and so we have \eqref{upperdegb}.
\end{case}
\begin{case}{$i_0\leq\frac{n-1}{1+4np+\ldots+(4np)^{d'+1}}$, $d'\geq 0$}
\normalfont
\newline
\newline
Given a set of $i_1$ vertices and one additional vertex, say $v$, in a graph from $G(n-i_0,p)$, we know that $D(n-i_0,p,d,i_1)$ is the probability that the edge distance between $v$ and any of the $i_1$ vertices is greater than $d$. By adding one more vertex to our set of $i_1$ vertices, it therefore follows that $D(n-i_0,p,d,i_1+1)\leq D(n-i_0,p,d,i_1)$. Thus, by Lemma \ref{4nplemma}, we have
\begin{align*}
D(n,p,d+1,i_0)&<\left(1-\frac{4}{5}\left(\frac{e}{3}\right)^{4npi_0}\right)^{-1}(1-p)^{i_0(n-i_0)}\sum_{i_1=0}^{4npi_0}\binom{n-1-i_0}{i_1}((1-p)^{-i_0}-1)^{i_1}D(n-i_0,p,d,i_1)\\
&<\left(1-\frac{4}{5}\left(\frac{e}{3}\right)^{4npi_0}\right)^{-1}(1-p)^{i_0(n-i_0)}\\
&\quad\cdot\left(1+\sum_{i_1=1}^{4npi_0}\binom{n-1-i_0}{i_1}((1-p)^{-i_0}-1)^{i_1}\right.\\
&\qquad\qquad\qquad\qquad\qquad\cdot\left.\left(1-f(n-i_0,p,d,i_1)\right)^{i_1g(n-i_0,p,d,d',i_1)}h(n-i_0,p,d,i_1)\right).
\end{align*}
We can deduce that $h(n-i_0,p,d,i_1)\leq h(n-i_0,p,d,4npi_0)$ and from Lemma \ref{lem3}, we can deduce that $f(n,p,d,4npi_0)<f(n-i_0,p,d,i_1)$. As well, $g(n-i_0,p,d,4npi_0)\leq g(n-i_0,p,d,i_1)$. Thus we have
\begin{align*}
&\quad D(n,p,d+1,i_0)\\
&<\left(1-\frac{4}{5}\left(\frac{e}{3}\right)^{4npi_0}\right)^{-1}(1-p)^{i_0(n-i_0)}\\
&\quad\cdot\left(1+\sum_{i_1=1}^{4npi_0}\binom{n-1-i_0}{i_1}((1-p)^{-i_0}-1)^{i_1}\right.\\
&\left.\qquad\qquad\qquad\qquad\qquad\cdot\left(1-f(n,p,d,4npi_0)\right)^{i_1g(n-i_0,p,d,d',4npi_0)}h(n-i_0,p,d,4npi_0)\right)\\
&<\left(1-\frac{4}{5}\left(\frac{e}{3}\right)^{4npi_0}\right)^{-1}h(n-i_0,p,d,4npi_0)(1-p)^{i_0(n-i_0)}\\
&\quad\cdot\left(1+\sum_{i_1=1}^{n-1-i_0}\binom{n-1-i_0}{i_1}((1-p)^{-i_0}-1)^{i_1}\left(1-f(n,p,d,4npi_0)\right)^{i_1g(n-i_0,p,d,d',4npi_0)}\right)\\
&=h(n,p,d+1,i_0)(1-p)^{i_0(n-i_0)}\left(1+\left((1-p)^{-i_0}-1\right)\left(1-f(n,p,d,4npi_0)\right)^{g(n-i_0,p,d,d',4npi_0)}\right)^{n-1-i_0}\\
&<h(n,p,d+1,i_0)(1-p)^{i_0}\left((1-p)^{i_0}+\left(1-(1-p)^{i_0}\right)\left(1-f(n,p,d,4npi_0)\right)^{g(n-i_0,p,d,d',4npi_0)}\right)^{n-i_0-1}.
\end{align*}
We note that $g(n-i_0,p,d,d',4npi_0)<n^{d-1}$ and so, using Lemma \ref{lem1}, we thus have
\begin{align*}
D(n,p,d+1,i_0)&<h(n,p,d+1,i_0)\\
&\quad\cdot\left(1-p\left(\frac{1-(1-p)^{i_0}}{pi_0}\right)\left(\frac{1-(1-p^d)^{n^{d-1}}}{n^{d-1}p^d}\right)f(n,p,d,4npi_0)\right)^{i_0(n-i_0)g(n-i_0,p,d,d',4npi_0)+i_0}\\
&=h(n,p,d+1,i_0)\left(1-f(n,p,d+1,i_0)\right)^{i_0(n-i_0)g(n-i_0,p,d,d',4npi_0)+i_0}.
\end{align*}
We can deduce that $(n-i_0-1)g(n-i_0,p,d,d',4npi_0)+1>g(n,p,d+1,d'+1,i_0)$ and so we have \eqref{upperdegb}.
\end{case}
By \eqref{upperdegb}, we have
\begin{equation*}
\sum_{b\in B}\deg b<s^{\binom{n}{2}}\binom{n}{2}h(n,p,d,1)\left(1-f(n,p,d,1)\right)^{g(n,p,d,d',1)}.
\end{equation*}
Hence, by the simple sieve, we have
\begin{align*}
P(G(n,p),d)&>1-\binom{n}{2}h(n,p,d,1)\left(1-f(n,p,d,1)\right)^{g(n,p,d,d',1)}.
\end{align*}
We now calculate $n(b_1,b_2)$ to get an upperbound for $\sum_{i=1}^{d}P(G(n,p),i)$ using the Tur\'an sieve. If the two pairs of vertices $b_1$ and $b_2$ are the same, then we just have $n(b_1,b_2)=\deg b$. If $b_1$ and $b_2$ have exactly one vertex in common, then we can see that $n(b_1,b_2)=C(n,r,s,d,2)$ and use \eqref{upperdegb}. Hence the only question is when the two pairs of vertices are disjoint.
\newline
\newline
As in our calculations for $\deg b$, to help calculate $n(b_1,b_2)$ in this case, we will calculate a generalised notion of $n(b_1,b_2)$ as follows. Let $0\leq i_0\leq n-2$ and $0\leq i_0'\leq n-2$ where $i_0+i_0'\leq n-2$. Pick two disjoint sets of vertices having $i_0$ and $i_0'$ vertices out of the $n$ labeled vertices, as well as two other vertices, say $v$ and $v'$, out of the $n$ labeled vertices. We will let $C'(n,r,s,d,i_0,i_0')$ denote the number of graphs in $A$ such that there is no path from any of the $i_0$ vertices to vertex $v$ that consists of at most $d$ edges, as well as the requirement that there is no path from any of the $i_0'$ vertices to the vertex $v'$ that consists of at most $d$ edges. If $i_0=0$, then we have $C'(n,r,s,d,i_0,i_0')=C(n,r,s,1,i_0')$ and if $i_0'=0$, then we have $C'(n,r,s,d,i_0,i_0')=C(n,r,s,1,i_0)$. So suppose that $i_0,i_0'>0$. Then we have
\begin{align}
C'(n,r,s,d+1,i_0,i_0')&<\sum_{i_1=0}^{n-2-i_0-i_0'}\binom{n-2-i_0-i_0'}{i_1}\left(s^{i_0}-(s-r)^{i_0}\right)^{i_1}(s-r)^{i_0(n-i_0-i_0'-i_1-1)}\nonumber\\
&\quad\cdot\sum_{i_1'=0}^{n-2-i_0-i_0'-i_1}\binom{n-2-i_0-i_0'-i_1}{i_1'}\left(s^{i_0'}-(s-r)^{i_0'}\right)^{i_1'}(s-r)^{i_0'(n-i_0-i_0'-i_1-i_1'-1)}\nonumber\\
&\quad\cdot s^{\binom{i_0}{2}+i_0i_0'+\binom{i_0'}{2}+i_0+i_1i_0'+i_0'}C'(n-i_0-i_0',r,s,d,i_1,i_1').\label{eqn3a}
\end{align} 
valid for all $1\leq i_0,i_0'\leq n-3$ with $i_0+i_0\leq n-2$, and $d\geq 1$. As well,
\begin{equation*}
C'(n,r,s,1,i_0,i_0')=(s-r)^{i_0+i_0'}s^{\binom{n}{2}-i_0-i_0'}
\end{equation*}
for all $0\leq i_0,i_0'\leq n-2$ with $i_0+i_0'\leq n-2$, completing the formula. Then we can deduce that $C(n,r,s,d,1,1)=n(b_1,b_2)$ if we are working with diameter $d$. Let $D'(n,p,d,i_0,i_0')=\frac{C'(n,r,s,1,i_0,i_0')}{s^{\binom{n}{2}}}$ so that $D'(n,p,d,i_0,i_0')$ is the probability that the edge distance between $v$ and any of the $i_0$ vertices is greater than $d$ and that the edge distance between $v'$ and any of the $i_0'$ vertices is greater than $d$. We will prove that for all $0\leq i_0,i_0'\leq n-2$, $i_0+i_0'\leq n-2$, $0<p<1$, $d\geq 1$ that
\begin{equation}\label{D'D}
D'(n,p,d,i_0,i_0')\leq D(n-1,p,d,i_0+i_0').
\end{equation}
For $d=1$, we have
\begin{align*}
D'(n,p,1,i_0,i_0')=(1-p)^{i_0+i_0'}=D(n-1,p,1,i_0+i_0')
\end{align*}
so \eqref{D'D} holds for $d=1$. Suppose for some $d\geq 1$ \eqref{D'D} holds for all $n\in\mathbb{N}$, $0\leq i_0,i_0'\leq n-2$, $i_0+i_0'\leq n-2$, $0<p<1$. We can see that \eqref{D'D} holds if $i_0=0$ or $i_0'=0$. So assume that $0<i_0,i_0'\leq n-3$ with $i_0+i_0'\leq n-2$. First we have
\begin{align*}
D'(n,p,d+1,i_0,i_0')&<\sum_{i_1=0}^{n-2-i_0-i_0'}\binom{n-2-i_0-i_0'}{i_1}\left(1-(1-p)^{i_0}\right)^{i_1}(1-p)^{i_0(n-i_0-i_0'-i_1-1)}\\
&\quad\cdot\sum_{i_1'=0}^{n-2-i_0-i_0'-i_1}\binom{n-2-i_0-i_0'-i_1}{i_1'}\left(1-(1-p)^{i_0'}\right)^{i_1'}(1-p)^{i_0'(n-i_0-i_0'-i_1-i_1'-1)}\\
&\quad\cdot D'(n-i_0-i_0',p,d,i_1,i_1')\\
&<(1-p)^{i_0(n-i_0-i_0'-1)}\sum_{i_1=0}^{n-2-i_0-i_0'}\binom{n-2-i_0-i_0'}{i_1}\left((1-p)^{-i_0}-1\right)^{i_1}\\
&\quad\cdot(1-p)^{i_0'(n-i_0-i_0'-i_1-1)}\sum_{i_1'=0}^{n-2-i_0-i_0'-i_1}\binom{n-2-i_0-i_0'-i_1}{i_1'}\left((1-p)^{-i_0'}-1\right)^{i_1'}\\
&\quad\cdot D(n-1-i_0-i_0',p,d,i_1+i_1').
\end{align*}
Writing $k=i_1+i_1'$, we have
\begin{align*}
D'(n,p,d+1,i_0,i_0')&<(1-p)^{(i_0+i_0')(n-i_0-i_0'-1)}\sum_{k=0}^{n-2-i_0-i_0'}\binom{n-2-i_0-i_0'}{k}D(n-1-i_0-i_0',p,d,k)\\
&\quad\cdot\left((1-p)^{-i_0'}-1\right)^{k}\sum_{i_1=0}^k\binom{k}{i_1}\left((1-p)^{-i_0}-1\right)^{i_1}(1-p)^{-i_1i_0'}\left((1-p)^{-i_0'}-1\right)^{-i_1}\\
&=(1-p)^{(i_0+i_0')(n-i_0-i_0'-1)}\sum_{k=0}^{n-2-i_0-i_0'}\binom{n-2-i_0-i_0'}{k}D(n-1-i_0-i_0',p,d,k)\\
&\quad\cdot\left((1-p)^{-i_0'}-1\right)^{k}\sum_{i_1=0}^k\binom{k}{i_1}\left(\frac{(1-p)^{-i_0-i_0'}-(1-p)^{-i_0'}}{(1-p)^{-i_0'}-1}\right)^{i_1}\\
&=(1-p)^{(i_0+i_0')(n-i_0-i_0'-1)}\sum_{k=0}^{n-2-i_0-i_0'}\binom{n-2-i_0-i_0'}{k}D(n-1-i_0-i_0',p,d,k)\\
&\quad\cdot\left((1-p)^{-i_0'}-1\right)^{k}\left(1+\frac{(1-p)^{-i_0-i_0'}-(1-p)^{-i_0'}}{(1-p)^{-i_0'}-1}\right)^{k}\\
&=(1-p)^{(i_0+i_0')(n-i_0-i_0'-1)}\sum_{k=0}^{n-2-i_0-i_0'}\binom{n-2-i_0-i_0'}{k}D(n-1-i_0-i_0',p,d,k)\\
&\quad\cdot\left((1-p)^{-i_0-i_0'}-1\right)^{k}\\
&<D(n-1,p,d+1,i_0+i_0').
\end{align*}
Suppose we have $n-1$ labeled vertices. Pick $i_0$ of these vertices where $0\leq i_0\leq n-2$ and another vertex $v$ among the $n-1$ vertices. The number of graphs from $G(n-1,p)$ on these $n-1$ vertices such that there is no path from any of the $i_0$ vertices to vertex $v$ that consists of at most $d$ edges is $C(n-1,r,s,d,i_0)$ where $\frac{r}{s}$ is the edge probability. Adding one more vertex to these $n$ labeled vertices, we can see that $s^{n-1}C(n-1,r,s,d,i_0)\geq C(n,r,s,d,i_0)$. We deduce that $D(n-1,p,d,i_0)\geq D(n,p,d,i_0)$. Thus we have $n(b_1,b_2)\leq s^{\binom{n}{2}}D(n-1,p,d,2)$ whenever $b_1$ and $b_2$ are not the same pair of vertices and hence we can use \eqref{upperdegb} to get an upper bound. Thus, by the Tur\'an sieve, we have
\begin{align*}
P(G(n,p),d)&<\frac{\binom{n}{2}^2D(n-1,p,d,2)+\binom{n}{2}D(n,p,d,1)}{\binom{n}{2}^2D(n,p,d,1)^2}-1\\
&=\frac{D(n-1,p,d,2)}{D(n,p,d,1)^2}-1+\frac{1}{\binom{n}{2}D(n,p,d,1)}\\
&<{(1-p^d)^{-2\left(n^{d-1}+\frac{n^{d-2}}{p}+\frac{n^{d-3}}{p^2}+\ldots+\frac{1}{p^{d-1}}\right)}}h(n-1,p,d,2)\left(1-f(n-1,p,d,2)\right)^{2\cdot g(n-1,p,d,d',2)}\\
&\quad-1+\frac{1}{\binom{n}{2}(1-p^d)^{\left(n^{d-1}+\frac{n^{d-2}}{p}+\frac{n^{d-3}}{p^2}+\ldots+\frac{1}{p^{d-1}}\right)}}.
\end{align*}
\section{Restricted Results for Diameter $d\geq 3$}
Here we impose further restrictions on $n$ and $p$ in Theorem \ref{bigthm} to make our result more clear and meaningful. Since the case $d=2$ was treated in Section $2$, we assume $d\geq 3$.The result is Corollary \ref{bigcor}.
\begin{corollary}\label{bigcor}
Let $d\geq 3$ be fixed. Suppose that
\begin{equation}\label{cond1}
n^{\frac{1}{d}-1}\leq p\leq n^{\frac{1}{d}+\frac{1}{2d^2}-1}.
\end{equation}
Also suppose that
\begin{equation}\label{cond2}
(4^dd)^{2d^2}<n.
\end{equation}
Then we have
\begin{equation*}
P(G(n,p),d)>1-\binom{n}{2}(1-p^d)^{n^{d-1}}\left(1+4^{d+1}dn^{\frac{-1}{2d^2}}\right)
\end{equation*}
and
\begin{equation*}
P(G(n,p),d)<\frac{2(1-p^d)^{-n^{d-1}}\left(1+2n^{\frac{-1}{2d^2}}\right)}{n(n-1)}+4^{d+2}dn^{\frac{-1}{2d^2}}.
\end{equation*}
\end{corollary}
We prove Corollary \ref{bigcor}. Suppose that \eqref{cond1} and \eqref{cond2} hold. From \eqref{cond1} and \eqref{cond2}, we have $2<n^{1/d}\leq np$. From \eqref{cond2}, we can derive that $4d\left(\frac{e}{3}\right)^{3n^{1/d}}<4^ddn^{\frac{-1}{2d^2}}<1$. Also, from \eqref{cond1} and \eqref{cond2}, we can deduce that $\frac{1}{n}+\frac{16(4np)^{d-4}}{7n}<\frac{1}{4}$. Thus
\begin{align}
h(n,p,d,1),h(n-1,p,d,2)&<\left(1-\frac{4}{5}\left(\frac{e}{3}\right)^{4\left(n-1-\frac{16(4np)^{d-4}}{7}\right)p}\right)^{2-d}\nonumber\\
&<\left(1-\frac{4}{5}\left(\frac{e}{3}\right)^{4n^{1/d}\left(1-\frac{1}{4}\right)}\right)^{-d}\nonumber\\
&=\left(1-\frac{4}{5}\left(\frac{e}{3}\right)^{3n^{1/d}}\right)^{-d}\nonumber\\
&<1+4d\left(\frac{e}{3}\right)^{3n^{1/d}}\nonumber\\
&<1+4^ddn^{\frac{-1}{2d^2}}.\label{eqn1}
\end{align}
with the last two inequalities following from \eqref{cond2}. Also, from \eqref{cond1}, we have
\begin{equation*}
p^{d-1}n^{d-2}\leq\left(n^{\frac{1}{d}+\frac{1}{2d^2}-1}\right)^{d-1}n^{d-2}=n^{\frac{-1}{2d}-\frac{1}{2d^2}}.
\end{equation*}
Also, from \eqref{cond2}, we have
\begin{align*}
4^{d-1}dn^{\frac{-1}{2d}-\frac{1}{2d^2}}<4^{d-1}d(4^dd)^{-1}=\frac{1}{4}.
\end{align*}
Thus, from \eqref{cond1}, we have that both $1-f(n,p,d,1)$ and $1-f(n,p,d,2)$ are bounded above by
\begin{align}
&\quad 1-p\prod_{i=0}^{d-2}\left(\frac{1-(1-p)^{2(4np)^i}}{2(4np)^i}\right)\prod_{j=1}^{d-2}\left(\frac{1-(1-p^{j+1})^{n^j}}{n^jp^{j+1}}\right)\nonumber\\
&<1-p\prod_{i=0}^{d-2}\left(\frac{2p(4np)^i-4p^2(4np)^{2i}}{2(4np)^i}\right)\prod_{j=1}^{d-2}\left(\frac{n^jp^{j+1}-n^{2j}p^{2j+2}}{n^jp^{j+1}}\right)\nonumber\\
&=1-p^d\prod_{i=0}^{d-2}\left(1-2p(4np)^i\right)\prod_{j=1}^{d-2}\left(1-n^jp^{j+1}\right)\nonumber\\
&<1-p^d\left(1-2p(4np)^{d-2}\right)^{2d-3}\nonumber\\
&<1-p^d\left(1-4^{d-1}(d-1)p^{d-1}n^{d-2}\right)\nonumber\\
&\leq 1-p^d\left(1-4^{d-1}(d-1)n^{\frac{-1}{2d}-\frac{1}{2d^2}}\right)\nonumber\\
&<(1-p^d)^{\left(1-4^{d-1}(d-1)n^{\frac{-1}{2d}-\frac{1}{2d^2}}\right)\left(1-\frac{p^d}{4}\right)}\nonumber\\
&<(1-p^d)^{\left(1-4^{d-1}(d-1)n^{\frac{-1}{2d}-\frac{1}{2d^2}}\right)\left(1-\frac{n^{\frac{-1}{2d}-\frac{1}{2d^2}}}{4}\right)}\nonumber\\
&<(1-p^d)^{\left(1-4^{d-1}dn^{\frac{-1}{2d}-\frac{1}{2d^2}}\right)}.\label{eqn3}
\end{align}
From \eqref{cond1} and \eqref{cond2}, we can derive
\begin{align*}
\frac{2(4np)^{d-1}}{1-\frac{1}{4np}}\leq n-2.
\end{align*}
Thus we have
\begin{align}
g(n,p,d,d-1,1)&>\left(n-1-\frac{(4np)^{d-3}}{1-\frac{1}{4np}}\right)^{d-1}\nonumber\\
&>\left(n-1-\frac{8(4np)^{d-3}}{7}\right)^{d-1}\nonumber\\
&>\left(n-1-\frac{4^dn^{1-\frac{1}{2d}-\frac{1}{2d^2}}}{56}\right)^{d-1}\nonumber\\
&>n^{d-1}\left(1-\frac{1}{n}-\frac{4^dn^{\frac{-1}{2d}-\frac{1}{2d^2}}}{56}\right)^{d-1}\nonumber\\
&>n^{d-1}\left(1-4^{d-2}dn^{\frac{-1}{2d}-\frac{1}{2d^2}}\right)\label{eqn4}
\end{align}
and
\begin{align}
g(n-1,p,d,d-1,2)&>\left(n-2-\frac{2(4np)^{d-3}}{1-\frac{1}{4np}}\right)^{d-1}\nonumber\\
&>\left(n-2-\frac{16(4np)^{d-3}}{7}\right)^{d-1}\nonumber\\
&>\left(n-2-\frac{4^dn^{1-\frac{1}{2d}-\frac{1}{2d^2}}}{28}\right)^{d-1}\nonumber\\
&>n^{d-1}\left(1-\frac{2}{n}-\frac{4^dn^{\frac{-1}{2d}-\frac{1}{2d^2}}}{28}\right)^{d-1}\nonumber\\
&>n^{d-1}\left(1-4^{d-2}dn^{\frac{-1}{2d}-\frac{1}{2d^2}}\right)\label{eqn12}.
\end{align}
Substituting in \eqref{eqn1}, \eqref{eqn3}, and \eqref{eqn4} into the lower bound in Theorem \ref{bigthm}, we obtain
\begin{align}
P(G(n,p),d)&>1-\binom{n}{2}\left(1+4^ddn^{\frac{-1}{2d^2}}\right)(1-p^d)^{n^{d-1}\left(1-2\cdot 4^{d-2}dn^{\frac{-1}{2d}-\frac{1}{2d^2}}\right)\left(1-4^{d-1}dn^{\frac{-1}{2d}-\frac{1}{2d^2}}\right)}\nonumber\\
&>1-\binom{n}{2}\left(1+4^ddn^{\frac{-1}{2d^2}}\right)(1-p^d)^{n^{d-1}\left(1-2\cdot4^{d-1}dn^{\frac{-1}{2d}-\frac{1}{2d^2}}\right)}.\label{eqn5}
\end{align}
From \eqref{cond1} and \eqref{cond2}, we have
\begin{equation*}
p^dn^{d-1}\left(2\cdot 4^{d-1}dn^{\frac{-1}{2d}-\frac{1}{2d^2}}\right)\leq 2\cdot 4^{d-1}dn^{\frac{-1}{2d^2}}<\frac{1}{2}.
\end{equation*}
Thus, from \eqref{cond2}, we obtain
\begin{align}
(1-p^d)^{-n^{d-1}\left(2\cdot 4^{d-1}dn^{\frac{-1}{2d}-\frac{1}{2d^2}}\right)}<1+4^ddn^{\frac{-1}{2d^2}}.\label{eqn9}
\end{align}
Thus we deduce
\begin{align*}
P(G(n,p),d)&>1-\binom{n}{2}(1-p^d)^{n^{d-1}}\left(1+3\cdot 4^ddn^{\frac{-1}{2d^2}}\right)\\
&>1-\binom{n}{2}(1-p^d)^{n^{d-1}}\left(1+4^{d+1}dn^{\frac{-1}{2d^2}}\right)
\end{align*}
Also, from \eqref{cond1} and \eqref{cond2}, we have
\begin{equation*}
{(1-p^d)^{-2\left(n^{d-1}+\frac{n^{d-2}}{p}+\frac{n^{d-3}}{p^2}+\ldots+\frac{1}{p^{d-1}}\right)}}<(1-p^d)^{-2(n^{d-1})(1+3/(2np))}.
\end{equation*}
and
\begin{equation*}
3p^{d-1}n^{d-2}\leq 3n^{\frac{-1}{2d}-\frac{1}{2d^2}}<3n^{\frac{-1}{2d^2}}<\frac{3}{4^dd}\leq\frac{1}{64}.
\end{equation*}
Thus we can deduce
\begin{equation*}
(1-p^d)^{-2(n^{d-1})(3/(2np))}\leq 1+\frac{64n^{\frac{-1}{2d^2}}}{21}<1+4^ddn^{\frac{-1}{2d^2}}.
\end{equation*}
Thus
\begin{equation}
(1-p^d)^{-2(n^{d-1})(1+3/(2np))}<(1-p^d)^{-2n^{d-1}}\left(1+4^ddn^{\frac{-1}{2d^2}}\right).\label{eqn6}
\end{equation}
Similarly, we can obtain
\begin{equation}
(1-p^d)^{-(n^{d-1})(1+3/(2np))}<(1-p^d)^{-n^{d-1}}\left(1+\frac{192}{127}n^{\frac{-1}{2d^2}}\right).\label{eqn7}
\end{equation}
Substituting in \eqref{eqn1}, \eqref{eqn3}, \eqref{eqn12}, \eqref{eqn9}, \eqref{eqn6}, and \eqref{eqn7} into the upper bound in Theorem \ref{bigthm}, we obtain
\begin{align*}
P(G(n,p),d)&<\left(1+4^ddn^{\frac{-1}{2d^2}}\right)^2(1-p^d)^{-n^{d-1}\left(4^ddn^{\frac{-1}{2d}-\frac{1}{2d^2}}\right)}-1+\frac{2(1-p^d)^{-n^{d-1}}\left(1+\frac{192}{127}n^{\frac{-1}{2d^2}}\right)}{n(n-1)}\\
&<\left(1+4^ddn^{\frac{-1}{2d^2}}\right)^4-1+\frac{2(1-p^d)^{-n^{d-1}}\left(1+\frac{192}{127}n^{\frac{-1}{2d^2}}\right)}{n(n-1)}\\
&<\frac{2(1-p^d)^{-n^{d-1}}\left(1+\frac{192}{127}n^{\frac{-1}{2d^2}}\right)}{n(n-1)}+4^{d+2}dn^{\frac{-1}{2d^2}}
\end{align*}
with the second inequality following from \eqref{eqn9} and the third inequality following from \eqref{cond2}.
\section{Directed Graphs for diameter $d\geq 2$}
Using the above methods, we can obtain similar results about the probability of a random directed graph on n vertices having diameter $d$ where each directed edge is chosen independently with probability $p$. Furthermore, for any two vertices, say $v_1$ and $v_2$, the existence of the edge from $v_1$ to $v_2$ has probability $p$, while the existence of the edge from $v_2$ to $v_1$ also occurs with probability $p$, and these two edges occur independently. We proceed exactly as above the only changes being as follows. We replace the factor of $s^{\binom{i_0}{2}}$ in \eqref{eqn2a} and \eqref{eqn2} with $s^{i_0(n-1)}$, replace the factor of $s^{\binom{i_0}{2}+i_0i_0'+\binom{i_0'}{2}+i_0+i_1i_0'+i_0'}$ with $s^{(i_0+i_0')n+i_1i_0'}$ in \eqref{eqn3a}, and replace $\binom{n}{2}$ wherever it occurs with $n(n-1)$. Consequently, in Theorem \ref{bigthm} and Corollary \ref{bigcor}, we multiply the second term of the lower bound by $2$, divide the last term in the upper bound in Theorem \ref{bigthm} by $2$, and divide the first term in the upper bound in Corollary \ref{bigcor} to get the analogous results for random directed graphs. Everything else is left unchanged.
\section{Analysis of $k$-partite Graphs for diameter $d\geq 2$}
Here we analyze the diameters of $k$-partite graphs for some fixed $k\geq 3$. Let $G(n_1,n_2,\ldots, n_k,p)$ denote the set of all simple $k$-partite graphs with partite sets of size $\mathbf{n^{(k)}}$ vertices respectively where each edge is chosen independently with probability $p$. Here we obtain upper and lower bounds on the probability of a random simple $k$-partite graph with partite sets of sizes $\mathbf{n^{(k)}}$ vertices with independent edge selection having diameter at most $d$ for any specific $d\geq 2$, $d\in\mathbb{N}$. Again, analogous to our treatment of the random graphs $G(n,p)$, we impose restrictions on $n_1,n_2,\ldots,n_k$, and $p$. Then in the next section, we refine this result to make it more clear and meaningful by imposing further restrictions on $n_1,n_2,\ldots,n_k$, and $p$. We use the following notation:
\begin{notation}
Let
\begin{equation*}
[n_j,l]_{1\leq j\leq k}:=\{(i_1,i_2,\ldots,i_k):0\leq i_l\leq n_l-1,\forall 1\leq j\leq k,j\neq l\quad 0\leq i_j\leq n_j\},
\end{equation*}
\begin{equation*}
\mathbf{i^{(0)}}:=(i_{0,1},i_{0,2},\ldots,i_{0,k}),
\end{equation*}
\begin{equation*}
\mathbf{i^{(1)}}:=(i_{1,1},i_{1,2},\ldots,i_{1,k}),
\end{equation*}
\begin{equation*}
\mathbf{n^{(k)}}:=(n_1,n_2,\ldots,n_k)
\end{equation*}
\end{notation}
\begin{note}
Throughout this note let
\begin{equation*}
u\left(\mathbf{n^{(k)}},m,q,l\right):=\sum_{(i_1,i_2,\ldots,i_m)\in[k]^{m,q,l\neq}}n_{i_1}n_{i_2}\cdots n_{i_m},
\end{equation*}
\begin{align*}
&\quad v_1\left(\mathbf{n^{(k)}},\mathbf{i^{(0)}},p,m,q,l\right)\\
&:=\sum_{(i_1,i_2,\ldots,i_m)\in[k]^{m,q,l\neq}}(n_{i_1}-\mathbbm{1}_l(i_1)-i_{0,i_1})(n_{i_2}-\mathbbm{1}_l(i_2)-i_{0,i_2}-4n_{i_2}pi_0)(n_{i_3}-i_{0,i_3}-4n_{i_3}pi_0-4n_{i_3}p(4npi_0))\\
&\qquad\qquad\qquad\qquad\cdot(n_{i_4}-\mathbbm{1}_l(i_4)-i_{0,i_4}-4n_{i_4}pi_0-4n_{i_4}p(4np)i_0-4n_{i_4}p(4np)^2i_0)\\
&\qquad\qquad\qquad\qquad\cdots(n_{i_m}-\mathbbm{1}_l(i_m)-i_{0,i_m}-4n_{i_m}pi_0-4n_{i_m}p(4np)i_0-\ldots-4n_{i_m}p(4np)^{m-2}i_0),
\end{align*}
and
\begin{align*}
&\quad v_2\left(\mathbf{n^{(k)}},\mathbf{i^{(0)}},p,m,q,l\right)\\
&:=\sum_{(i_1,i_2,\ldots,i_m)\in[k]^{m,q,l\neq}}(n_{i_1}-\mathbbm{1}_l(i_1)-i_{0,i_1})(n_{i_2}-\mathbbm{1}_l(i_2)-i_{0,i_2}-4n_{i_2}pi_0)\\
&\qquad\qquad\qquad\qquad\cdot(n_{i_3}-\mathbbm{1}_l(i_3)-i_{0,i_3}-4n_{i_3}pi_0-4n_{i_3}p(4npi_0))\\
&\qquad\qquad\qquad\qquad\cdot(n_{i_4}-\mathbbm{1}_l(i_4)-i_{0,i_4}-4n_{i_4}pi_0-4n_{i_4}p(4np)i_0-4n_{i_4}p(4np)^2i_0)\\
&\qquad\qquad\qquad\qquad\cdots(n_{i_{m-1}}-\mathbbm{1}_l(i_{m-1})-i_{0,i_{m-1}}-4n_{i_{m-1}}pi_0-4n_{i_{m-1}}p(4np)i_0-\ldots-4n_{i_{m-1}}p(4np)^{m-3}i_0)\\
&\qquad\qquad\qquad\qquad\cdot(n_{i_m}-\mathbbm{1}_l(i_m)-i_{0,i_m}-4n_{i_m}pi_0-4n_{i_m}p(4np)i_0-\ldots-4n_{i_m}p(4np)^{m-3}i_0),
\end{align*}
where
\begin{equation*}
[k]^{m,q,l,\neq}:=\{(i_1,i_2,\ldots,i_m):i_1\neq q,i_m\neq l,\forall 1\leq j\leq m-1\quad i_j\neq i_{j+1}\}
\end{equation*}
and
\begin{equation*}
\mathbbm{1}_l(x):=
\begin{cases} 
      1 & x=l \\
      0 & x\neq l
   \end{cases}
\end{equation*}
with the conventions that
\begin{equation*}
v_1\left(\mathbf{n^{(k)}},\mathbf{i^{(0)}},p,1,q,l\right)= v_2\left(\mathbf{n^{(k)}},\mathbf{i^{(0)}},p,1,q,l\right)=\sum_{(i_1)\in[k]^{1,q,l,\neq}}(n_{i_1}-\mathbbm{1}_l(i_1)-i_{0,i_1})
\end{equation*}
and
\begin{equation*}
v_2\left(\mathbf{n^{(k)}},\mathbf{i^{(0)}},p,2,q,l\right)=\sum_{(i_1,i_2)\in[k]^{1,q,l,\neq}}(n_{i_1}-\mathbbm{1}_l(i_1)-i_{0,i_1})(n_{i_2}-\mathbbm{1}_l(i_2)-i_{0,i_2}-4n_{i_2}pi_0).
\end{equation*}
Also, let
\begin{equation*}
h_k\left(\mathbf{n^{(k)}},p,d,\mathbf{i^{(0)}}\right):=\prod_{j=1}^k\left(1-\frac{4}{5}\left(\frac{e}{3}\right)^{4p\left(n_j-i_{0,j}-4n_jpi_0-(4np)4n_jpi_0-(4np)^24n_jpi_0-\ldots-(4np)^{d-5}4n_jpi_0\right)}\right)^{2-d}
\end{equation*}
with the conventions that $h_k\left(\mathbf{n^{(k)}},p,2,\mathbf{i^{(0)}}\right)=1$,
\begin{equation*}
h_k\left(\mathbf{n^{(k)}},p,3,\mathbf{i^{(0)}}\right)=\prod_{j=1}^k\left(1-\frac{4}{5}\left(\frac{e}{3}\right)^{4n_jp}\right)^{-1},
\end{equation*}
and
\begin{equation*}
h_k\left(\mathbf{n^{(k)}},p,4,\mathbf{i^{(0)}}\right)=\prod_{j=1}^k\left(1-\frac{4}{5}\left(\frac{e}{3}\right)^{4p(n_j-i_{0,j})}\right)^{-2}.
\end{equation*}
Also, let
\begin{equation*}
g_k\left(\mathbf{n^{(k)}},j,l,p,d,d',\mathbf{i^{(0)}}\right):=\begin{cases}
                      n-n_l & d=2,j=l\\
                      n-n_j-n_l & d=2,j\neq l\\
                      1+\sum_{m=1}^{d'+1}v_1\left(\mathbf{n^{(k)}},\mathbf{i^{(0)}},p,m,j,l\right) & j\neq l,d\geq 3, d'<d-3\\
                      \sum_{m=1}^{d'+1}v_1\left(\mathbf{n^{(k)}},\mathbf{i^{(0)}},p,m,l,l\right) & j=l,d\geq 3, d'<d-3\\
                      1+\sum_{m=1}^{d-3}v_1\left(\mathbf{n^{(k)}},\mathbf{i^{(0)}},p,m,j,l\right)\\
                     \quad+v_2\left(\mathbf{n^{(k)}},\mathbf{i^{(0)}},p,d-1,j,l\right) & j\neq l,d\geq 3, d'\geq d-3\\
                     \sum_{m=1}^{d-3}v_1\left(\mathbf{n^{(k)}},\mathbf{i^{(0)}},p,m,l,l\right)\\
                     \quad+v_2\left(\mathbf{n^{(k)}},\mathbf{i^{(0)}},p,d-1,l,l\right) & j=l,d\geq 3, d'\geq d-3
                       \end{cases}
\end{equation*}
and
\begin{equation*}
g_k'\left(\mathbf{n^{(k)}},j,l,p,d,d',\mathbf{i^{(0)}}\right):=\begin{cases}
                    
                    g_k\left(\mathbf{n^{(k)}},j,l,p,d+2,d',\mathbf{i^{(0)}}\right) & d\geq 2, d'\leq d-2\\
                     
                    g_k\left(\mathbf{n^{(k)}},j,l,p,d+2,d-2,\mathbf{i^{(0)}}\right) & d\geq 2, d'>d-2.
                      \end{cases}
\end{equation*}
\end{note}
We will prove the following theorem.
\begin{theorem}\label{bigthmkpartite}
Fix $d\geq 2$, $d\in\mathbb{N}$. Let $G\left(\mathbf{n^{(k)}},p\right)$ denote the set of all simple $k$-partite graphs with partite vertex sets of sizes $n_1,n_2,\ldots,n_k$ and where each edge is chosen independently with probability $p$. Also, let $P\left(G\left(\mathbf{n^{(k)}},p\right),d\right)$ be the probability of a graph from $G\left(\mathbf{n^{(k)}},p\right)$ having diameter at most $d$. Suppose that
\begin{equation*}
8n_jp\left(1+4np+(4np)^2+\ldots+(4np)^{d'-1}\right)\leq n_j-4
\end{equation*}
for all $1\leq j\leq k$ where $d'\geq 0$. We have
\begin{align*}
P(G(\mathbf{n^{(k)}},p),d)&>1-\sum_{l=1}^k\binom{n_l}{ 2}h_k\left(\mathbf{n^{(k)}},p,d,\mathbf{1^{(l)}}\right)\left(1-f(n,p,d,1)\right)^{g_k\left(\mathbf{n^{(k)}},l,l,p,d,d',\mathbf{1^{(j)}}\right)}\\
&\qquad-\sum_{1\leq j<l\leq k}n_jn_lh_k\left(\mathbf{n^{(k)}},p,d,\mathbf{1^{(j)}}\right)\left(1-f(n,p,d,1)\right)^{g_k\left(\mathbf{n^{(k)}},j,l,p,d,d',\mathbf{1^{(j)}}\right)}
\end{align*}
\newpage
and
\begin{align*}
&\quad P\left(G\left(\mathbf{n^{(k)}},p\right),d\right)\left(\sum_{l=1}^k\binom{n_l}{2}(1-p^d)^{\sum_{m=0}^{d-1}u\left(\mathbf{n^{(k)}},m,l,l\right)p^{m-d+1}}+\sum_{1\leq j<l\leq k}n_jn_l(1-p^d)^{\sum_{m=0}^{d-1}u\left(\mathbf{n^{(k)}},m,j,l\right)p^{m-d+1}}\right)^2\\
&<\sum_{1\leq l\leq k}\left(n_l\left(n_l-1\right)\left(n_l-2\right)+\binom{n_l}{2}\binom{n_l-2}{2}\right)h_k\left(\mathbf{n^{(k)}}-\mathbf{1^{(l)}},p,d,2\cdot\mathbf{1^{(l)}}\right)\\
&\qquad\qquad\qquad\qquad\qquad\qquad\qquad\qquad\qquad\qquad\qquad\cdot\left(1-f(n,p,d,2)\right)^{2g_k\left(\mathbf{n^{(k)}}-\mathbf{1^{(l)}},l,l,p,d,d',2\cdot\mathbf{1^{(l)}}\right)}\\
&\quad+\sum_{1\leq j\neq l\leq k }\left(n_l\left(n_l-1\right)n_j+2\binom{n_l}{2}\left(n_l-2\right)n_j\right)h_k\left(\mathbf{n^{(k)}}-\mathbf{1^{(l)}},p,d,\mathbf{1^{(l)}}+\mathbf{1^{(j)}}\right)\\
&\qquad\qquad\qquad\qquad\qquad\qquad\qquad\cdot\left(1-f(n,p,d,2)\right)^{g_k\left(\mathbf{n^{(k)}}-\mathbf{1^{(l)}},l,l,p,d,d',\mathbf{1^{(l)}}+\mathbf{1^{(j)}}\right)+g_k\left(\mathbf{n^{(k)}}-\mathbf{1^{(l)}},j,l,p,d,d',\mathbf{1^{(l)}}+\mathbf{1^{(j)}}\right)}\\
&\quad+\sum_{1\leq j\neq l\leq k }\left(n_ln_j\left(n_j-1\right)+n_ln_j\left(n_l-1\right)\left(n_j-1\right)\right)h_k\left(\mathbf{n^{(k)}}-\mathbf{1^{(l)}},p,d,2\cdot\mathbf{1^{(j)}}\right)\\
&\qquad\qquad\qquad\qquad\cdot\left(1-f(n,p,d,2)\right)^{2g_k\left(\mathbf{n^{(k)}}-\mathbf{1^{(l)}},j,l,p,d,d',2\cdot\mathbf{1^{(j)}}\right)}\\
&\quad+\sum_{1\leq l_1\neq l_2\leq k}\binom{n_{l_1}}{2}\binom{n_{l_2}}{2}h_k\left(\mathbf{n^{(k)}},p,d,\mathbf{1^{(l_1)}}+\mathbf{1^{(l_2)}}\right)\left(1-f\left(n,p,d,2\right)\right)^{g_k'\left(\mathbf{n^{(k)}}-\mathbf{1^{(l_2)}},l_1,l_1,p,d,d',\mathbf{1^{(l_1)}}+\mathbf{1^{(l_2)}}\right)}\\
&\qquad\qquad\qquad\qquad\cdot\left(1-f\left(n,p,d,2\right)\right)^{g_k'\left(\mathbf{n^{(k)}}-\mathbf{1^{(l_1)}},l_2,l_2,p,d,d',\mathbf{1^{(l_1)}}+\mathbf{1^{(l_2)}}\right)}\\
&\quad+\sum_{\substack{1\leq j_1,j_2,l\leq k\\l,j_1,j_2\text{ all distinct}}}n_l^2n_{j_1}n_{j_2}h_k\left(\mathbf{n^{(k)}}-\mathbf{1^{(l)}},p,d,\mathbf{1^{(j_1)}}+\mathbf{1^{(j_2)}}\right)\\
&\qquad\qquad\qquad\qquad\cdot\left(1-f(n,p,d,2)\right)^{g_k\left(\mathbf{n^{(k)}}-\mathbf{1^{(l)}},j_1,l,p,d,d',\mathbf{1^{(j_1)}}+\mathbf{1^{(j_2)}}\right)+g_k\left(\mathbf{n^{(k)}}-\mathbf{1^{(l)}},j_2,l,p,d,d',\mathbf{1^{(j_1)}}+\mathbf{1^{(j_2)}}\right)}\\
&\quad+\sum_{\substack{1\leq j,l_1,l_2\leq k\\j,l_1,l_2\text{ all distinct}}}\binom{n_{l_1}}{2}n_{l_2}n_jh_k\left(\mathbf{n^{(k)}},p,d,\mathbf{1^{(l_1)}}+\mathbf{1^{(j)}}\right)\left(1-f\left(n,p,d,2\right)\right)^{g_k'\left(\mathbf{n^{(k)}}-\mathbf{1^{(l_2)}},l_1,l_1,p,d,d',\mathbf{1^{(l_1)}}+\mathbf{1^{(j)}}\right)}\\
&\qquad\qquad\qquad\qquad\cdot\left(1-f\left(n,p,d,2\right)\right)^{g_k'\left(\mathbf{n^{(k)}}-\mathbf{1^{(l_1)}},j,l_2,p,d,d',\mathbf{1^{(l_1)}}+\mathbf{1^{(j)}}\right)}\\
&\quad+\sum_{\substack{1\leq l_1,l_2,j_1,j_2\leq k\\l_1,l_2,j_1,j_2\text{ all distinct}}}\frac{l_1l_2j_1j_2}{4}h_k\left(\mathbf{n^{(k)}},p,d,\mathbf{1^{(j_1)}}+\mathbf{1^{(j_2)}}\right)\left(1-f\left(n,p,d,2\right)\right)^{g_k'\left(\mathbf{n^{(k)}}-\mathbf{1^{(l_2)}},j_1,l_1,p,d,d',\mathbf{1^{(j_1)}}+\mathbf{1^{(j_2)}}\right)}\\
&\qquad\qquad\qquad\qquad\cdot\left(1-f\left(n,p,d,2\right)\right)^{g_k'\left(\mathbf{n^{(k)}}-\mathbf{1^{(l_1)}},j_2,l_2,p,d,d',\mathbf{1^{(j_1)}}+\mathbf{1^{(j_2)}}\right)}\\
&\quad-\left(\sum_{l=1}^k\binom{n_l}{2}(1-p^d)^{\sum_{m=0}^{d-1}u\left(\mathbf{n^{(k)}},m,l,l\right)p^{m-d+1}}+\sum_{1\leq j<l\leq k}n_jn_l(1-p^d)^{\sum_{m=0}^{d-1}u\left(\mathbf{n^{(k)}},m,j,l\right)p^{m-d+1}}\right)^2\\
&\quad+\sum_{l=1}^k\binom{n_l}{2}(1-p^d)^{\sum_{m=0}^{d-1}u\left(\mathbf{n^{(k)}},m,l,l\right)p^{m-d+1}}+\sum_{1\leq j<l\leq k}n_jn_l(1-p^d)^{\sum_{m=0}^{d-1}u\left(\mathbf{n^{(k)}},m,j,l\right)p^{m-d+1}}.
\end{align*}
\end{theorem}
We will now prove Theorem \ref{bigthmkpartite}.
\newline
\newline
For each $n\in\mathbb{N}$, let $G(\mathbf{n^{(k)}},p)$ denote the set of all $k$-partite graphs with partite sets of sizes $\mathbf{n^{(k)}}$ vertices with edge probability $p$, and let $P(G(\mathbf{n^{(k)}},p))$ be the probability of a graph from $G(n_1,n_2,p)$ having diameter at most $d$. Let $p=\frac{r}{s}$ where $r=r(n), s=s(n)\in\mathbb{N}$. We let $A$ be the set of all graphs in $G(\mathbf{n^{(k)}},p)$, allowing for a number of duplicates of each possible graph to accommodate the edge probability $p$, so that
\begin{equation*}
|A|=\sum_{k=0}^t\binom{t}{k}r^k(s-r)^{t-k}=s^t
\end{equation*}
where $t$ is the number of edges in the respective complete $k$-partite graph. We let $B$ be all pairs of vertices that occur in the graph so $|B|=\binom{n}{2}$ where $n$ is the total number of vertices. For a graph $a\in A$ and a pair of vertices $b\in B$, we say $a\sim b$ if there is no path between the pair of vertices $b$ that consists of at most $d-1$ edges. Thus, we will have $\omega(a)=0$ if and only if $a$ is connected with diameter at most $d$.
\newline
\newline
Pick a pair of vertices $b\in B$ and call them $v_1$ and $v_2$. To calculate $\deg b$, we need to calculate the number of graphs in $A$ such that there is no path from $v_1$ to $v_2$ that consists of at most $d$ edges. To help with this calculation, we will calculate a generalised notion of $\deg b$ as follows. Let $1\leq l\leq k$ and $0\leq i_{0,j}\leq n_j$ for all $j\neq l$ and $0\leq i_{0,l}\leq n_l-1$ . Pick a specific set of $i_{0,j}$ vertices out of the labeled vertices in the partite set consisting of $n_j$ vertices, as well as another vertex, say $v$, in the partite set consisting of $n_l$ vertices. We will let $C_k(\mathbf{n^{(k)}},l,r,s,d,\mathbf{i^{(0)}})$ denote the number of graphs in $A$ such that there is no path from any of the $i_0$ vertices to vertex $v$ that consists of at most $d$ edges where the $i_0$ vertices come from the partite set that consists of $v_j$ vertices. Let $i_0=i_{0,1}+i_{0,2}+\ldots+i_{0,k}$. We can derive the recursive formula
\begin{align}
&\quad C_k(\mathbf{n^{(k)}},l,r,s,d+1,\mathbf{i^{(0)}})\nonumber\\
&=\sum_{\mathbf{i^{(1)}}\in[n_j-i_{0,j},l]_{1\leq j\leq k}}\prod_{j=1}^k\binom{n_j-i_{0,j}-\mathbbm{1}_l(j)}{i_{1,j}}\left(s^{i_0-i_{0,j}}-(s-r)^{i_0-i_{0,j}}\right)^{i_{1,j}}(s-r)^{(i_0-i_{0,j})(n_j-i_{0,j}-i_{1,j})}s^{\frac{(i_0-i_{0,j})i_{0,j}}{2}}\nonumber\\
&\quad\cdot C_k(\mathbf{n^{(k)}}-\mathbf{i^{(0)}},l,r,s,d,\mathbf{i^{(1)}})\label{eqn32}
\end{align}
valid so long as at least two of the $i_{0,j}$ values are nonzero. If, however, only one of them is nonzero, say $i_{0,j}$, then the factor $\left(s^{i_0-i_{0,j}}-(s-r)^{i_0-i_{0,j}}\right)^{i_{1,j}}$ is replaced by $1$ for the respective $j$ value. Everything else is left unchanged.
Also,
\begin{equation*}
C_k(\mathbf{n^{(k)}},l,r,s,1,\mathbf{i^{(0)}})=(s-r)^{i_0-i_{0,l}}s^{t-i_0+i_{0,l}}.
\end{equation*}
Then we can deduce that $C_k(\mathbf{n^{(k)}},l,r,s,d,\mathbf{i^{(0)}})=\deg b$ if all of the $i_{0,j}$ values are $0$, except for one of them having the value $1$ if we are working with diameter $d$. Let $D_k\left(\mathbf{n^{(k)}},l,p,d,\mathbf{i^{(0)}}\right)=\frac{C_k(\mathbf{n^{(k)}},l,r,s,d,\mathbf{i^{(0)}})}{s^t}$ so that $D_k\left(\mathbf{n^{(k)}},l,p,d,\mathbf{i^{(0)}}\right)$ is the probability that the edge distance between $v$ and any of the $i_0$ vertices is greater than $d$. We will prove that for all $(\mathbf{i^{(0)}})\in[n_j,l]_{1\leq j\leq k}$ we have
\begin{equation}
D_k\left(\mathbf{n^{(k)}},l,p,d,\mathbf{i^{(0)}}\right)\geq(1-p^d)^{\sum_{j=1}^ki_{0,j}\sum_{m=0}^{d-1}u\left(\mathbf{n^{(k)}},m,j,l\right)p^{m-d+1}}.\label{eqn31}
\end{equation}
If we also have the additional constraints $1<\frac{n_j}{i_{0,j}+4n_jpi_0+4n_jp(4np)i_0+(4n_jp)(4np)^2i_0+\ldots+(4n_jp)(4np)^{d'-1}i_0}$ for all $1\leq j\leq k$ with $j\neq l$, and $1<\frac{n_l-1}{i_{0,l}+4n_lpi_0+4n_lp(4np)i_0+(4n_lp)(4np)^2i_0+\ldots+(4n_lp)(4np)^{d'-1}i_0}$ where $d'\geq 0$, then we also have
\begin{equation}
D_k\left(\mathbf{n^{(k)}},l,p,d,\mathbf{i^{(0)}}\right)\leq h_k\left(\mathbf{n^{(k)}},p,d+1,\mathbf{i^{(0)}}\right)\left(1-f(n,p,d+1,i_0)\right)^{\sum_{j=1}^{k}i_{0,j}g_k\left(\mathbf{n^{(k)}},j,l,p,d,d',\mathbf{i^{(0)}}\right)}\label{eqn33}.
\end{equation}
For $d=1$, we have $D_k\left(\mathbf{n^{(k)}},l,p,\mathbf{i^{(0)}}\right)=(1-p)^{i_0-i_{0,l}}$. Suppose for some $d\geq 1$ \eqref{eqn31} holds for all $(\mathbf{i^{(0)}})\in[n_j,l]_{1\leq j\leq k}$, and $0<p<1$. We will prove it holds for $d+1$. From \eqref{eqn32} we have
\begin{align*}
&\quad D_k\left(\mathbf{n^{(k)}},l,p,d+1,\mathbf{i^{(0)}}\right)\\
&=\sum_{\mathbf{i^{(1)}}\in[n_j-i_{0,j},l]_{1\leq j\leq k}}\prod_{j=1}^k\binom{n_j-i_{0,j}-\mathbbm{1}_l(j)}{i_{1,j}}(1-p)^{(i_0-i_{0,j})(n_j-i_{0,j})}\left((1-p)^{i_{0,j}-i_0}-1\right)^{i_{1,j}}\\
&\quad\cdot D_k\left(\mathbf{n^{(k)}}-\mathbf{i^{(0)}},l,p,d,\mathbf{i^{(1)}}\right)\\
&>\sum_{\mathbf{i^{(1)}}\in[n_j-i_{0,j},l]_{1\leq j\leq k}}\prod_{j=1}^k\binom{n_j-i_{0,j}-\mathbbm{1}_l(j)}{ i_{1,j}}(1-p)^{(i_0-i_{0,j})(n_j-i_{0,j})}\left((1-p)^{i_{0,j}-i_0}-1\right)^{i_{1,j}}\\
&\qquad\qquad\qquad\qquad\qquad\cdot(1-p^d)^{i_{1,j}\sum_{m=0}^{d-1}u\left(\mathbf{n^{(k)}},m,j,l\right)p^{m-d+1}}\\
&=\prod_{j=1}^k(1-p)^{(i_0-i_{0,j})(n_j-i_{0,j})}\left(1+\left((1-p)^{i_{0,j}-i_0}-1\right)(1-p^d)^{\sum_{m=0}^{d-1}u\left(\mathbf{n^{(k)}},m,j,l\right)p^{m-d+1}}\right)^{n_j-i_{0,j}-\mathbbm{1}_l(j)}\\
&=(1-p)^{(i_0-i_{0,l})}\prod_{j=1}^k\left((1-p)^{i_0-i_{0,j}}+\left(1-(1-p)^{i_0-i_{0,j}}\right)(1-p^d)^{\sum_{m=0}^{d-1}u\left(\mathbf{n^{(k)}},m,j,l\right)p^{m-d+1}}\right)^{n_j-i_{0,j}-\mathbbm{1}_l(j)}.
\end{align*}
Using Lemma \ref{lem1} and doing a change in variables we thus have
\begin{align*}
&\quad D_k\left(\mathbf{n^{(k)}},l,p,d+1,\mathbf{i^{(0)}}\right)\\
&>(1-p)^{i_0-i_{0,l}}\prod_{j=1}^k(1-p^{d+1})^{(i_0-i_{0,j})(n_j-\mathbbm{1}_l(j))\sum_{m=0}^{d-1}u\left(\mathbf{n^{(k)}},m,j,l\right)p^{m-d+1}}\\
&>(1-p^{d+1})^{(1+p^{-1}+p^{-2}+\ldots+p^{-d})(i_0-i_{0,l})}\prod_{j=1}^k(1-p^{d+1})^{(i_0-i_{0,j})(n_j-\mathbbm{1}_l(j))\sum_{m=0}^{d-1}u\left(\mathbf{n^{(k)}},m,j,l\right)p^{m-d+1}}\\
&>(1-p^{d+1})^{p^{-d}(i_0-i_{0,l})}\prod_{j=1}^k(1-p^{d+1})^{(i_0-i_{0,j})n_j\sum_{m=0}^{d-1}u\left(\mathbf{n^{(k)}},m,j,l\right)p^{m-d+1}}\\
&=(1-p^{d+1})^{p^{-d}(i_0-i_{0,l})}\prod_{q=1}^k(1-p^{d+1})^{i_{0,q}\sum_{\substack{j=1\\j\neq q}}^{k}n_j\sum_{m=0}^{d-1}u\left(\mathbf{n^{(k)}},m,j,l\right)p^{m-d+1}}\\
&=\prod_{q=1}^k(1-p^{d+1})^{i_{0,q}\sum_{m=0}^du\left(\mathbf{n^{(k)}},m,q,l\right)p^{m-d}}.
\end{align*}
Thus \eqref{eqn31} is proved. Next we prove \eqref{eqn33} again by induction on $d$. For $d=2$, applying Lemma \ref{lem1} we have
\begin{align*}
&\quad D_k\left(\mathbf{n^{(k)}},l,p,2,\mathbf{i^{(0)}}\right)\\
&=\prod_{j=1}^k(1-p)^{(i_0-i_{0,j})(n_j-i_{0,j})}\sum_{i_{1,j}=0}^{n_j-i_{0,j}-\mathbbm{1}_l(j)}\binom{n_j-i_{0,j}-\mathbbm{1}_l(j)}{i_{1,j}}\left((1-p)^{i_{0,j}-i_0}-1\right)^{i_{1,j}}(1-p)^{i_{1,j}}\\
&=(1-p)^{i_0-i_{0,l}}\prod_{\substack{j=1\\j\neq l}}^k\left(1-p+p(1-p)^{i_0-i_{0,j}}\right)^{n_j-i_{0,j}}\\
&\leq\prod_{\substack{j=1\\j\neq l}}^k(1-f(n,p,2,i_0))^{(i_0-i_{0,j})n_j}\\
&=\prod_{m=1}^k\prod_{\substack{j=1\\j\neq l,m}}^k(1-f(n,p,2,i_0))^{i_{0,m}n_j}\\
&=(1-f(n,p,2,i_0))^{i_{0,l}(n-n_l)}\prod_{\substack{m=1\\m\neq l}}^k(1-f(n,p,2,i_0))^{i_{0,m}(n-n_m-n_l)}.
\end{align*}
Suppose for some $d\geq 2$ \eqref{eqn33} holds for all $\mathbf{i^{(0)}}$ in the stated ranges, and $0<p<1$. We will prove \eqref{eqn33} holds for $d+1$. We divide into three cases.
\setcounter{case}{0}
\begin{case}{$\frac{n_j}{i_{0,j}+4n_jpi_0}\leq 1$ for all $1\leq j\leq k$ with $j\neq l$, and $\frac{n_l-1}{i_{0,l}+4n_lpi_0}\leq 1$}
\normalfont
\newline
\newline
We have the following:
\begin{align*}
&\quad D_k\left(\mathbf{n^{(k)}},l,p,d+1,\mathbf{i^{(0)}}\right)\\
&<\sum_{\mathbf{i^{(1)}}\in[n_j-i_{0,j},l]_{1\leq j\leq k}}h_k\left(\mathbf{n^{(k)}}-\mathbf{i^{(0)}},p,d,\mathbf{i^{(1)}}\right)\\
&\quad\cdot\prod_{j=1}^k\binom{n_j-i_{0,j}-\mathbbm{1}_l(j)}{i_{1,j}}(1-p)^{(i_0-i_{0,j})(n_j-i_{0,j})}\left((1-p)^{i_{0,j}-i_0}-1\right)^{i_{1,j}}\\
&\qquad\qquad\cdot(1-f(n-i_0,p,d,i_1))^{i_{1,j}g_k\left(\mathbf{n^{(k)}}-\mathbf{i^{(0)}},j,l,p,d,0,\mathbf{i^{(1)}}\right)}.
\end{align*}
We can deduce that $h_k\left(\mathbf{n^{(k)}}-\mathbf{i^{(0)}},p,d,\mathbf{i^{(1)}}\right)\leq h_k\left(\mathbf{n^{(k)}}-\mathbf{i^{(0)}},p,d,4np\mathbf{i^{(0)}}\right)<h_k\left(\mathbf{n^{(k)}},p,d+1,\mathbf{i^{(0)}}\right)$ for all $1\leq j\leq k$ and from Lemma \ref{lem3}, we can deduce that $f(n,p,d,4npi_0)<f(n-i_0,p,d,i_1)$. As well, $g_k\left(\mathbf{n^{(k)}}-\mathbf{i^{(0)}},j,l,p,d,0,\mathbf{n^{(k)}}-\mathbf{i^{(0)}}-\mathbf{1^{(l)}}\right)\leq g_k\left(\mathbf{n^{(k)}}-\mathbf{i^{(0)}},j,l,p,d,0,\mathbf{i^{(1)}}\right)$. Thus we have
\begin{align*}
&\quad D_k\left(\mathbf{n^{(k)}},l,p,d+1,\mathbf{i^{(0)}}\right)\\
&<h_k\left(\mathbf{n^{(k)}},p,d+1,\mathbf{i^{(0)}}\right)\prod_{j=1}^k(1-p)^{(i_0-i_{0,j})(n_j-i_{0,j})}\\
&\qquad\qquad\qquad\qquad\qquad\qquad\cdot\sum_{i_{1,j}=0}^{n_j-i_{0,j}-\mathbbm{1}_l(j)}\binom{n_j-i_{0,j}-\mathbbm{1}_l(j)}{i_{1,j}}\left((1-p)^{i_{0,j}-i_0}-1\right)^{i_{1,j}}\\
&\qquad\qquad\qquad\qquad\qquad\qquad\cdot(1-f(n,p,d,4npi_0))^{i_{1,j}g_k\left(\mathbf{n^{(k)}}-\mathbf{i^{(0)}},j,l,p,d,0,\mathbf{n^{(k)}}-\mathbf{i^{(0)}}-\mathbf{1^{(l)}}\right)}\\
&=h_k\left(\mathbf{n^{(k)}},p,d+1,\mathbf{i^{(0)}}\right)\prod_{j=1}^k(1-p)^{(i_0-i_{0,j})(n_j-i_{0,j})}\\
&\qquad\qquad\cdot\left(1+\left((1-p)^{i_{0,j}-i_0}-1\right)(1-f(n,p,d,4npi_0))^{g_k\left(\mathbf{n^{(k)}}-\mathbf{i^{(0)}},j,l,p,d,0,\mathbf{n^{(k)}}-\mathbf{i^{(0)}}-\mathbf{1^{(l)}}\right)}\right)^{n_j-i_{0,j}-\mathbbm{1}_l(j)}\\
&=(1-p)^{i_0-i_{0,l}}h_k\left(\mathbf{n^{(k)}},p,d,4np\mathbf{i^{(0)}}\right)\\
&\quad\cdot\prod_{j=1}^k\left((1-p)^{i_0-i_{0,j}}+\left(1-(1-p)^{i_0-i_{0,j}}\right)(1-f(n,p,d,4npi_0))^{g_k\left(\mathbf{n^{(k)}}-\mathbf{i^{(0)}},j,l,p,d,0,\mathbf{n^{(k)}}-\mathbf{i^{(0)}}-\mathbf{1^{(l)}}\right)}\right)^{n_j-i_{0,j}-\mathbbm{1}_l(j)}.
\end{align*}
We note that $g_k\left(\mathbf{n^{(k)}}-\mathbf{i^{(0)}},j,l,p,d,0,\mathbf{n^{(k)}}-\mathbf{i^{(0)}}-\mathbf{1^{(l)}}\right)<n^{d-1}$ and so, using Lemma \ref{lem1}, we thus have
\begin{align*}
&\quad D_k\left(\mathbf{n^{(k)}},l,p,d+1,\mathbf{i^{(0)}}\right)\\
&<h_k\left(\mathbf{n^{(k)}},p,d+1,\mathbf{i^{(0)}}\right)\prod_{j=1}^k\left(1-f(n,p,d+1,i_0)\right)^{(i_0-i_{0,j})g_k\left(\mathbf{n^{(k)}}-\mathbf{i^{(0)}},j,l,p,d,0,\mathbf{n^{(k)}}-\mathbf{i^{(0)}}-\mathbf{1^{(l)}}\right)(n_j-i_{0,j}-\mathbbm{1}_l(j))}\\
&=h_k\left(\mathbf{n^{(k)}},p,d+1,\mathbf{i^{(0)}}\right)\left(1-f(n,p,d+1,i_0)\right)^{i_{0,l}\sum_{\substack{j=1\\j\neq l}}^k(n_j-i_{0,j})g_k\left(\mathbf{n^{(k)}}-\mathbf{i^{(0)}},j,l,p,d,0,\mathbf{n^{(k)}}-\mathbf{i^{(0)}}-\mathbf{1^{(l)}}\right)}\\
&\quad\cdot\prod_{\substack{q=1\\q\neq l}}^k\left(1-f(n,p,d+1,i_0)\right)^{i_{0,q}\left(1+\sum_{\substack{j=1\\j\neq q}}^k(n_j-i_{0,j}-\mathbbm{1}_l(j))g_k\left(\mathbf{n^{(k)}}-\mathbf{i^{(0)}},j,l,p,d,0,\mathbf{n^{(k)}}-\mathbf{i^{(0)}}-\mathbf{1^{(l)}}\right)\right)}.
\end{align*}
We can deduce that
\begin{equation*}
\sum_{\substack{j=1\\j\neq l}}^k(n_j-i_{0,j})g_k\left(\mathbf{n^{(k)}}-\mathbf{i^{(0)}},j,l,p,d,0,\mathbf{n^{(k)}}-\mathbf{i^{(0)}}-\mathbf{1^{(l)}}\right)\geq g_k\left(\mathbf{n^{(k)}},l,l,p,d,0,\mathbf{i^{(0)}}\right)
\end{equation*}
and
\begin{equation*}
1+\sum_{\substack{j=1\\j\neq q}}^k(n_j-i_{0,j}-\mathbbm{1}_l(j))g_k\left(\mathbf{n^{(k)}}-\mathbf{i^{(0)}},j,l,p,d,0,\mathbf{n^{(k)}}-\mathbf{i^{(0)}}-\mathbf{1^{(l)}}\right)\geq g_k\left(\mathbf{n^{(k)}},q,l,p,d,0,\mathbf{i^{(0)}}\right)
\end{equation*}
for all $1\leq q\leq k$ with $q\neq l$ and so we have \eqref{eqn33}.
\end{case}
\begin{case}{$1<\frac{n_j}{i_{0,j}+4n_jpi_0}$ for some $1\leq j\leq k$ with $j\neq l$, and/or $1<\frac{n_l-1}{i_{0,l}+4n_jpi_0}$}.
\normalfont
\newline
\newline
We proceed exactly as in Case $1$, except that for every $j\neq l$ with $1<\frac{n_j}{i_{0,j}+4n_jpi_0}$, we replace the summation $\sum_{i_{1,j}}^{n_j-i_{0,j}}$ with $\sum_{i_{1,j}}^{4n_ji_0}$ and multiply all the expressions following $D_k\left(\mathbf{n^{(k)}},l,p,d+1,\mathbf{i^{(0)}}\right)<$ by $\left(1-\frac{4}{5}\left(\frac{e}{3}\right)^{4n_jpi_0}\right)^{-1}$. As well, if it's the case that $1<\frac{n_l-1}{i_{0,l}+4n_lpi_0}$, then we  replace the summation $\sum_{i_{1,l}}^{n_l-i_{0,l}-1}$ with $\sum_{i_{1,l}}^{4n_li_0}$ and multiply all the expressions following $D_k\left(\mathbf{n^{(k)}},l,p,d+1,\mathbf{i^{(0)}}\right)<$ by $\left(1-\frac{4}{5}\left(\frac{e}{3}\right)^{4n_lpi_0}\right)^{-1}$.
\end{case}
\begin{case}{$1<\frac{n_j}{i_{0,j}+4n_jpi_0+4n_jp(4np)i_0+(4n_jp)(4np)^2i_0+\ldots+(4n_jp)(4np)^{d'}i_0}$ for all $1\leq j\leq k$ with $j\neq l$, and $1<\frac{n_l-1}{i_{0,l}+4n_lpi_0+4n_lp(4np)i_0+(4n_lp)(4np)^2i_0+\ldots+(4n_lp)(4np)^{d'}i_0}$}.
\normalfont
\newline
\newline
We proceed exactly as in Case $1$, except that for every $j\neq l$, we replace the summation $\sum_{i_{1,j}}^{n_j-i_{0,j}}$ with $\sum_{i_{1,j}}^{4n_ji_0}$ and multiply all the expressions following $D_k\left(\mathbf{n^{(k)}},l,p,d+1,\mathbf{i^{(0)}}\right)<$ by $\left(1-\frac{4}{5}\left(\frac{e}{3}\right)^{4n_jpi_0}\right)^{-1}$. As well, we  replace the summation $\sum_{i_{1,l}}^{n_l-i_{0,l}-1}$ with $\sum_{i_{1,l}}^{4n_li_0}$ and multiply all the expressions following $D_k\left(\mathbf{n^{(k)}},l,p,d+1,\mathbf{i^{(0)}}\right)<$ by $\left(1-\frac{4}{5}\left(\frac{e}{3}\right)^{4n_lpi_0}\right)^{-1}$. As well, we replace $g_k\left(\mathbf{n^{(k)}}-\mathbf{i^{(0)}},j,l,p,d,0,\mathbf{n^{(k)}}-\mathbf{i^{(0)}}-\mathbf{1^{(l)}}\right)$ with $g_k\left(\mathbf{n^{(k)}}-\mathbf{i^{(0)}},j,l,p,d,d',4pi_0\cdot\mathbf{n^{(k)}}\right)$ and use
\begin{equation*}
\sum_{\substack{j=1\\j\neq l}}^k(n_j-i_{0,j})g_k\left(\mathbf{n^{(k)}}-\mathbf{i^{(0)}},j,l,p,d,d',4pi_0\cdot\mathbf{n^{(k)}}\right)\geq g_k\left(\mathbf{n^{(k)}},l,l,p,d+1,d'+1,\mathbf{i^{(0)}}\right)
\end{equation*}
and
\begin{align*}
1+\sum_{\substack{j=1\\j\neq q}}^k(n_j-i_{0,j}-\mathbbm{1}_l(j))g_k\left(\mathbf{n^{(k)}}-\mathbf{i^{(0)}},j,l,p,d,d',4pi_0\cdot\mathbf{n^{(k)}}\right)\geq g_k\left(\mathbf{n^{(k)}},q,l,p,d+1,d'+1,\mathbf{i^{(0)}}\right)
\end{align*}
for all $1\leq q\leq k$ with $q\neq l$.
\end{case}
By \eqref{eqn33}, we have
\begin{align*}
\frac{\sum_{b\in B}\deg b}{s^t}&<\sum_{l=1}^k\binom{n_l}{ 2}h_k\left(\mathbf{n^{(k)}},p,d,\mathbf{1^{(l)}}\right)\left(1-f(n,p,d,1)\right)^{g_k\left(\mathbf{n^{(k)}},l,l,p,d,d',\mathbf{1^{(j)}}\right)}\\
&\quad+\sum_{1\leq j<l\leq k}n_jn_lh_k\left(\mathbf{n^{(k)}},p,d,\mathbf{1^{(j)}}\right)\left(1-f(n,p,d,1)\right)^{g_k\left(\mathbf{n^{(k)}},j,l,p,d,d',\mathbf{1^{(j)}}\right)}.
\end{align*}
Hence, by the simple sieve, we have
\begin{align*}
P(G(\mathbf{n^{(k)}},p),d)&>1-\sum_{l=1}^k\binom{n_l}{ 2}h_k\left(\mathbf{n^{(k)}},p,d,\mathbf{1^{(l)}}\right)\left(1-f(n,p,d,1)\right)^{g_k\left(\mathbf{n^{(k)}},l,l,p,d,d',\mathbf{1^{(j)}}\right)}\\
&\qquad-\sum_{1\leq j<l\leq k}n_jn_lh_k\left(\mathbf{n^{(k)}},p,d,\mathbf{1^{(j)}}\right)\left(1-f(n,p,d,1)\right)^{g_k\left(\mathbf{n^{(k)}},j,l,p,d,d',\mathbf{1^{(j)}}\right)}.
\end{align*}
We now calculate $n(b_1,b_2)$ to get an upper bound for $P\left(G\left(\mathbf{n^{(k)}},p\right),d\right)$ using the Tur\'an sieve. If the two pairs of vertices $b_1$ and $b_2$ are the same, then we just have $n(b_1,b_2)=\deg b$. If $b_1$ and $b_2$ have exactly one vertex in common, then we can see that $n(b_1,b_2)=C_k\left(\mathbf{n^{(k)}},l,r,s,d+1,\mathbf{i^{(0)}}\right)$ where $i_0=2$, and use \eqref{upperdegb}. Hence the only question is when the two pairs of vertices are disjoint.
\newline
\newline
As in our calculations for $\deg b$, to help calculate $n(b_1,b_2)$ in this case, we will calculate a generalised notion of $n(b_1,b_2)$ as follows. Let $1\leq l_1\leq k$ and $1\leq l_2\leq k$. Also, let $0\leq i_{0,j},i_{0,j}'\leq n_j$ with $i_{0,j}+i_{0,j}'\leq n_j$ for all $j\neq l_1,l_2$. Also, if $l_1\neq l_2$, pick $0\leq i_{0,l_1},i_{0,l_1}'\leq n_{l_1}-1$ with $i_{0,l_1}+i_{0,l_1}'\leq n_{l_1}-1$ and $0\leq i_{0,l_2},i_{0,l_2}'\leq n_{l_2}-1$ with $i_{0,l_2}+i_{0,l_2}'\leq n_{l_2}-1$. Otherwise, pick $0\leq i_{0,l_1},i_{0,l_1}'\leq n_{l_1}-2$ with $i_{0,l_1}+i_{0,l_1}'\leq n_{l_1}-2$. Let $i_0=i_{0,1}+i_{0,2}+\ldots+i_{0,k}$ and $i_0'=i_{0,1}'+i_{0,2}'+\ldots+i_{0,k}'$.
\newline
\newline
First, let $l_1=l_2=l$. Pick two disjoint specific sets of $i_{0,j}$ and $i_{0,j}'$ vertices out of the labeled vertices in the partite set consisting of $n_j$ vertices, as well as two other vertices, say $v_1$ and $v_2$, in the partite set consisting of $n_l$ vertices. We will let $C_k'\left(\mathbf{n^{(k)}},l,r,s,d,\mathbf{i^{(0)}},\mathbf{i^{(0)}}'\right)$ denote the number of graphs in $A$ such that there is no path from any of the $i_0$ vertices to vertex $v_1$ that consists of at most $d$ edges and no path from any of the $i_0'$ vertices to vertex $v_2$ that consists of at most $d$ edges. We can derive the recursive formula
\begin{align}
&\quad C_k'\left(\mathbf{n^{(k)}},l,r,s,d+1,\mathbf{i^{(0)}},\mathbf{i^{(0)}}'\right)\nonumber\\
&<\sum_{\mathbf{i^{(1)}}\in[n_j-i_{0,j}-i_{0,j}',l,l]_{1\leq j\leq k}}\sum_{\mathbf{i^{(1)}}'\in[n_j-i_{0,j}-i_{0,j}'-i_{1,j},l,l]_{1\leq j\leq k}}(s-r)^{i_{0,l}+i_{0,l}'-i_0-i_0'}\nonumber\\
&\quad\cdot\prod_{j=1}^k\binom{n_j-i_{0,j}-i_{0,j}'-2\cdot\mathbbm{1}_l(j)}{i_{1,j}}\left(s^{i_0-i_{0,j}}-(s-r)^{i_0-i_{0,j}}\right)^{i_{1,j}}(s-r)^{(i_0-i_{0,j})(n_j-i_{0,j}-i_{0,j}'-i_{1,j})}\nonumber\\
&\qquad\qquad\cdot\binom{n_j-i_{0,j}-i_{0,j}'-i_{1,j}-2\cdot\mathbbm{1}_l(j)}{i_{1,j}'}\left(s^{i_0'-i_{0,j}'}-(s-r)^{i_0'-i_{0,j}'}\right)^{i_{1,j}}(s-r)^{(i_0'-i_{0,j}')(n_j-i_{0,j}-i_{0,j}'-i_{1,j}-i_{1,j}')}\nonumber\\
&\qquad\qquad\cdot s^{t'+i_0+i_0'-i_{0,l}-i_{0,l}'}C_k'\left(\mathbf{n^{(k)}}-\mathbf{i^{(0)}}-\mathbf{i^{(0)}}',l,r,s,d,\mathbf{i^{(1)}},\mathbf{i^{(1)}}'\right)\label{eqn34}
\end{align}
valid so long as at least two of the $i_{0,j}$ values are nonzero and at least two of the $i_{0,j}'$ values are nonzero where $t'$ is the sum of the number of potential edges among the $i_0$ and $i_0'$ vertices and the number of potential edges with one vertex among the $i_0'$ vertices and the other vertex among the $i_{1,j}$ vertices for all $1\leq j\leq k$. If, however, only one of the $i_{0,j}$ is nonzero, say $i_{0,j}$, then the factor $\left(s^{i_0-i_{0,j}}-(s-r)^{i_0-i_{0,j}}\right)^{i_{1,j}}$ is replaced by $1$ for the respective $j$ value. The same holds if only of the $i_{0,j}'$ is nonzero. Everything else is left unchanged. As well,
\begin{equation*}
C_k'\left(\mathbf{n^{(k)}},l,r,s,d,\mathbf{i^{(0)}},\mathbf{i^{(0)}}'\right)=(s-r)^{i_0+i_0'-i_{0,l}-i_{0,l}'}s^{t-i_0-i_0'+i_{0,l}+i_{0,l}'}.
\end{equation*}
Let $D_k'\left(\mathbf{n^{(k)}},l,p,d,\mathbf{i^{(0)}},\mathbf{i^{(0)}}'\right):=\frac{C_k'\left(\mathbf{n^{(k)}},l,r,s,d,\mathbf{i^{(0)}},\mathbf{i^{(0)}}'\right)}{s^t}$ so that $D_k'\left(\mathbf{n^{(k)}},l,r,s,d,\mathbf{i^{(0)}},\mathbf{i^{(0)}}'\right)$ is the probability that the edge distance between $v$ and any of the $i_0$ vertices is greater than $d$ and that the edge distance between $v'$ and any of the $i_0'$ vertices is greater than $d$. We will prove that for all $(\mathbf{i^{(0)}})\in[n_j,l]_{1\leq j\leq k}$ we have
\begin{equation}
D_k'\left(\mathbf{n^{(k)}},l,p,d,\mathbf{i^{(0)}},\mathbf{i^{(0)}}'\right)\leq D_k\left(\mathbf{n^{(k)}}-\mathbf{1^{(l)}},l,p,d,\mathbf{i^{(0)}}+\mathbf{i^{(0)}}'\right).\label{eqn35}
\end{equation}
For $d=1$, we have
\begin{equation*}
D_k'\left(\mathbf{n^{(k)}},l,p,d,\mathbf{i^{(0)}},\mathbf{i^{(0)}}'\right)=(1-p)^{i_0+i_0'-i_{0,l}-i_{0,l}'}=D_k'\left(\mathbf{n^{(k)}}-\mathbf{1^{(l)}},l,p,1,\mathbf{i^{(0)}}+\mathbf{i^{(0)}}'\right)
\end{equation*}
so \eqref{eqn35} holds for $d=1$. Suppose for some $d\geq 1$ \eqref{eqn35} holds for all $(\mathbf{i^{(0)}})\in[n_j,l]_{1\leq j\leq k}$ and $0<p<1$. Assume that two of the $i_{0,j}$ are nonzero and two of the $i_{0,j}'$ are nonzero (if the case is otherwise, then we can proceed similarly). First we have
\begin{align*}
&\quad D_k'\left(\mathbf{n^{(k)}},l,p,d+1,\mathbf{i^{(0)}},\mathbf{i^{(0)}}'\right)\\
&\leq (1-p)^{i_{0,l}+i_{0,l}'-i_0-i_0'}\\
&\quad\cdot\sum_{\mathbf{i^{(1)}}\in[n_j-i_{0,j}-i_{0,j}',l,l]_{1\leq j\leq k}}\sum_{\mathbf{i^{(1)}}'\in[n_j-i_{0,j}-i_{0,j}'-i_{1,j},l,l]_{1\leq j\leq k}}\\
&\quad\cdot\prod_{j=1}^k\binom{n_j-i_{0,j}-i_{0,j}'-2\cdot\mathbbm{1}_l(j)}{i_{1,j}}\left((1-p)^{i_{0,j}-i_0}-1\right)^{i_{1,j}}(1-p)^{(i_0-i_{0,j})(n_j-i_{0,j}-i_{0,j}')}\\
&\qquad\qquad\cdot\binom{n_j-i_{0,j}-i_{0,j}'-i_{1,j}-2\cdot\mathbbm{1}_l(j)}{i_{1,j}'}\left((1-p)^{i_{0,j}'-i_0'}-1\right)^{i_{1,j}'}(1-p)^{(i_0'-i_{0,j}')(n_j-i_{0,j}-i_{0,j}'-i_{1,j})}\\
&\qquad\qquad\cdot D_k\left(\mathbf{n^{(k)}}-\mathbf{i^{(0)}}-\mathbf{i^{(0)}}'-\mathbf{1}^{(l)},l,p,d,\mathbf{i^{(1)}}+\mathbf{i^{(1)}}'\right).
\end{align*} 
Writing $\mathbf{v}=\mathbf{i^{(1)}}+\mathbf{i^{(1)}}'=(v_1,v_2,\ldots,v_k)$, we have
\begin{align*}
&\quad D_k'\left(\mathbf{n^{(k)}},l,p,d+1,\mathbf{i^{(0)}},\mathbf{i^{(0)}}'\right)\\
&\leq\sum_{\mathbf{v}\in[n_j-i_{0,j}-i_{0,j}',l,l]_{1\leq j\leq k}}D_k\left(\mathbf{n^{(k)}}-\mathbf{i^{(0)}}-\mathbf{i^{(0)}}'-\mathbf{1}^{(l)},l,p,d,\mathbf{v}\right)(1-p)^{i_{0,l}+i_{0,l}'-i_0-i_0'}\\
&\quad\cdot\prod_{j=1}^k(1-p)^{(i_0+i_0'-i_{0,j}-i_{0,j}')(n_j-i_{0,j}-i_{0,j}')}\binom{n_j-i_{0,j}-i_{0,j}'-2\cdot\mathbbm{1}_l(j)}{v_j}\left((1-p)^{i_{0,j}'-i_0'}-1\right)^{v_j}\\
&\qquad\qquad\cdot\sum_{i_{1,j}=0}^{v_j}{v_j}{i_{1,j}}\left((1-p)^{i_{0,j}-i_0}-1\right)^{i_{1,j}}(1-p)^{-i_{1,j}(i_0'-i_{0,j}')}\left((1-p)^{i_{0,j}'-i_0'}-1\right)^{-i_{1,j}}.
\end{align*}
Thus
\begin{align*}
&\quad D_k'\left(\mathbf{n^{(k)}},l,p,d+1,\mathbf{i^{(0)}},\mathbf{i^{(0)}}'\right)\\
&\leq\sum_{\mathbf{v}\in[n_j-i_{0,j}-i_{0,j}',l,l]_{1\leq j\leq k}}D_k\left(\mathbf{n^{(k)}}-\mathbf{i^{(0)}}-\mathbf{i^{(0)}}'-\mathbf{1}^{(l)},l,p,d,\mathbf{v}\right)(1-p)^{i_{0,l}+i_{0,l}'-i_0-i_0'}\\
&\quad\cdot\prod_{j=1}^k(1-p)^{(i_0+i_0'-i_{0,j}-i_{0,j}')(n_j-i_{0,j}-i_{0,j}')}\binom{n_j-i_{0,j}-i_{0,j}'-2\cdot\mathbbm{1}_l(j)}{v_j}\left((1-p)^{i_{0,j}'-i_0'}-1\right)^{v_j}\\
&\qquad\qquad\cdot\left(1+\frac{(1-p)^{i_{0,j}-i_0}-1}{1-(1-p)^{i_0'-i_{0,j}'}}\right)^{v_j}.
\end{align*}
Thus
\begin{align*}
&\quad D_k'\left(\mathbf{n^{(k)}},l,p,d+1,\mathbf{i^{(0)}},\mathbf{i^{(0)}}'\right)\\
&\leq\sum_{\mathbf{v}\in[n_j-i_{0,j}-i_{0,j}',l,l]_{1\leq j\leq k}}D_k\left(\mathbf{n^{(k)}}-\mathbf{i^{(0)}}-\mathbf{i^{(0)}}'-\mathbf{1}^{(l)},l,p,d,\mathbf{v}\right)\\
&\quad\cdot\prod_{j=1}^k(1-p)^{(i_0+i_0'-i_{0,j}-i_{0,j}')(n_j-i_{0,j}-i_{0,j}'-\mathbbm{1}_l(j))}\binom{n_j-i_{0,j}-i_{0,j}'-2\cdot\mathbbm{1}_l(j)}{v_j}\left((1-p)^{i_{0,j}+i_{0,j}'-i_0-i_0'}-1\right)^{v_j}\\
&=D_k\left(\mathbf{n^{(k)}}-\mathbf{1}^{(l)},l,p,d+1,\mathbf{i^{(0)}}+\mathbf{i^{(0)}}'\right).
\end{align*}
Thus we have \eqref{eqn35}. Now assume that $l_1\neq l_2$. Pick two disjoint specific sets of $i_{0,j}$ and $i_{0,j}'$ vertices out of the labeled vertices in the partite set consisting of $n_j$ vertices, as well as two other vertices, say $v_1$ and $v_2$, the first being in the partite set consisting of $n_{l_1}$ vertices and the second being in the partite set consisting of $n_{l_2}$ vertices. We will let $C_k''\left(\mathbf{n^{(k)}},l_1,l_2,r,s,d,\mathbf{i^{(0)}},\mathbf{i^{(0)}}',\mathbf{i^{(0)}}''\right)$ denote the number of graphs in $A$ such that there is no path from any of the $i_0$ vertices to vertex $v_1$ that avoids vertex $v_2$ and consists of at most $d$ edges, no path from any of the $i_0'$ vertices to vertex $v_2$ that avoids vertex $(v_1)$ and consists of at most $d$ edges, and no paths from any of the $i_0''$ vertices to either vertex $v_1$ or $v_2$ that consists of at most $d$ edges. We can derive the recursive formula
\begin{align}
&\quad C_k''\left(\mathbf{n^{(k)}},l_1,l_2,r,s,d+1,\mathbf{i^{(0)}},\mathbf{i^{(0)}}',\mathbf{i^{(0)}}''\right)\nonumber\\
&<(s-r)^{i_{0,l_1}'-i_0'+i_{0,l_2}-i_0}s^{i_0'-i_{0,l_1}'+i_0-i_{0,l_2}}\nonumber\\
&\quad\cdot\prod_{j=1}^k\sum_{0\leq i_{1,j}+i_{1,j}'+i_{1,j}''\leq n_j-i_{0,j}'''-\mathbbm{1}_{l_1}(j)-\mathbbm{1}_{l_2}(j)}\binom{n_j-i_{0,j}'''-\mathbbm{1}_{l_1}(j)-\mathbbm{1}_{l_2}(j)}{i_{1,j},i_{1,j}',i_{1,j}'',n_j-i_{0,j}'''-\mathbbm{1}_{l_1}(j)-\mathbbm{1}_{l_2}(j)-i_{1,j}'''}\nonumber\\
&\qquad\cdot(s-r)^{\left(i_0'''-i_{0,j}'''\right)\left(n_j-i_{0,j}'''-i_{1,j}'''\right)}\nonumber\\
&\qquad\cdot(s-r)^{\left(i_0'+i_0''-i_{0,j}'-i_{0,j}''\right)i_{1,j}+\left(i_0+i_0''-i_{0,j}-i_{0,j}''\right)i_{1,j}'}\nonumber\\
&\qquad\cdot\left(s^{i_0-i_{0,j}}-(s-r)^{i_0-i_{0,j}}\right)^{i_{1,j}}\left(s^{i_0'-i_{0,j}'}-(s-r)^{i_0'-i_{0,j}'}\right)^{i_{1,j}'}\nonumber\\
&\qquad\cdot\left(s^{i_0''-i_{0,j}''}-(s-r)^{i_0''-i_{0,j}''}+(s-r)^{i_0''-i_{0,j}''}\left(s^{i_0-i_{0,j}}-(s-r)^{i_0-i_{0,j}}\right)\left(s^{i_0'-i_{0,j}'}-(s-r)^{i_0'-i_{0,j}'}\right)\right)^{i_{1,j}''}\nonumber\\
&\qquad\cdot s^{\frac{i_{0,j}'''\left(i_0'''-i_{0,j}'''\right)}{2}}C_k''\left(\mathbf{n^{(k)}}-\mathbf{i^{(0)}}''',l_1,l_2,p,d,\mathbf{i^{(1)}},\mathbf{i^{(1)}}',\mathbf{i^{(1)}}''\right)\label{eqn36}
\end{align}
valid so long as no expressions in \eqref{eqn36} do not evaluate to $0^0$. If, however, an expression does evaluate to $0^0$, it is replaced by $1$. Everything else is left unchanged. As well,
\begin{equation*}
C_k''\left(\mathbf{n^{(k)}},l_1,l_2,r,s,1,\mathbf{i^{(0)}},\mathbf{i^{(0)}}',\mathbf{i^{(0)}}''\right)=(s-r)^{i_0-i_{0,l_1}+i_0'-i_{0,l_2}'+2i_0''-i_{0,l_1}''-i_{0,l_2}''}s^{t-i_0+i_{0,l_1}-i_0'+i_{0,l_2}'-2i_0''+i_{0,l_1}''+i_{0,l_2}''}.
\end{equation*}
Let $D_k''\left(\mathbf{n^{(k)}},l_1,l_2,p,d,\mathbf{i^{(0)}},\mathbf{i^{(0)}}',\mathbf{i^{(0)}}''\right):=\frac{C_k''\left(\mathbf{n^{(k)}},l_1,l_2,r,s,d,\mathbf{i^{(0)}},\mathbf{i^{(0)}}',\mathbf{i^{(0)}}''\right)}{s^t}$ so that
\newline
$D_k''\left(\mathbf{n^{(k)}},l_1,l_2,p,d,\mathbf{i^{(0)}},\mathbf{i^{(0)}}',\mathbf{i^{(0)}}''\right)$ is the probability that the edge distance between $v_1$ and any of the $i_0$ vertices is greater than $d$, that the edge distance between $v'$, any of the $i_0'$ vertices is greater than $d$, that the edge distance between $v_1$ and any of the $i_0''$ vertices is greater than $d$, and that the edge distance between $v'$, any of the $i_0''$ vertices is greater than $d$. We prove that
\begin{align}
&\quad D_k''\left(\mathbf{n^{(k)}},l_1,l_2,p,d,\mathbf{i^{(0)}},\mathbf{i^{(0)}}',\mathbf{i^{(0)}}''\right)\nonumber\\
&\leq h_k\left(\mathbf{n^{(k)}},p,d,\mathbf{i^{(0)}}'''\right)\left(1-f\left(n,p,d,i_0'''\right)\right)^{\sum_{j=1}^k\left(i_{0,j}+i_{0,j}''\right)g_k\left(\mathbf{n^{(k)}}-\mathbf{1^{(l_2)}},j,l_1,p,d,d',\mathbf{i^{(0)}}'''\right)}\nonumber\\
&\quad\cdot\left(1-f\left(n,p,d,i_0'''\right)\right)^{\sum_{j=1}^k\left(i_{0,j}'+i_{0,j}''\right)g_k\left(\mathbf{n^{(k)}}-\mathbf{1^{(l_1)}},j,l_2,p,d,d',\mathbf{i^{(0)}}'''\right)}\label{eqn40}
\end{align}
assuming $1<\frac{n_j-\mathbbm{1}_{l_1}(j)-\mathbbm{1}_{l_2}(j)}{i_{0,j}+i_{0,j}'+i_{0,j}''+\left(4n_jp+4n_jp(4np)+(4n_jp)(4np)^2+\ldots+(4n_jp)(4np)^{d'-1}\right)\left(i_0+i_0'+i_0''\right)}$ for all $1\leq j\leq k$ where $d'\geq 0$ by induction on $d$. First, two lemmas.
\begin{lemma}\label{lem6}
Let $0<p<1$ and $0<C_1<1$. Also, let $y_1,y_2,y_3,t_1,t_2\geq 1$, $M\geq t_1+t_2$, and $N\geq y_1+y_2+y_3$. Suppose that
\begin{equation*}
0<C_2\leq\frac{1-(1-x)^M}{Mx}
\end{equation*}
and
\begin{equation*}
0<C_3\leq\frac{1-(1-p)^N}{Np}.    
\end{equation*}
Then
\begin{align*}
&\quad(1-p)^{y_1+y_2+y_3}+\left((1-p)^{y_2+y_3}-(1-p)^{y_1+y_2+y_3}\right)(1-C_1x)^{t_1}\\
&\quad+\left((1-p)^{y_1+y_3}-(1-p)^{y_1+y_2+y_3}\right)(1-C_1x)^{t_2}\\
&\quad+\left(1-(1-p)^{y_2+y_3}-(1-p)^{y_1+y_3}+(1-p)^{y_1+y_2+y_3}\right)(1-C_1x)^{t_1+t_2}\\
&<\left(1-C_1C_2C_3xp\right)^{(y_1+y_3)t_1+(y_2+y_3)t_2}.
\end{align*}
\end{lemma}
\begin{proof}
From Lemma \ref{lem3} we have
\begin{equation}\label{eqn37}
C_2\leq\frac{1-(1-x)^M}{Mx}\leq\frac{1-(1-x)^{t_1+t_2}}{x(t_1+t_2)}<\frac{1-(1-x)^{t_1}}{xt_1},\frac{1-(1-x)^{t_2}}{xt_2}
\end{equation}
and
\begin{equation}\label{eqn38}
C_3\leq\frac{1-(1-p)^N}{Np}<\frac{1-(1-p)^{y_1+y_3}}{p(y_1+y_3)},\frac{1-(1-p)^{y_2+y_3}}{p(y_2+y_3)}.
\end{equation}
We have the following:
\begin{align*}
&\quad(1-p)^{y_1+y_2+y_3}+\left((1-p)^{y_2+y_3}-(1-p)^{y_1+y_2+y_3}\right)(1-C_1x)^{t_1}\\
&\quad+\left((1-p)^{y_1+y_3}-(1-p)^{y_1+y_2+y_3}\right)(1-C_1x)^{t_2}\\
&\quad+\left(1-(1-p)^{y_2+y_3}-(1-p)^{y_1+y_3}+(1-p)^{y_1+y_2+y_3}\right)(1-C_1x)^{t_1+t_2}\\
&=1-\left((1-p)^{y_2+y_3}-(1-p)^{y_1+y_2+y_3}\right)\left(1-(1-C_1x)^{t_1}\right)\\
&\quad-\left((1-p)^{y_1+y_3}-(1-p)^{y_1+y_2+y_3}\right)\left(1-(1-C_1x)^{t_2}\right)\\
&\quad-\left(1-(1-p)^{y_2+y_3}-(1-p)^{y_1+y_3}+(1-p)^{y_1+y_2+y_3}\right)\left(1-(1-C_1x)^{t_1+t_2}\right)\\
&<1-C_1\left((1-p)^{y_2+y_3}-(1-p)^{y_1+y_2+y_3}\right)\left(1-(1-x)^{t_1}\right)\\
&\quad-C_1\left((1-p)^{y_1+y_3}-(1-p)^{y_1+y_2+y_3}\right)\left(1-(1-x)^{t_2}\right)\\
&\quad-C_1\left(1-(1-p)^{y_2+y_3}-(1-p)^{y_1+y_3}+(1-p)^{y_1+y_2+y_3}\right)\left(1-(1-x)^{t_1+t_2}\right)\\
&<1-C_1C_2xt_1\left((1-p)^{y_2+y_3}-(1-p)^{y_1+y_2+y_3}\right)\\
&\quad-C_1C_2xt_2\left((1-p)^{y_1+y_3}-(1-p)^{y_1+y_2+y_3}\right)\\
&\quad-C_1C_2x(t_1+t_2)\left(1-(1-p)^{y_2+y_3}-(1-p)^{y_1+y_3}+(1-p)^{y_1+y_2+y_3}\right)\left(1-(1-x)^{t_1+t_2}\right)\\
&=1-C_1C_2xt_1\left(1-(1-p)^{y_1+y_3}\right)-C_1C_2xt_2\left(1-(1-p)^{y_2+y_3}\right)\\
&<1-C_1C_2C_3pxt_1(y_1+y_3)-C_1C_2C_3pxt_2(y_2+y_3)\\
&<\left(1-C_1C_2C_3xp\right)^{(y_1+y_3)t_1+(y_2+y_3)t_2}
\end{align*}
with the first inequality following from Lemma \ref{lem3}, the second inequality following from \eqref{eqn37}, and the third inequality following from \eqref{eqn38}.
\end{proof}
\begin{lemma}\label{lem7}
Suppose $f_1,f_2,f_3,f_4:\mathbb{N}\times\mathbb{N}$ all satisfy $f_q(n,i+1)\leq f_q(n,i)\leq 1$ for all $i,n\in\mathbb{N}$ and $1\leq q\leq 4$. Let $r_1,r_2,r_3\in\mathbb{R}$, $r_1,r_2,r_3>0$ satisfy $\frac{4(r_1+r_2+r_3)}{r_1+r_2+r_3+1}<1$. Then for all $n\in\mathbb{N}$ and for all $\frac{4n(r_1+r_2+r_3)}{r_1+r_2+r_3+1}\leq t\leq n$ we have
\begin{align*}
&\quad\sum_{0\leq i_1+i_2+i_3\leq n}\binom{n}{i_1,i_2,i_3,n-i_1-i_2-i_3}r_1^{i_1}r_2^{i_2}r_3^{i_3}f_4\left(n,i_1+i_2+i_3\right)\\
&\qquad\qquad\qquad\cdot f_1(n,i_1+i_2+i_3)^{i_1}f_2(n,i_1+i_2+i_3)^{i_2}f_3(n,i_1+i_2+i_3)^{i_3}\\
&<\left(1-\frac{4}{5}\left(\frac{e}{3}\right)^t\right)^{-1}\sum_{0\leq i_1+i_2+i_3\leq\lfloor t\rfloor}\binom{n}{i_1,i_2,i_3,n-i_1-i_2-i_3}r_1^{i_1}r_2^{i_2}r_3^{i_3}f_4\left(n,i_1+i_2+i_3\right)\\
&\qquad\qquad\qquad\cdot f_1(n,i_1+i_2+i_3)^{i_1}f_2(n,i_1+i_2+i_3)^{i_2}f_3(n,i_1+i_2+i_3)^{i_3}.
\end{align*}
\end{lemma}
\begin{proof}
For all $0\leq i\leq n$ we have
\begin{align*}
&\quad\sum_{i_1+i_2+i_3=i}\binom{n}{i_1,i_2,i_3,i}r_1^{i_1}r_2^{i_2}r_3^{i_3}f_4(n,i)f_1\left(n,i\right)^{i_1}f_2\left(n,i\right)^{i_2}f_3\left(n,i\right)^{i_3}\\
&=\binom{n}{i}\left(r_1f_1(n,i)+r_2f_2(n,i)+r_3f_3(n,i)\right)^if_4(n,i).
\end{align*}
Noting that
\begin{equation*}
\frac{4n(r_1f_1(n,i)+r_2f_2(n,i)+r_3f_3(n,i))}{r_1f_1(n,i)+r_2f_2(n,i)+r_3f_3(n,i)+1}\leq\frac{4n(r_1+r_2+r_3)}{r_1+r_2+r_3+1}
\end{equation*}
the result follows from Lemma \ref{4nplemma}.
\end{proof}
We prove \eqref{eqn40} by induction on $d$. 
\begin{note}
For simplicity of notation, we define the following:
\begin{equation*}
\mathbf{i^{(0)}}''':=\mathbf{i^{(0)}}+\mathbf{i^{(0)}}'+\mathbf{i^{(0)}}''
\end{equation*}
\begin{equation*}
i_0''':=i_0+i_0'+i_0''
\end{equation*}
\begin{equation*}
i_{0,j}''':=i_{0,j}+i_{0,j}'+i_{0,j}''
\end{equation*}
\begin{equation*}
i_0^{(4)}:=i_0+i_0''
\end{equation*}
\begin{equation*}
i_{0,j}^{(4)}:=i_{0,j}+i_{0,j}''
\end{equation*}
\begin{equation*}
i_0^{(5)}:=i_0'+i_0''    
\end{equation*}
\begin{equation*}
i_{0,j}^{(5)}:=i_{0,j}'+i_{0,j}''
\end{equation*}
\begin{equation*}
\mathbf{i^{(1)}}''':=\mathbf{i^{(1)}}+\mathbf{i^{(1)}}'+\mathbf{i^{(1)}}'' 
i_1''':=i_1+i_1'+i_1''
\end{equation*}
\begin{equation*}
i_{1,j}''':=i_{1,j}+i_{1,j}'+i_{1,j}''.\\    
\end{equation*}
\end{note}
For $d\geq 2$, we have the following by \eqref{eqn36} with the second inequality following from Lemma \ref{lem6}:
\begin{align*}
&\quad D_k''\left(\mathbf{n^{(k)}},l_1,l_2,p,2,\mathbf{i^{(0)}},\mathbf{i^{(0)}}',\mathbf{i^{(0)}}''\right)\\
&<\prod_{j=1}^k\sum_{0\leq i_{1,j}+i_{1,j}'+i_{1,j}''\leq n_j-i_{0,j}'''-\mathbbm{1}_{l_1}(j)-\mathbbm{1}_{l_2}(j)}\left(1-(1-p)^{i_0-i_{0,j}}\right)^{i_{1,j}}(1-p)^{\left(i_0'+i_0''-i_{0,j}'-i_{0,j}''\right)i_{1,j}}\\
&\quad\cdot\left(1-(1-p)^{i_0'-i_{0,j}'}\right)^{i_{1,j}'}(1-p)^{\left(i_0+i_0''-i_{0,j}-i_{0,j}''\right)i_{1,j}'}\\
&\quad\cdot\left(1-(1-p)^{i_0''-i_{0,j}''}+(1-p)^{i_0''-i_{0,j}''}\left(1-(1-p)^{i_0-i_{0,j}}\right)\left(1-(1-p)^{i_0'-i_{0,j}'}\right)\right)^{i_{1,j}''}\\
&\quad\cdot\left((1-p)^{i_0'''-i_{0,j}'''}\right)^{n_j-i_{0,j}'''-i_{1,j}'''-\mathbbm{1}_{l_1}(j)-\mathbbm{1}_{l_2}(j)}\\
&\quad\cdot(1-p)^{i_0-i_{0,l_1}+i_0''-i_{0,l_1}''+i_0'-i_{0,l_2}'+i_0''-i_{0,l_2}''}\\
&\quad\cdot(1-p)^{i_1-i_{1,l_1}+i_1''-i_{1,l_1}''+i_1'-i_{1,l_2}'+i_1''-i_{1,l_2}''}\\
&=\prod_{\substack{j=1\\j\neq l_1,l_2}}^k\left((1-p)^2+p(1-p)^{i_0'+i_0''-i_{0,j}'-i_{0,j}''+1}+p(1-p)^{i_0+i_0''-i_{0,j}-i_{0,j}''+1}+p^2(1-p)^{i_0'''-i_{0,j}'''}\right)^{n_j-i_{0,j}'''}\\
&\quad\cdot(1-p)^{i_0-i_{0,l_1}+i_0''-i_{0,l_1}''}\left(1-p+p(1-p)^{i_0'+i_0''-i_{0,l_1}'-i_{0,l_1}''}\right)^{n_{l_1}-i_{0,l_1}'''-1}\\
&\quad\cdot(1-p)^{i_0'-i_{0,l_2}'+i_0''-i_{0,l_2}''}\left(1-p+p(1-p)^{i_0+i_0''-i_{0,l_2}-i_{0,l_2}''}\right)^{n_{l_2}-i_{0,l_2}'''-1}
\end{align*}
Thus
\begin{align*}
&\quad D_k''\left(\mathbf{n^{(k)}},l_1,l_2,p,2,\mathbf{i^{(0)}},\mathbf{i^{(0)}}',\mathbf{i^{(0)}}''\right)\\
&<\prod_{\substack{j=1\\j\neq l_1,l_2}}^k\left(1-f\left(n,p,2,i_0'''\right)\right)^{\left(i_0+i_0'+2i_0''-i_{0,j}-i_{0,j}'-2i_{0,j}''\right)\left(n_j-i_{0,j}'''\right)}\\
&\quad\cdot\left(1-f\left(n,p,2,i_0'''\right)\right)^{\left(i_0'+i_0''-i_{0,l_1}'-i_{0,l_1}''\right)\left(n_{l_1}-i_{0,l_1}'''-1\right)}\\
&\quad\cdot\left(1-f\left(n,p,2,i_0'''\right)\right)^{\left(i_0+i_0''-i_{0,l_2}-i_{0,l_2}''\right)\left(n_{l_2}-i_{0,l_2}'''-1\right)}
\end{align*}
so that
\begin{align*}
&\quad D_k''\left(\mathbf{n^{(k)}},l_1,l_2,p,2,\mathbf{i^{(0)}},\mathbf{i^{(0)}}',\mathbf{i^{(0)}}''\right)\\
&<\prod_{\substack{m=1\\m\neq l_1}}^k\left(1-f\left(n,p,2,i_0'''\right)\right)^{\left(i_{0,m}+i_{0.m}''\right)\left(n-n_{l_1}-n_m-i_0'''+i_{0,l_1}'''+i_{0,m}'''-1+\mathbbm{1}_{l_2}(m)\right)}\\
&\quad\cdot\left(1-f(n,p,2,i_0''')\right)^{\left(i_{0,l_1}+i_{0.l_1}''\right)\left(n-n_{l_1}-i_0'''+i_{0,l_1}'''-1\right)}\\
&\quad\cdot\prod_{\substack{m=1\\m\neq l_2}}^k\left(1-f\left(n,p,2,i_0'''\right)\right)^{\left(i_{0,m}'+i_{0.m}''\right)\left(n-n_{l_2}-n_m-i_0'''+i_{0,l_2}'''+i_{0,m}'''-1+\mathbbm{1}_{l_1}(m)\right)}\\
&\quad\cdot\left(1-f\left(n,p,2,i_0'''\right)\right)^{\left(i_{0,l_2}'+i_{0.l_2}''\right)\left(n-n_{l_2}-i_0'''+i_{0,l_2}'''-1\right)}.
\end{align*}
Suppose for some $d\geq 2$ \eqref{eqn40} holds for all $\mathbf{i^{(0)}}$, $\mathbf{i^{(0)}}'$, and $\mathbf{i^{(0)}}''$ in the stated ranges, and $0<p<1$. We will prove \eqref{eqn40} holds for $d+1$. We have
\begin{align*}
&\quad D_k''\left(\mathbf{n^{(k)}},l_1,l_2,p,2,\mathbf{i^{(0)}},\mathbf{i^{(0)}}',\mathbf{i^{(0)}}''\right)\\
&<\prod_{j=1}^k\sum_{0\leq i_{1,j}+i_{1,j}'+i_{1,j}''\leq n_j-i_{0,j}'''-\mathbbm{1}_{l_1}(j)-\mathbbm{1}_{l_2}(j)}\binom{n_j-i_{0,j}'''-\mathbbm{1}_{l_1}(j)-\mathbbm{1}_{l_2}(j)}{i_{1,j},i_{1,j}',i_{1,j}'',n_j-i_{0,j}'''-\mathbbm{1}_{l_1}(j)-\mathbbm{1}_{l_2}(j)-i_{1,j}'''}\\
&\quad\cdot(1-p)^{(i_0'''-i_{0,j}''')(n_j-i_{0,j}'''-\mathbbm{1}_{l_1}(j)-\mathbbm{1}_{l_2}(j))}\\
&\quad\cdot(1-p)^{i_0-i_{0,l_1}+i_0'-i_{0,l_2}'+2i_0''-i_{0,l_1}''-i_{0,l_2}''}\left((1-p)^{i_{0,j}-i_0}-1\right)^{i_{1,j}}\left((1-p)^{i_{0,j}'-i_0'}-1\right)^{i_{1,j}'}\\
&\quad\cdot\left((1-p)^{i_{0,j}'''-i_0'''}-(1-p)^{i_{0,j}-i_0}-(1-p)^{i_{0,j}'-i_0'}+1\right)^{i_{1,j}''}\\
&\quad\cdot D_k''\left(\mathbf{n^{(k)}}-\mathbf{i^{(0)}}''',l_1,l_2,p,d,\mathbf{i^{(1)}},\mathbf{i^{(1)}}',\mathbf{i^{(1)}}''\right).
\end{align*}
We divide into three cases.
\setcounter{case}{0}
\begin{case}{$\frac{n_j-\mathbbm{1}_j(l_1)-\mathbbm{1}_j(l_2)}{i_{0,j}'''+4n_jpi_0'''}\leq 1$ for all $1\leq j\leq k$.}
\normalfont
\newline
\newline
We have the following:
\begin{align*}
&\quad D_k''\left(\mathbf{n^{(k)}},l_1,l_2,p,d+1,\mathbf{i^{(0)}},\mathbf{i^{(0)}}',\mathbf{i^{(0)}}''\right)\\
&<\prod_{j=1}^k\sum_{0\leq i_{1,j}+i_{1,j}'+i_{1,j}''\leq n_j-i_{0,j}'''-\mathbbm{1}_{l_1}(j)-\mathbbm{1}_{l_2}(j)}h_k\left(\mathbf{n^{(k)}}-\mathbf{i^{(0)}}''',p,d,\mathbf{i^{(1)}}'''\right)\\
&\quad\cdot\binom{n_j-i_{0,j}'''-\mathbbm{1}_{l_1}(j)-\mathbbm{1}_{l_2}(j)}{i_{1,j},i_{1,j}',i_{1,j}'',n_j-i_{0,j}'''-\mathbbm{1}_{l_1}(j)-\mathbbm{1}_{l_2}(j)-i_{1,j}'''}\\
&\quad\cdot(1-p)^{\left(i_0'''-i_{0,j}'''\right)\left(n_j-i_{0,j}-i_{0,j}'-i_{0,j}''-\mathbbm{1}_{l_1}(j)-\mathbbm{1}_{l_2}(j)\right)}\\
&\quad\cdot(1-p)^{i_0-i_{0,l_1}+i_0'-i_{0,l_2}'+2i_0''-i_{0,l_1}''-i_{0,l_2}''}\left((1-p)^{i_{0,j}-i_0}-1\right)^{i_{1,j}}\left((1-p)^{i_{0,j}'-i_0'}-1\right)^{i_{1,j}'}\\
&\quad\cdot\left((1-p)^{i_{0,j}'''-i_0'''}-(1-p)^{i_{0,j}-i_0}-(1-p)^{i_{0,j}'-i_0'}+1\right)^{i_{1,j}''}\\
&\quad\cdot\left(1-f\left(n-i_0''',p,d,i_1'''\right)\right)^{\left(i_{1,j}+i_{1,j}''\right)g_k\left(\mathbf{n^{(k)}}-\mathbf{i^{(0)}}'''-\mathbf{1^{(l_2)}},j,l_1,p,d,0,\mathbf{i^{(1)}}'''\right)}\\
&\quad\cdot\left(1-f\left(n-i_0''',p,d,i_1'''\right)\right)^{\left(i_{1,j}'+i_{1,j}''\right)g_k\left(\mathbf{n^{(k)}}-\mathbf{i^{(0)}}'''-\mathbf{1^{(l_1)}},j,l_2,p,d,0,\mathbf{i^{(1)}}'''\right)}.
\end{align*}
We can deduce that
\begin{equation*}
h_k\left(\mathbf{n^{(k)}}-\mathbf{i^{(0)}}''',p,d,\mathbf{i^{(1)}}'''\right)\leq h_k\left(\mathbf{n^{(k)}}-\mathbf{i^{(0)}}''',p,d,4np\mathbf{i^{(0)}}'''\right)<h_k\left(\mathbf{n^{(k)}},p,d+1,\mathbf{i^{(0)}}'''\right)
\end{equation*}
for all $1\leq j\leq k$ and from Lemma \ref{lem3}, we can deduce that \begin{equation*}
f\left(n,p,d,4npi_0'''\right)<f\left(n-i_0''',p,d,i_1'''\right).
\end{equation*}
As well
\begin{equation*}
g_k\left(\mathbf{n^{(k)}}-\mathbf{i^{(0)}}'''-\mathbf{1^{(l_2)}},j,l_1,p,d,0,\mathbf{n^{(k)}}-\mathbf{i^{(0)}}'''-\mathbf{1^{(l_1)}}-\mathbf{1^{(l_2)}}\right)\leq g_k\left(\mathbf{n^{(k)}}-\mathbf{i^{(0)}}'''-\mathbf{1^{(l_2)}},j,l_1,p,d,0,\mathbf{i^{(1)}}'''\right)
\end{equation*}
and
\begin{align*}
g_k\left(\mathbf{n^{(k)}}-\mathbf{i^{(0)}}'''-\mathbf{1^{(l_1)}},j,l_2,p,d,0,\mathbf{n^{(k)}}-\mathbf{i^{(0)}}'''-\mathbf{1^{(l_1)}}-\mathbf{1^{(l_2)}}\right)\leq g_k\left(\mathbf{n^{(k)}}-\mathbf{i^{(0)}}'''-\mathbf{1^{(l_1)}},j,l_2,p,d,0,\mathbf{i^{(1)}}'''\right)
\end{align*}
Thus we have
\begin{align*}
&\quad D_k''\left(\mathbf{n^{(k)}},l_1,l_2,p,d+1,\mathbf{i^{(0)}},\mathbf{i^{(0)}}',\mathbf{i^{(0)}}''\right)\\    
&<h_k\left(\mathbf{n^{(k)}},p,d+1,\mathbf{i^{(0)}}'''\right)(1-p)^{i_0-i_{0,l_1}+i_0'-i_{0,l_2}'+2i_0''-i_{0,l_1}''-i_{0,l_2}''}\\\\
&\quad\cdot\prod_{j=1}^k(1-p)^{\left(i_0'''-i_{0,j}'''\right)\left(n_j-i_{0,j}'''-\mathbbm{1}_{l_1}(j)-\mathbbm{1}_{l_2}(j)\right)}\\
&\quad\cdot\sum_{0\leq i_{1,j}+i_{1,j}'+i_{1,j}''\leq n_j-i_{0,j}'''-\mathbbm{1}_{l_1}(j)-\mathbbm{1}_{l_2}(j)}\binom{n_j-i_{0,j}'''-\mathbbm{1}_{l_1}(j)-\mathbbm{1}_{l_2}(j)}{i_{1,j},i_{1,j}',i_{1,j}'',n_j-i_{0,j}'''-\mathbbm{1}_{l_1}(j)-\mathbbm{1}_{l_2}(j)-i_{1,j}'''}\\
&\quad\cdot\left((1-p)^{i_{0,j}-i_0}-1\right)^{i_{1,j}}\left((1-p)^{i_{0,j}'-i_0'}-1\right)^{i_{1,j}'}\\
&\quad\cdot\left((1-p)^{i_{0,j}'''-i_0'''}-(1-p)^{i_{0,j}-i_0}-(1-p)^{i_{0,j}'-i_0'}+1\right)^{i_{1,j}''}\\
&\quad\cdot\left(1-f\left(n,p,d,4npi_0'''\right)\right)^{\left(i_{1,j}+i_{1,j}''\right)g_k\left(\mathbf{n^{(k)}}-\mathbf{i^{(0)}}'''-\mathbf{1^{(l_2)}},j,l_1,p,d,0,\mathbf{n^{(k)}}-\mathbf{i^{(0)}}'''-\mathbf{1^{(l_1)}}-\mathbf{1^{(l_2)}}\right)}\\
&\quad\cdot\left(1-f\left(n,p,d,4np\left(i_0'''\right)\right)\right)^{\left(i_{1,j}'+i_{1,j}''\right)g_k\left(\mathbf{n^{(k)}}-\mathbf{i^{(0)}}'''-\mathbf{1^{(l_1)}},j,l_2,,p,d,0,\mathbf{n^{(k)}}-\mathbf{i^{(0)}}'''-\mathbf{1^{(l_1)}}-\mathbf{1^{(l_2)}}\right)}
\end{align*}
so that
\begin{align*}
&\quad D_k''\left(\mathbf{n^{(k)}},l_1,l_2,p,d+1,\mathbf{i^{(0)}},\mathbf{i^{(0)}}',\mathbf{i^{(0)}}''\right)\\    
&<h_k\left(\mathbf{n^{(k)}},p,d+1,\mathbf{i^{(0)}}'''\right)(1-p)^{i_0-i_{0,l_1}+i_0'-i_{0,l_2}'+2i_0''-i_{0,l_1}''-i_{0,l_2}''}\\\\
&\quad\cdot\prod_{j=1}^k(1-p)^{\left(i_0'''-i_{0,j}'''\right)\left(n_j-i_{0,j}'''-\mathbbm{1}_{l_1}(j)-\mathbbm{1}_{l_2}(j)\right)}\\
&\quad\cdot\left(1+\left((1-p)^{i_{0,j}-i_0}-1\right)\right.\\
&\qquad\qquad\left.\cdot\left(1-f\left(n,p,d,4npi_0'''\right)\right)^{g_k\left(\mathbf{n^{(k)}}-\mathbf{i^{(0)}}'''-\mathbf{1^{(l_2)}},j,l_1,p,d,0,\mathbf{n^{(k)}}-\mathbf{i^{(0)}}'''-\mathbf{1^{(l_1)}}-\mathbf{1^{(l_2)}}\right)}\right.\\
&\qquad\quad\left.+\left((1-p)^{i_{0,j}'-i_0'}-1\right)\right.\\
&\qquad\qquad\left.\cdot\left(1-f\left(n,p,d,4npi_0'''\right)\right)^{g_k\left(\mathbf{n^{(k)}}-\mathbf{i^{(0)}}'''-\mathbf{1^{(l_1)}},j,l_2,p,d,0,\mathbf{n^{(k)}}-\mathbf{i^{(0)}}'''-\mathbf{1^{(l_1)}}-\mathbf{1^{(l_2)}}\right)}\right.\\
&\qquad\quad\left.+\left((1-p)^{i_{0,j}'''-i_0'''}-(1-p)^{i_{0,j}-i_0}-(1-p)^{i_{0,j}'-i_0'}+1\right)\right.\\
&\qquad\qquad\left.\cdot\left(1-f\left(n,p,d,4npi_0'''\right)\right)^{g_k\left(\mathbf{n^{(k)}}-\mathbf{i^{(0)}}'''-\mathbf{1^{(l_2)}},j,l_1,p,d,0,\mathbf{n^{(k)}}-\mathbf{i^{(0)}}'''-\mathbf{1^{(l_1)}}-\mathbf{1^{(l_2)}}\right)}\right.\\
&\qquad\qquad\left.\cdot\left(1-f\left(n,p,d,4npi_0'''\right)\right)^{g_k\left(\mathbf{n^{(k)}}-\mathbf{i^{(0)}}'''-\mathbf{1^{(l_1)}},j,l_2,p,d,0,\mathbf{n^{(k)}}-\mathbf{i^{(0)}}'''-\mathbf{1^{(l_1)}}-\mathbf{1^{(l_2)}}\right)}\right)^{n_j-i_{0,j}'''-\mathbbm{1}_{l_1}(j)-\mathbbm{1}_{l_2}(j)}.
\end{align*}
Thus
\begin{align*}
&\quad D_k''\left(\mathbf{n^{(k)}},l_1,l_2,p,d+1,\mathbf{i^{(0)}},\mathbf{i^{(0)}}',\mathbf{i^{(0)}}''\right)\\   
&<h_k\left(\mathbf{n^{(k)}},p,d+1,\mathbf{i^{(0)}}'''\right)(1-p)^{i_0-i_{0,l_1}+i_0'-i_{0,l_2}'+2i_0''-i_{0,l_1}''-i_{0,l_2}''}\\\\
&\quad\cdot\prod_{j=1}^k\left((1-p)^{i_0'''-i_{0,j}'''}\right.\\
&\qquad\qquad\left.+\left((1-p)^{i_0'-i_{0,j}'+i_0''-i_{0,j}''}-(1-p)^{i_0'''-i_{0,j}'''}\right)\right.\\
&\qquad\qquad\quad\left.\cdot\left(1-f\left(n,p,d,4npi_0'''\right)\right)^{g_k\left(\mathbf{n^{(k)}}-\mathbf{i^{(0)}}'''-\mathbf{1^{(l_2)}},j,l_1,p,d,0,\mathbf{n^{(k)}}-\mathbf{i^{(0)}}'''-\mathbf{1^{(l_1)}}-\mathbf{1^{(l_2)}}\right)}\right.\\
&\qquad\qquad\left.+\left((1-p)^{i_0-i_{0,j}+i_0''-i_{0,j}''}-(1-p)^{i_0'''-i_{0,j}'''}\right)\right.\\
&\qquad\qquad\quad\left.\cdot\left(1-f\left(n,p,d,4npi_0'''\right)\right)^{g_k\left(\mathbf{n^{(k)}}-\mathbf{i^{(0)}}'''-\mathbf{1^{(l_1)}},j,l_2,p,d,0,\mathbf{n^{(k)}}-\mathbf{i^{(0)}}'''-\mathbf{1^{(l_1)}}-\mathbf{1^{(l_2)}}\right)}\right.\\
&\qquad\qquad\left.+\left((1-p)^{i_{0,j}'''-i_0'''}-(1-p)^{i_{0,j}-i_0}-(1-p)^{i_{0,j}'-i_0'}+1\right)\right.\\
&\qquad\qquad\quad\left.\cdot\left(1-f\left(n,p,d,4npi_0'''\right)\right)^{g_k\left(\mathbf{n^{(k)}}-\mathbf{i^{(0)}}'''-\mathbf{1^{(l_2)}},j,l_1,p,d,0,\mathbf{n^{(k)}}-\mathbf{i^{(0)}}'''-\mathbf{1^{(l_1)}}-\mathbf{1^{(l_2)}}\right)}\right.\\
&\qquad\qquad\quad\left.\cdot\left(1-f\left(n,p,d,4npi_0'''\right)\right)^{g_k\left(\mathbf{n^{(k)}}-\mathbf{i^{(0)}}'''-\mathbf{1^{(l_1)}},j,l_2,p,d,0,\mathbf{n^{(k)}}-\mathbf{i^{(0)}}'''-\mathbf{1^{(l_1)}}-\mathbf{1^{(l_2)}}\right)}\right)^{n_j-i_{0,j}'''-\mathbbm{1}_{l_1}(j)-\mathbbm{1}_{l_2}(j)}.
\end{align*}
We note that $g_k\left(\mathbf{n^{(k)}}-\mathbf{i^{(0)}}'''-\mathbf{1^{(l_1)}},j,j',p,d,0,\mathbf{n^{(k)}}-\mathbf{i^{(0)}}'''-\mathbf{1^{(l_1)}}-\mathbf{1^{(l_2)}}\right)<n^{d-1}$ for all $1\leq j\leq k$ and $j'=l_1,l_2$ and so, using Lemma \ref{lem6}, we have
\begin{align*}
&\quad D_k''\left(\mathbf{n^{(k)}},l_1,l_2,p,d+1,\mathbf{i^{(0)}},\mathbf{i^{(0)}}',\mathbf{i^{(0)}}''\right)\\
&<h_k\left(\mathbf{n^{(k)}},p,d+1,\mathbf{i^{(0)}}'''\right)(1-p)^{i_0-i_{0,l_1}+i_0'-i_{0,l_2}'+2i_0''-i_{0,l_1}''-i_{0,l_2}''}\\
&\quad\cdot\prod_{j=1}^k\left(1-f\left(n,p,d+1,i_0'''\right)\right)^{\left(i_0^{(4)}-i_{0,j}^{(4)}\right)g_k\left(\mathbf{n^{(k)}}-\mathbf{i^{(0)}}'''-\mathbf{1^{(l_2)}},j,l_1,p,d,0,\mathbf{n^{(k)}}-\mathbf{i^{(0)}}'''-\mathbf{1^{(l_1)}}-\mathbf{1^{(l_2)}}\right)\left(n_j-i_{0,j}'''-\mathbbm{1}_{l_1}(j)-\mathbbm{1}_{l_2}(j)\right)}\\
&\qquad\quad\cdot\left(1-f\left(n,p,d+1,i_0'''\right)\right)^{\left(i_0^{(5)}-i_{0,j}^{(5)}\right)g_k\left(\mathbf{n^{(k)}}-\mathbf{i^{(0)}}'''-\mathbf{1^{(l_1)}},j,l_2,p,d,0,\mathbf{n^{(k)}}-\mathbf{i^{(0)}}'''-\mathbf{1^{(l_1)}}-\mathbf{1^{(l_2)}}\right)\left(n_j-i_{0,j}'''-\mathbbm{1}_{l_1}(j)-\mathbbm{1}_{l_2}(j)\right)}\\
&<h_k\left(\mathbf{n^{(k)}},p,d+1,\mathbf{i^{(0)}}'''\right)\\
&\quad\cdot\left(1-f\left(n,p,d+1,i_0'''\right)\right)^{i_{0,l_1}^{(4)}\sum_{\substack{j=1\\j\neq l_1}}^kg_k\left(\mathbf{n^{(k)}}-\mathbf{i^{(0)}}'''-\mathbf{1^{(l_2)}},j,l_1,p,d,0,\mathbf{n^{(k)}}-\mathbf{i^{(0)}}'''-\mathbf{1^{(l_1)}}-\mathbf{1^{(l_2)}}\right)\left(n_j-i_{0,j}'''-\mathbbm{1}_{l_2}(j)\right)}\\
&\quad\cdot\prod_{\substack{q=1\\q\neq l_1}}^k\left(1-f\left(n,p,d+1,i_0'''\right)\right)^{i_{0,q}^{(4)}\left(1+\sum_{\substack{j=1\\j\neq q}}^kg_k\left(\mathbf{n^{(k)}}-\mathbf{i^{(0)}}'''-\mathbf{1^{(l_2)}},j,l_1,p,d,0,\mathbf{n^{(k)}}-\mathbf{i^{(0)}}'''-\mathbf{1^{(l_1)}}-\mathbf{1^{(l_2)}}\right)\left(n_j-i_{0,j}'''-\mathbbm{1}_{l_1}(j)-\mathbbm{1}_{l_2}(j)\right)\right)}\\
&\quad\cdot\left(1-f\left(n,p,d+1,i_0'''\right)\right)^{i_{0,l_2}^{(5)}\sum_{\substack{j=1\\j\neq l_2}}^kg_k\left(\mathbf{n^{(k)}}-\mathbf{i^{(0)}}'''-\mathbf{1^{(l_1)}},j,l_2,p,d,0,\mathbf{n^{(k)}}-\mathbf{i^{(0)}}'''-\mathbf{1^{(l_1)}}-\mathbf{1^{(l_2)}}\right)\left(n_j-i_{0,j}'''-\mathbbm{1}_{l_1}(j)\right)}\\
&\quad\cdot\prod_{\substack{q=1\\q\neq l_1}}^k\left(1-f\left(n,p,d+1,i_0'''\right)\right)^{i_{0,q}^{(5)}\left(1+\sum_{\substack{j=1\\j\neq q}}^kg_k\left(\mathbf{n^{(k)}}-\mathbf{i^{(0)}}'''-\mathbf{1^{(l_1)}},j,l_2,p,d,0,\mathbf{n^{(k)}}-\mathbf{i^{(0)}}'''-\mathbf{1^{(l_1)}}-\mathbf{1^{(l_2)}}\right)\left(n_j-i_{0,j}'''-\mathbbm{1}_{l_1}(j)-\mathbbm{1}_{l_2}(j)\right)\right)}.
\end{align*}
We can deduce that
\begin{align}
&\quad\sum_{\substack{j=1\\j\neq l_1}}^k\left(n_j-i_{0,j}'''-\mathbbm{1}_{l_2}(j)\right)g_k\left(\mathbf{n^{(k)}}-\mathbf{i^{(0)}}'''-\mathbf{1^{(l_2)}},j,l_1,p,d,0,\mathbf{n^{(k)}}-\mathbf{i^{(0)}}'''-\mathbf{1^{(l_1)}}-\mathbf{1^{(l_2)}}\right)\nonumber\\
&\geq g_k\left(\mathbf{n^{(k)}}-\mathbf{1^{(l_2)}},l_1,l_1,p,d,0,\mathbf{i^{(0)}}'''\right)\label{eqn41}
\end{align}
and
\begin{align}
&\quad\sum_{\substack{j=1\\j\neq l_2}}^k\left(n_j-i_{0,j}'''-\mathbbm{1}_{l_1}(j)\right)g_k\left(\mathbf{n^{(k)}}-\mathbf{i^{(0)}}'''-\mathbf{1^{(l_1)}},j,l_2,p,d,0,\mathbf{n^{(k)}}-\mathbf{i^{(0)}}'''-\mathbf{1^{(l_1)}}-\mathbf{1^{(l_2)}}\right)\nonumber\\
&\geq g_k\left(\mathbf{n^{(k)}}-\mathbf{1^{(l_1)}},l_2,l_2,p,d,0,\mathbf{i^{(0)}}'''\right).\label{eqn42}
\end{align}
Also,
\begin{align}
&\quad 1+\sum_{\substack{j=1\\j\neq q}}^k\left(n_j-i_{0,j}'''-\mathbbm{1}_{l_1}(j)-\mathbbm{1}_{l_2}(j)\right)g_k\left(\mathbf{n^{(k)}}-\mathbf{i^{(0)}}'''-\mathbf{1^{(l_2)}},j,l_1,p,d,0,\mathbf{n^{(k)}}-\mathbf{i^{(0)}}'''-\mathbf{1^{(l_1)}}-\mathbf{1^{(l_2)}}\right)\nonumber\\
&\geq g_k\left(\mathbf{n^{(k)}}-\mathbf{1^{(l_2)}},q,l_1,p,d,0,\mathbf{i^{(0)}}'''\right)\label{eqn43}
\end{align}
for all $\leq q\leq k$ with $q\neq l_1$, and
\begin{align}
&\quad 1+\sum_{\substack{j=1\\j\neq q}}^k\left(n_j-i_{0,j}'''-\mathbbm{1}_{l_1}(j)-\mathbbm{1}_{l_2}(j)\right)g_k\left(\mathbf{n^{(k)}}-\mathbf{i^{(0)}}'''-\mathbf{1^{(l_1)}},j,l_2,p,d,0,\mathbf{n^{(k)}}-\mathbf{i^{(0)}}'''-\mathbf{1^{(l_1)}}-\mathbf{1^{(l_2)}}\right)\nonumber\\
&\geq g_k\left(\mathbf{n^{(k)}}-\mathbf{1^{(l_1)}},q,l_2,p,d,0,\mathbf{i^{(0)}}'''\right)\label{eqn44}
\end{align}
for all $\leq q\leq k$ with $q\neq l_2$. Thus we have \eqref{eqn40}.
\end{case}
\begin{case}{$1<\frac{n_j-\mathbbm{1}_{l_1}(j)-\mathbbm{1}_{l_2}(j)}{i_{0,j}'''+4n_jpi_0'''}$ for some $1\leq j\leq k$.}
\normalfont
\newline
\newline
We proceed exactly as in Case $1$, except that for every $j$ with $1<\frac{n_j-\mathbbm{1}_{l_1}(j)-\mathbbm{1}_{l_2}(j)}{i_{0,j}'''+4n_jpi_0'''}$, we add in the condition $i_{1,j}'''\leq 4n_jpi_0'''$ over the summations of $i_{1,j},i_{1,j}',i_{1,j}''$ and multiply all the expressions following $D_k''\left(\mathbf{n^{(k)}},l_1,l_2,p,d+1,\mathbf{i^{(0)}},\mathbf{i^{(0)}}',\mathbf{i^{(0)}}''\right)<$ by $\left(1-\frac{4}{5}\left(\frac{e}{3}\right)^{4n_jpi_0'''}\right)^{-1}$, which can be justified by Lemma \ref{lem7}.
\end{case}
\begin{case}{$1<\frac{n_j-\mathbbm{1}_{l_1}(j)-\mathbbm{1}_{l_2}(j)}{i_{0,j}+i_{0,j}'+i_{0,j}''+\left(4n_jp+4n_jp(4np)+(4n_jp)(4np)^2+\ldots+(4n_jp)(4np)^{d'}\right)\left(i_0+i_0'+i_0''\right)}$ for all $1\leq j\leq k$.}
\normalfont
\newline
\newline
We proceed exactly as in Case $1$, except that for every $j$ with $1<\frac{n_j-\mathbbm{1}_{l_1}(j)-\mathbbm{1}_{l_2}(j)}{i_{0,j}+i_{0,j}'+i_{0,j}''+4n_jpi_0'''}$, we add in the condition $i_{1,j}'''\leq 4n_jpi_0'''$ over the summations of $i_{1,j},i_{1,j}',i_{1,j}''$ and multiply all the expressions following $D_k''\left(\mathbf{n^{(k)}},l_1,l_2,p,d+1,\mathbf{i^{(0)}},\mathbf{i^{(0)}}',\mathbf{i^{(0)}}''\right)<$ by $\left(1-\frac{4}{5}\left(\frac{e}{3}\right)^{4n_jpi_0'''}\right)^{-1}$, which can be justified by Lemma \ref{lem7}. As well, in all the functions $g_k'$, we replace $0$ with $d'$ (except where we replace it with $d'+1$ in the right-hand side of the inequalities \eqref{eqn41}, \eqref{eqn42}, \eqref{eqn43}, and \eqref{eqn44}) and replace $\mathbf{n^{(k)}}-\mathbf{i^{(0)}}'''-\mathbf{1^{(l_1)}}-\mathbf{1^{(l_2)}}$ with $4np\mathbf{i^{(0)}}'''$.
\end{case}
We can deduce that $D_k\left(\mathbf{n^{(k)}},l,p,d,\mathbf{i^{(0)}}\right)< D_k\left(\mathbf{n^{(k)}}-\mathbf{1^{(l)}},l,p,d,\mathbf{i^{(0)}}\right)$. Thus we have $n(b_1,b_2)<s^tD_k\left(\mathbf{n^{(k)}}-\mathbf{1^{(l)}},l,p,d,\mathbf{1^{(j_1)}}+\mathbf{1^{(j_2)}}\right)$ whenever $b_1$ and $b_2$ are not the same pair of vertices with one pair having its vertices in the partite set consisting of $n_l$ and $n_{j_1}$ vertices, and the other pair having its vertices in the partite set consisting of $n_l$ and $n_{j_2}$ vertices. Also, we have $n(b_1,b_2)\leq D_k''\left(\mathbf{n^{(k)}},l_1,l_2,p,d,\mathbf{1^{(j_1)}},\mathbf{1^{(j_2)}},\mathbf{0}\right)$ whenever $b_1$ and $b_2$ are not the same pair of vertices with one pair having its vertices in the partite set consisting of $n_{l_1}$ and $n_{j_1}$ vertices, and the other pair having its vertices in the partite set consisting of $n_{l_2}$ and $n_{j_2}$ vertices with $l_1\neq l_2$. Thus, by the Tur\'an sieve, we have
\begin{align*}
&\quad P\left(G\left(\mathbf{n^{(k)}},p\right),d\right)\\
&<\frac{\sum_{1\leq l\leq k}\left(n_l\left(n_l-1\right)\left(n_l-2\right)+\binom{n_l}{2}\binom{n_l-2}{2}\right)D_k\left(\mathbf{n^{(k)}}-\mathbf{1^{(l)}},l,p,d,2\cdot\mathbf{1^{(l)}}\right)}{\left(\sum_{l=1}^k\binom{n_l}{2}D\left(\mathbf{n^{(k)}},l,p,d,\mathbf{1^{(l)}}\right)+\sum_{1\leq j<l\leq k}n_jn_lD\left(\mathbf{n^{(k)}},l,p,d,\mathbf{1^{(j)}}\right)\right)^2}\\
&\quad+\frac{\sum_{1\leq j\neq l\leq k }\left(n_l\left(n_l-1\right)n_j+2\binom{n_l}{2}\left(n_l-2\right)n_j\right)D_k\left(\mathbf{n^{(k)}}-\mathbf{1^{(l)}},l,p,d,\mathbf{1^{(l)}}+\mathbf{1^{(j)}}\right)}{\left(\sum_{l=1}^k\binom{n_l}{2}D\left(\mathbf{n^{(k)}},l,p,d,\mathbf{1^{(l)}}\right)+\sum_{1\leq j<l\leq k}n_jn_lD\left(\mathbf{n^{(k)}},l,p,d,\mathbf{1^{(j)}}\right)\right)^2}\\
&\quad+\frac{\sum_{1\leq j\neq l\leq k }\left(n_ln_j\left(n_j-1\right)+n_ln_j\left(n_l-1\right)\left(n_j-1\right)\right)D_k\left(\mathbf{n^{(k)}}-\mathbf{1^{(l)}},l,p,d,2\cdot\mathbf{1^{(j)}}\right)}{\left(\sum_{l=1}^k\binom{n_l}{2}D\left(\mathbf{n^{(k)}},l,p,d,\mathbf{1^{(l)}}\right)+\sum_{1\leq j<l\leq k}n_jn_lD\left(\mathbf{n^{(k)}},l,p,d,\mathbf{1^{(j)}}\right)\right)^2}\\
&\quad+\frac{\sum_{1\leq j\neq l\leq k}\binom{n_l}{2}\binom{n_j}{2}D_k''\left(\mathbf{n^{(k)}},l_1,l_2,p,d,\mathbf{1^{(j_1)}},\mathbf{1^{(j_2)}},\mathbf{0}\right)}{\left(\sum_{l=1}^k\binom{n_l}{2}D\left(\mathbf{n^{(k)}},l,p,d,\mathbf{1^{(l)}}\right)+\sum_{1\leq j<l\leq k}n_jn_lD\left(\mathbf{n^{(k)}},l,p,d,\mathbf{1^{(j)}}\right)\right)^2}\\
&\quad+\frac{\sum_{\substack{1\leq j_1,j_2,l\leq k\\l,j_1,j_2\text{ all distinct}}}n_l^2n_{j_1}n_{j_2}D_k\left(\mathbf{n^{(k)}}-\mathbf{1^{(l)}},l,p,d,\mathbf{1^{(j_1)}}+\mathbf{1^{(j_2)}}\right)}{\left(\sum_{l=1}^k\binom{n_l}{2}D\left(\mathbf{n^{(k)}},l,p,d,\mathbf{1^{(l)}}\right)+\sum_{1\leq j<l\leq k}n_jn_lD\left(\mathbf{n^{(k)}},l,p,d,\mathbf{1^{(j)}}\right)\right)^2}\\
&\quad+\frac{\sum_{\substack{1\leq j,l_1,l_2\leq k\\j,l_1,l_2\text{ all distinct}}}\binom{n_{l_1}}{2}n_{l_2}n_jD_k''\left(\mathbf{n^{(k)}},l_1,l_2,p,d,\mathbf{1^{(l_1)}},\mathbf{1^{(j)}},\mathbf{0}\right)}{\left(\sum_{l=1}^k\binom{n_l}{2}D\left(\mathbf{n^{(k)}},l,p,d,\mathbf{1^{(l)}}\right)+\sum_{1\leq j<l\leq k}n_jn_lD\left(\mathbf{n^{(k)}},l,p,d,\mathbf{1^{(j)}}\right)\right)^2}\\
&\quad+\frac{\sum_{\substack{1\leq l_1,l_2,j_1,j_2\leq k\\l_1,l_2,j_1,j_2\text{ all distinct}}}\frac{l_1l_2j_1j_2}{4}D_k''\left(\mathbf{n^{(k)}},l_1,l_2,p,d,\mathbf{1^{(j_1)}},\mathbf{1^{(j_2)}},\mathbf{0}\right)}{\left(\sum_{l=1}^k\binom{n_l}{2}D\left(\mathbf{n^{(k)}},l,p,d,\mathbf{1^{(l)}}\right)+\sum_{1\leq j<l\leq k}n_jn_lD\left(\mathbf{n^{(k)}},l,p,d,\mathbf{1^{(j)}}\right)\right)^2}\\
&\quad-1+\frac{1}{\sum_{l=1}^k\binom{n_l}{2}D\left(\mathbf{n^{(k)}},l,p,d,\mathbf{1^{(l)}}\right)+\sum_{1\leq j<l\leq k}n_jn_lD\left(\mathbf{n^{(k)}},l,p,d,\mathbf{1^{(j)}}\right)}.
\end{align*}
Theorem \ref{bigthmkpartite} follows.
\section{Restricted Results for $k$-partite Graphs with diameter $d\geq 3$}
We impose further restrictions on $n_1,n_2,\ldots,n_k$, and $p$ in Theorem \ref{bigthmkpartite} to make our result more clear and meaningful. Since the case $d=2$ was treated in Section $3$, we assume $d\geq 3$. The result is Corollary \ref{bigcor3}.
\begin{corollary}\label{bigcor3}
Let $d\geq 3$ be fixed. Suppose that \eqref{cond1} and \eqref{cond2} hold. Also suppose that $n_1\leq n_2\leq\ldots\leq n_k$ and
\begin{equation}\label{eqn45}
n^{1-\frac{1}{2d}+\frac{1}{2d^2}}\leq n_1.
\end{equation}
Then we have
\begin{equation}\label{eqn51}
 P(G(\mathbf{n^{(k)}},p),d)>1-\left(1+2^{k+1}4^ddn^{\frac{-1}{2d^2}}\right)\sum_{1\leq j,l\leq k}\frac{n_jn_l}{2}\left(1-p^d\right)^{u\left(\mathbf{n^{(k)}},d-1,j,l\right)}.
\end{equation}
and
\begin{equation}\label{eqn52}
P(G(\mathbf{n^{(k)}},p),d)<4^dd2^{k+3}n^{\frac{-1}{2d^2}}+\left(\sum_{1\leq j,l\leq k}\frac{n_jn_l}{2}\left(1-p^d\right)^{u\left(\mathbf{n^{(k)}},d-1,j,l\right)}\right)^{-1}\left(1+5n^{\frac{-1}{2d^2}}\right).
\end{equation}
If we have $\frac{n}{k}-1<n_j<\frac{n}{k}+1$ for all $1\leq j\leq k$, i.e. we are dealing with $k$-partite Tur\'an graphs, then we have
\begin{equation}\label{eqn55}
P(G(\mathbf{n^{(k)}},p),d)>1-\frac{n^2\left(1-p^d\right)^{\left(n-\frac{n}{k}-1\right)^{d-1}}}{2k}\left(1+(k-1)\left(1-p^d\right)^{-\left(n-\frac{n}{k}-1\right)^{d-2}\left(1+\frac{k}{n}\right)}\right)\left(1+\frac{k}{n}\right)^2
\end{equation}
and
\begin{equation}\label{eqn56}
P(G(\mathbf{n^{(k)}},p),d)<\frac{2k}{n^2\left(1-p^d\right)^{\left(n-\frac{n}{k}+1\right)^{d-1}}}\left(1+(k-1)\left(1-p^d\right)^{-\left(n-\frac{n}{k}+1\right)^{d-2}\left(\frac{n}{k}-1\right)}\right)^{-1}\left(1-\frac{2k}{n}\right)^{-1}.
\end{equation}
\end{corollary}
We prove Corollary \ref{bigcor3}. Suppose \eqref{cond1}, \eqref{cond2}, and \eqref{eqn45} all hold. As in the proof of Corollary \ref{bigcor} we can derive \eqref{eqn3}. From \eqref{cond1}, \eqref{cond2}, and \eqref{eqn45}, we can deduce that $\frac{3}{n_k}+\frac{64p(4np)^{d-5}}{7}<\frac{1}{4}$. Thus all of the $h_k$ functions in Theorem \ref{bigthmkpartite} are bounded above by
\begin{align}
&\quad\prod_{j=1}^k\left(1-\frac{4}{5}\left(\frac{e}{3}\right)^{4p\left(n_j-3-8n_jp\left(1+4np+\ldots+(4np)^{d-5}\right)\right)}\right)^{2-d}\nonumber\\
&<\left(1-\frac{4}{5}\left(\frac{e}{3}\right)^{4pn_j\left(1-\frac{1}{4}\right)}\right)^{-dk}\nonumber\\
&<\left(1-\frac{4}{5}\left(\frac{e}{3}\right)^{3n^{1/(2d)}}\right)^{-dk}\nonumber\\
&<\left(1+4d\left(\frac{e}{3}\right)^{3n^{1/(2d)}}\right)^k\nonumber\\
&<\left(1+4^ddn^{-1/(2d^2)}\right)^k\label{eqn46}
\end{align}
with the last two inequalities following from \eqref{cond2}. From \eqref{cond1}, \eqref{cond2}, and \eqref{eqn45}, we can derive
\begin{equation*}
\frac{8n_jp(4np)^{d-2}}{1-\frac{1}{4np}}\leq n_j-4
\end{equation*}
for all $1\leq j\leq k$ and so we can apply Theorem \ref{bigthmkpartite} with $d'=d-1$. All of the $g_k$ and $g_k'$ functions in Theorem \ref{bigthmkpartite} wiht $d'=d-1$ are bounded below by
\begin{align}
&\quad v_1\left(\mathbf{n^{(k)}}-\sum_{r=1}^k\mathbf{1^{(r)}},2\cdot\sum_{r=1}^k\mathbf{(1)^{(r)}},p,d-1,j,l\right)\nonumber\\
&>\sum_{(i_1,i_2,\ldots,i_{d-1})\in[k]^{d-1,q,l\neq}}\prod_{m=1}^{d-1}n_{i_m}\left(1-\frac{4}{n_{i_m}}-\frac{8p(4np)^{d-3}}{1-\frac{1}{4np}}\right)\nonumber\\
&>\sum_{(i_1,i_2,\ldots,i_{d-1})\in[k]^{d-1,q,l\neq}}\prod_{m=1}^{d-1}n_{i_m}\left(1-\frac{4}{n_{i_m}}-\frac{64p(4np)^{d-3}}{7}\right)\nonumber\\
&>\left(1-4^{d-1}n^{-3/(2d)}\right)^{d-1}u\left(\mathbf{n^{(k)}},d-1,j,l\right)\nonumber\\
&>\left(1-4^{d-1}dn^{-3/(2d)}\right)u\left(\mathbf{n^{(k)}},d-1,j,l\right)\nonumber\\
&>\left(1-4^{d-1}dn^{\frac{-1}{2d}-\frac{1}{2d^2}}\right)u\left(\mathbf{n^{(k)}},d-1,j,l\right)\label{eqn47}
\end{align}
with the last three inequalities being derived from \eqref{cond1}, \eqref{cond2}, and \eqref{eqn45}. Making use of \eqref{eqn3}, \eqref{eqn46}, and \eqref{eqn47}, we obtain
\begin{align*}
P(G(\mathbf{n^{(k)}},p),d)&>1-\sum_{l=1}^k\binom{n_l}{2}\left(1+4^ddn^{\frac{-1}{2d^2}}\right)^k\left(1-p^d\right)^{\left(1-4^{d-1}dn^{\frac{-1}{2d}-\frac{1}{2d^2}}\right)^2u\left(\mathbf{n^{(k)}},d-1,l,l\right)}\\
&\qquad-\sum_{1\leq j<l\leq k}n_jn_l\left(1+4^ddn^{\frac{-1}{2d^2}}\right)^k\left(1-p^d\right)^{\left(1-4^{d-1}dn^{\frac{-1}{2d}-\frac{1}{2d^2}}\right)^2u\left(\mathbf{n^{(k)}},d-1,j,l\right)}.
\end{align*}
As in the proof of Corollary \ref{bigcor}, we deduce
\begin{equation*}
P(G(\mathbf{n^{(k)}},p),d)>1-\left(1+2^{k+1}4^ddn^{\frac{-1}{2d^2}}\right)\sum_{1\leq j,l\leq k}\frac{n_jn_l}{2}\left(1-p^d\right)^{u\left(\mathbf{n^{(k)}},d-1,j,l\right)}.
\end{equation*}
By \eqref{cond2} and \eqref{eqn45}, we have
\begin{equation*}
n_jp\geq n^{\frac{1}{2d}+\frac{1}{2d^2}}>4^dd\geq 192
\end{equation*}
for all $1\leq j\leq k$. So we have
\begin{align*}
\sum_{m=0}^{d-1}u\left(\mathbf{n^{(k)}},m,j,l\right)p^{m-d+1}&<u\left(\mathbf{n^{(k)}},d-1,j,l\right)\left(1-n^{\frac{-1}{2d}-\frac{1}{2d^2}}\right)^{-1}\\
&<u\left(\mathbf{n^{(k)}},d-1,j,l\right)\left(1+\frac{192}{191}n^{\frac{-1}{2d}-\frac{1}{2d^2}}\right).
\end{align*}
We can derive
\begin{equation*}
2\cdot\frac{192}{191}p^dn^{d-1-\frac{1}{2d}-\frac{1}{2d^2}}\leq 2\cdot\frac{192}{191}n^{\frac{-1}{2d^2}}<\frac{1}{32}.
\end{equation*}
Thus we have
\begin{equation*}
(1-p^d)^{-2u\left(\mathbf{n^{(k)}},d-1,j,l\right)\frac{192}{191}n^{\frac{-1}{2d}-\frac{1}{2d^2}}}<\left(1-p^d\right)^{-\frac{384}{191}n^{d-1-\frac{1}{2d}-\frac{1}{2d^2}}}<\left(1+\frac{384\cdot 32n^{\frac{-1}{2d^2}}}{191\cdot 31}\right)<1+3n^{\frac{-1}{2d^2}}.
\end{equation*}
Thus
\begin{align}
&\quad\left(\sum_{l=1}^k\binom{n_l}{2}(1-p^d)^{\sum_{m=0}^{d-1}u\left(\mathbf{n^{(k)}},m,l,l\right)p^{m-d+1}}+\sum_{1\leq j<l\leq k}n_jn_l(1-p^d)^{\sum_{m=0}^{d-1}u\left(\mathbf{n^{(k)}},m,j,l\right)p^{m-d+1}}\right)^{-1}\nonumber\\
&<\left(\sum_{l=1}^k\binom{n_l}{2}(1-p^d)^{u\left(\mathbf{n^{(k)}},d-1,l,l\right)}+\sum_{1\leq j<l\leq k}n_jn_l(1-p^d)^{u\left(\mathbf{n^{(k)}},d-1,j,l\right)}\right)^{-1}\left(1+3n^{\frac{-1}{2d^2}}\right)\label{eqn48}
\end{align}
and
\begin{align}
&\quad\left(\sum_{l=1}^k\binom{n_l}{2}(1-p^d)^{\sum_{m=0}^{d-1}u\left(\mathbf{n^{(k)}},m,l,l\right)p^{m-d+1}}+\sum_{1\leq j<l\leq k}n_jn_l(1-p^d)^{\sum_{m=0}^{d-1}u\left(\mathbf{n^{(k)}},m,j,l\right)p^{m-d+1}}\right)^{-2}\nonumber\\
&<\left(\sum_{l=1}^k\binom{n_l}{2}(1-p^d)^{u\left(\mathbf{n^{(k)}},d-1,l,l\right)}+\sum_{1\leq j<l\leq k}n_jn_l(1-p^d)^{u\left(\mathbf{n^{(k)}},d-1,j,l\right)}\right)^{-2}\left(1+3n^{\frac{-1}{2d^2}}\right)\label{eqn49}
\end{align}
If we let $h_k()$, $g_{k,j,l}$, and $g_{k,j,l}'$ stand for any of the $h_k$, $g_k$, and $g_k'$ functions respectively where $j$ and $l$ are in the second and third arguments respectively in the functions in Theorem \ref{bigthmkpartite}, then for any $1\leq j_1,j_2,l_1,l_2\leq k$ we have
\begin{align}
&\quad h_k()\left(1-f(n,p,d,2)\right)^{g_{k,j_1,l_1}+g_{k,j_2,l_2}}\nonumber\\
&<\left(1+4^ddn^{\frac{-1}{2d^2}}\right)^k\left(1-p^d\right)^{\left(1-4^{d-1}dn^{\frac{-1}{2d}-\frac{1}{2d^2}}\right)^2\left(u\left(\mathbf{n^{(k)}},d-1,j_1,l_1\right)+u\left(\mathbf{n^{(k)}},d-1,j_2,l_2\right)\right)}\nonumber\\
&<\left(1+4^ddn^{\frac{-1}{2d^2}}\right)^{k+2}\left(1-p^d\right)^{u\left(\mathbf{n^{(k)}},d-1,j_1,l_1\right)+u\left(\mathbf{n^{(k)}},d-1,j_2,l_2\right)}\label{eqn50}
\end{align}
where we also have the above if we replace $g_{k,j_1,l_1}$ with $g_{k,j_1,l_1}'$ or if we replace $g_{k,j_2,l_2}$ with $g_{k,j_2,l_2}'$ making use of \eqref{eqn3}, \eqref{eqn46}, and \eqref{eqn47}. Substituting \eqref{eqn48}, \eqref{eqn49}, and \eqref{eqn50} into the upper bound in Theorem \ref{bigthmkpartite} and noting that
\begin{equation*}
n_l\left(n_l-1\right)\left(n_l-2\right)+\binom{n_l}{2}\binom{n_l-2}{2}<\binom{n_l}{2}^2,    
\end{equation*}
\begin{equation*}
n_l\left(n_l-1\right)n_j+2\binom{n_l}{2}\left(n_l-2\right)n_j<2\binom{n_l}{2}n_ln_j,    
\end{equation*}
and
\begin{equation*}
n_ln_j\left(n_j-1\right)n_ln_j\left(n_l-1\right)\left(n_j-1\right)<n_l^2n_j^2
\end{equation*}
for all $1\leq l\neq j\leq k$ we obtain
\begin{align*}
P(G(\mathbf{n^{(k)}},p),d)&<\left(1+4^ddn^{\frac{-1}{2d^2}}\right)^{k+3}-1\\
&\quad+\left(\sum_{l=1}^k\binom{n_l}{2}(1-p^d)^{u\left(\mathbf{n^{(k)}},d-1,l,l\right)}+\sum_{1\leq j<l\leq k}n_jn_l(1-p^d)^{u\left(\mathbf{n^{(k)}},d-1,j,l\right)}\right)^{-1}\left(1+3n^{\frac{-1}{2d^2}}\right)\\
&<4^dd2^{k+3}n^{\frac{-1}{2d^2}}\\
&\quad+\left(\sum_{1\leq j,l\leq k}\frac{n_jn_l}{2}\left(1-p^d\right)^{u\left(\mathbf{n^{(k)}},d-1,j,l\right)}\right)^{-1}\left(1-\frac{1}{n_1}\right)^{-1}\left(1+3n^{\frac{-1}{2d^2}}\right)\\
&<4^dd2^{k+3}n^{\frac{-1}{2d^2}}\\
&\quad+\left(\sum_{1\leq j,l\leq k}\frac{n_jn_l}{2}\left(1-p^d\right)^{u\left(\mathbf{n^{(k)}},d-1,j,l\right)}\right)^{-1}\left(1-n^{\frac{1}{2d}+\frac{1}{2d^2}-1}\right)^{-1}\left(1+3n^{\frac{-1}{2d^2}}\right)\\
&<4^dd2^{k+3}n^{\frac{-1}{2d^2}}+\left(\sum_{1\leq j,l\leq k}\frac{n_jn_l}{2}\left(1-p^d\right)^{u\left(\mathbf{n^{(k)}},d-1,j,l\right)}\right)^{-1}\left(1+5n^{\frac{-1}{2d^2}}\right).
\end{align*}
To get the results for the $k$-partite Tur\'an graphs, we proceed as follows. In this case we know that $\frac{n}{k}-1<n_j<\frac{n}{k}+1$ for all $1\leq j\leq k$ and so we can deduce that
\begin{equation}\label{eqn53}
\left(n-\frac{n}{k}-1\right)^{d-2}\left(n-\frac{2n}{k}-2\right)<u\left(\mathbf{n^{(k)}},d-1,j,l\right)<\left(n-\frac{n}{k}+1\right)^{d-2}\left(n-\frac{2n}{k}+2\right)
\end{equation}
if $j\neq l$ and
\begin{equation}\label{eqn54}
\left(n-\frac{n}{k}-1\right)^{d-1}<u\left(\mathbf{n^{(k)}},d-1,l,l\right)<\left(n-\frac{n}{k}+1\right)^{d-1}.
\end{equation}
Using \eqref{eqn51}, \eqref{eqn52}, \eqref{eqn53}, and \eqref{eqn54} we obtain \eqref{eqn55} and \eqref{eqn56}.
\section{Directed $k$-partite Graphs for diameter $d\geq 2$}
Using the above methods, we can obtain similar results about the probability of a random directed $k$-partite graph with the partite sets containing $n_1\leq n_2\leq\ldots\leq n_k$ vertices respectively having diameter $d$ where each directed edge is chosen independently with probability $p$. Furthermore, for any two vertices, say $v_1$ and $v_2$, the existence of the edge from $v_1$ to $v_2$ has probability $p$, while the existence of the edge from $v_2$ to $v_1$ also occurs with probability $p$, and these two edges occur independently. We proceed exactly as above the only changes being replacing the factor of $s^{\frac{i_{0,j}\left(i_0-i_{0,j}\right)}{2}}$ with $s^{i_{0,j}\left(n-n_j\right)}$ in \eqref{eqn32}, replacing $t'$ with $\left(i_{0,j}+i_{0,j}'\right)\left(n-n_j\right)$ in \eqref{eqn34}, and replacing the factor of $s^{\frac{i_{0,j}'''\left(i_0'''-i_{0,j}'''\right)}{2}}$ with $s^{i_{0,j}'''\left(n-n_j\right)}$ in \eqref{eqn36}, and replacing $\binom{n_j}{2}$ and $n_jn_l$ whenever they occur with $n_j\left(n_j-1\right)$ and $2n_jn_l$ respectively. The only other extra consideration is in our calculation for $n(b_1,b_2)$ where one pair of vertices has its vertices in the partite sets consisting of $n_{j_1}$ and $n_l$ vertices and the other pair of vertices has its vertices in the partite sets consisting of $n_{j_2}$ and $n_l$ vertices (here $j_1$ and $j_2$ may or may not be the same) where the paths concerned ends at one of the vertices in the $n_l$ set and begins at the other vertex in the $n_l$ set. To deal with this case, we would define $C_k'''\left(\mathbf{n^{(k)}},l,p,d,\mathbf{i^{(0)}},\mathbf{i^{0}}'\right)$, which we define the same way as $C_k'\left(\mathbf{n^{(k)}},l,p,d,\mathbf{i^{(0)}},\mathbf{i^{0}}'\right)$, except we consider directed paths from the $i_0$ vertices to vertex $v$, and directed paths from the vertex $v'$ to the $i_0'$ vertices and this case can be dealt with in exactly the same way as $C_k'\left(\mathbf{n^{(k)}},l,p,d,\mathbf{i^{(0)}},\mathbf{i^{0}}'\right)$. Consequently, in Theorem \ref{bigthmkpartite} and Corollary \ref{bigcor3}, we multiply the second and third terms of the lower bound by $2$, divide the last term in the upper bound in Theorem \ref{bigthmkpartite} by $2$, and divide the first term in the upper bound in Corollary \ref{bigcor3} by $2$ to get the analogous results for random directed graphs. Everything else is left unchanged.
\section{Bipartite Graphs for diameter $d\geq 3$}
Here we analyze the diameters of bipartite graphs. Let $G(n_1,n_2,p)$ denote the set of all simple bipartite graphs with partite sets of size $n_1$ vertices and $n_2$ vertices where each edge is chosen independently with probability $p$. Here we obtain upper and lower bounds on the probability of a random simple bipartite graph with partite sets of size $n_1$ vertices and $n_2$ vertices with independent edge selection having diameter at most $d$ for any specific $d\geq 2$, $d\in\mathbb{N}$. Again, we impose restrictions on $n_1,n_2$, and $p$. Then in the next section, we refine this result to make it more clear and meaningful by imposing further restrictions on $n_1,n_2$, and $p$. First, a note.
\begin{note}
Throughout this note let
\begin{align*}
&\quad g_b(n,n_j,p,d,d',i_0)\\
&:=\begin{cases}
                      n-n_j &d=2\\
                      n-n_j\\
                      \quad+(n-n_j)\sum_{l=1}^{\frac{d'}{2}}\prod_{m=1}^{l}\left(n-n_j-\sum_{q=1}^m(4np)^{2q-1}i_0\right)\\
                      \qquad\qquad\qquad\qquad\qquad\cdot\left(n_j-1-\sum_{q=0}^{m-1}(4np)^{2q}i_0\right) & d,d'\text{ both even, }d'<d-3\\
                      n-n_j\\
                      \quad+(n-n_j)\sum_{l=1}^{\frac{d'+1}{2}}\prod_{m=1}^{l}\left(n-n_j-\sum_{q=1}^m(4np)^{2q-1}i_0\right)\\
                      \qquad\qquad\qquad\qquad\qquad\cdot\left(n_j-1-\sum_{q=0}^{m-1}(4np)^{2q}i_0\right) & d\text{ even, }d'\text{ odd, }d'<d-3\\
                      1+\sum_{l=0}^{\frac{d'}{2}}\prod_{m=0}^{l}\left(n_j-1-\sum_{q=1}^m(4np)^{2q-1}i_0\right)\\
                      \qquad\qquad\qquad\cdot\left(n-n_j-\sum_{q=0}^{m}(4np)^{2q}i_0\right) & d\text{ odd },d'\text{ even, }d'<d-3\\
                      1+\sum_{l=0}^{\frac{d'-1}{2}}\prod_{m=0}^{l}\left(n_j-1-\sum_{q=1}^m(4np)^{2q-1}i_0\right)\\
                      \qquad\qquad\qquad\cdot\left(n-n_j-\sum_{q=0}^{m}(4np)^{2q}i_0\right) & d,d'\text{ both odd, }d'<d-3\\
                      n-n_j\\
                      \quad+(n-n_j)\sum_{l=1}^{\frac{d-2}{2}}\prod_{m=1}^{l}\left(n-n_j-\sum_{q=1}^m(4np)^{2q-1}i_0\right)\\
                      \qquad\qquad\qquad\qquad\qquad\cdot\left(n_j-1-\sum_{q=0}^{m-1}(4np)^{2q}i_0\right) & d\text{ even, }d-3\leq d'\\
                      1+\sum_{l=0}^{\frac{d-3}{2}}\prod_{m=0}^{l}\left(n_j-1-\sum_{q=1}^m(4np)^{2q-1}i_0\right)\\
                      \qquad\qquad\qquad\cdot\left(n-n_j-\sum_{q=0}^{m}(4np)^{2q}i_0\right) & d\text{ odd, }d-3\leq d'.
                       \end{cases}
\end{align*}
and
\begin{align*}
&\quad g_b'(n_1,n_2,p,d,d',i_0,i_0')\\
&:=\begin{cases}
n_2-i_0' & d=2\\\\
n_2-i_0'\\
\quad+\sum_{l=0}^{\frac{d'}{2}-1}\left(n_2-1-\sum_{q=0}^l(4np)^{2q+1}i_0-\sum_{q=0}^{l+1}(4np)^{2q}i_0'\right)\\
\qquad\qquad\cdot\prod_{m=0}^l\left(n_1-1-\sum_{q=0}^m(4np)^{2q}i_0-\sum_{q=0}^m(4np)^{2q+1}i_0'\right) & d\geq 3\\
\qquad\qquad\qquad\cdot\left(n_2-1-\sum_{q=0}^m(4np)^{2q}i_0'-\sum_{q=0}^{m-1}(4np)^{2q+1}i_0\right) & d'\leq d-3, d'\text{ even,}\\
\quad+\sum_{l=0}^{\frac{d'}{2}-1}\prod_{m=0}^l\left(n_1-1-\sum_{q=0}^m(4np)^{2q}i_0-\sum_{q=0}^m(4np)^{2q+1}i_0'\right) & \\
\qquad\qquad\qquad\cdot\left(n_2-1-\sum_{q=0}^m(4np)^{2q}i_0'-\sum_{q=0}^{m-1}(4np)^{2q+1}i_0\right) & \\\\
n_2-i_0'\\
\quad+\sum_{l=0}^{\frac{d'-1}{2}}\left(n_2-1-\sum_{q=0}^l(4np)^{2q+1}i_0-\sum_{q=0}^{l+1}(4np)^{2q}i_0'\right)\\
\qquad\qquad\cdot\prod_{m=0}^l\left(n_1-1-\sum_{q=0}^m(4np)^{2q}i_0-\sum_{q=0}^m(4np)^{2q+1}i_0'\right) & d\geq 4\\
\qquad\qquad\qquad\cdot\left(n_2-1-\sum_{q=0}^m(4np)^{2q}i_0'-\sum_{q=0}^{m-1}(4np)^{2q+1}i_0\right) & d'\leq d-3, d'\text{ odd,}\\
\quad+\sum_{l=0}^{\frac{d'-1}{2}}\prod_{m=0}^l\left(n_1-1-\sum_{q=0}^m(4np)^{2q}i_0-\sum_{q=0}^m(4np)^{2q+1}i_0'\right) & \\
\qquad\qquad\qquad\cdot\left(n_2-1-\sum_{q=0}^m(4np)^{2q}i_0'-\sum_{q=0}^{m-1}(4np)^{2q+1}i_0\right) & \\\\
n_2-i_0'\\
\quad+\sum_{l=0}^{\frac{d}{2}-2}\left(n_2-1-\sum_{q=0}^l(4np)^{2q+1}i_0-\sum_{q=0}^{l+1}(4np)^{2q}i_0'\right)\\
\qquad\qquad\cdot\prod_{m=0}^l\left(n_1-1-\sum_{q=0}^m(4np)^{2q}i_0-\sum_{q=0}^m(4np)^{2q+1}i_0'\right) & d\geq 4\\
\qquad\qquad\qquad\cdot\left(n_2-1-\sum_{q=0}^m(4np)^{2q}i_0'-\sum_{q=0}^{m-1}(4np)^{2q+1}i_0\right) & d'>d-3, d\text{ even,}\\
\quad+\sum_{l=0}^{\frac{d}{2}-2}\prod_{m=0}^l\left(n_1-1-\sum_{q=0}^m(4np)^{2q}i_0-\sum_{q=0}^m(4np)^{2q+1}i_0'\right) & \\
\qquad\qquad\qquad\cdot\left(n_2-1-\sum_{q=0}^m(4np)^{2q}i_0'-\sum_{q=0}^{m-1}(4np)^{2q+1}i_0\right) & \\\\
n_2-i_0'\\
\quad+\sum_{l=0}^{\frac{d-3}{2}}\left(n_2-1-\sum_{q=0}^l(4np)^{2q+1}i_0-\sum_{q=0}^{l+1}(4np)^{2q}i_0'\right)\\
\qquad\qquad\cdot\prod_{m=0}^l\left(n_1-1-\sum_{q=0}^m(4np)^{2q}i_0-\sum_{q=0}^m(4np)^{2q+1}i_0'\right) & d\geq 3\\
\qquad\qquad\qquad\cdot\left(n_2-1-\sum_{q=0}^m(4np)^{2q}i_0'-\sum_{q=0}^{m-1}(4np)^{2q+1}i_0\right) & d'>d-3, d\text{ odd,}\\
\quad+\sum_{l=0}^{\frac{d-3}{2}}\prod_{m=0}^l\left(n_1-1-\sum_{q=0}^m(4np)^{2q}i_0-\sum_{q=0}^m(4np)^{2q+1}i_0'\right) & \\
\qquad\qquad\qquad\cdot\left(n_2-1-\sum_{q=0}^m(4np)^{2q}i_0'-\sum_{q=0}^{m-1}(4np)^{2q+1}i_0\right). &
\end{cases}
\end{align*}
\end{note}
We will prove the following theorem.
\begin{theorem}\label{bigthmbipartite}
Fix $d\geq 2$, $d\in\mathbb{N}$. Let $G(n_1,n_2,p)$ denote the set of all simple bipartite graphs with partite vertex sets of size $n_1$ and $n_2$ vertices and where each edge is chosen independently with probability $p$. Also, let $P(G(n_1,n_2,p),d)$ be the probability of a graph from $G(n_1,n_2,p)$ having diameter at most $d$. Suppose that $d$ is odd. Let $d'\geq 0$. Suppose that
\begin{equation*}
2+8np+2(4np)^2+\ldots+2(4np)^{d'}\leq n_j-1
\end{equation*}
for $j=1,2$ where $d'\geq 0$. If $d$ is odd, we have
\begin{align*}
P(G(n_1,n_2,p),d)&>1-n_1n_2h(n,p,d,1)\left(1-f(n,p,d,1)\right)^{g_b(n,n_1,p,d,d',1)}
\end{align*}
and
\begin{align*}
P(G(n_1,n_2,p),d)&<(1-p^d)^{-2\left((n_1n_2)^{\frac{d-1}{2}}+\sum_{j=1}^{\frac{d-1}{2}}(n_1n_2)^{\frac{d-1}{2}-j}\left(p^{1-2j}+p^{-2j}\right)\right)}h(n-1,p,d,2)\\
&\quad\cdot\left(1-f(n,p,d,2)\right)^{2\cdot g_b(n-1,n_j-1,p,d,d',2)}\\
&\quad-1+\frac{1}{n_1n_2(1-p^d)^{\left((n_1n_2)^{\frac{d-1}{2}}+\sum_{l=1}^{\frac{d-1}{2}}(n_1n_2)^{\frac{d-1}{2}-l}\left(p^{1-2l}+p^{-2l}\right)\right)}}.
\end{align*}
If $d$ is even, we have
\begin{align*}
P(G(n_1,n_2,p),d)&>1-\sum_{j=1}^{2}\binom{n_j}{2}h(n,p,d,1)\left(1-f(n,p,d,1)\right)^{g_b(n,n_j,p,d,d',1)}
\end{align*}
and
\begin{align*}
P(G(n_1,n_2,p),d)&<\left(\sum_{j=1}^2\binom{n_j}{2}(1-p^d)^{(n-n_j)\left((n_1n_2)^{\frac{d-2}{2}}+\sum_{l=1}^{\frac{d-2}{2}}(n_1n_2)^{\frac{d-2}{2}-l}\left(p^{1-2l}+p^{-2l}\right)\right)}\right)^{-2}\\
&\quad\cdot\left(\left(\sum_{j=1}^2\binom{n_j}{2}^2h(n-1,p,d,2)\left(1-f(n,p,d,2)\right)^{2\cdot g_b(n-1,n_j-1,p,d,d',2)}\right)\right.\\
&\qquad\quad\left.+2\binom{n_1}{2}\binom{n_2}{2}h(n,p,d,1)^2\left(1-f(n,p,d,1)\right)^{g_b'(n_2,n_1,p,d,d',1,1)+g_b'(n_1,n_2,p,d,d',1,1)}\right)\\
&\quad-1\\
&\quad+\frac{1}{\sum_{j=1}^2\binom{n_j}{2}(1-p^d)^{(n-n_j)\left((n_1n_2)^{\frac{d-2}{2}}+\sum_{l=1}^{\frac{d-2}{2}}(n_1n_2)^{\frac{d-2}{2}-l}\left(p^{1-2j}+p^{-2j}\right)\right)}}.
\end{align*}
\end{theorem}
We will now prove Theorem \ref{bigthmbipartite}.
\newline
\newline
For each $n\in\mathbb{N}$, let $G(n_1,n_2,p)$ denote the set of all bipartite graphs with partite sets of size $n_1$ vertices and $n_2$ vertices with edge probability $p$, and let $P(G(n_1,n_2,p))$ be the probability of a graph from $G(n_1,n_2,p)$ having diameter at most $d$. Let $p=\frac{r}{s}$ where $r=r(n), s=s(n)\in\mathbb{N}$. We let $A$ be the set of all graphs in $G(n_1,n_2,p)$, allowing for a number of duplicates of each possible graph to accommodate the edge probability $p$, so that
\begin{equation*}
|A|=\sum_{k=0}^{n_1n_2}\binom{n_1n_2}{k}r^k(s-r)^{n_1n_2-k}=s^{n_1n_2}.
\end{equation*}
If $d$ is odd, we let $B$ be all pairs of vertices that occur in the same partite set so $|B|=\binom{n_1}{2}+\binom{n_2}{2}$. If $d$ is even, we let $B$ be all pairs of vertices where the vertices in the pair occur in different partite sets so that $|B|=n_1n_2$. For a graph $a\in A$ and a pair of vertices $b\in B$, we say $a\sim b$ if there is no path between the pair of vertices $b$ that consists of at most $d-1$ edges. Thus, we will have $\omega(a)=0$ if and only if $a$ is connected with diameter at most $d$.
\newline
\newline
Pick a pair of vertices $b\in B$ and call them $v_1$ and $v_2$. To calculate $\deg b$, we need to calculate the number of graphs in $A$ such that there is no path from $v_1$ to $v_2$ that consists of at most $d-1$ edges. To help with this calculation, we will calculate a generalised notion of $\deg b$ as follows. First suppose that $d$ is odd. Let $0\leq i_0\leq\max\{n_1,n_2\}-1$. Pick a specific set of $i_0$ vertices out of the labeled vertices in one of the partite sets, as well as another vertex, say $v$, in the same partite set. We will let $C_b(n,n_j,r,s,d-1,i_0)$ denote the number of graphs in $A$ such that there is no path from any of the $i_0$ vertices to vertex $v$ that consists of at most $d-1$ edges where the $i_0$ vertices come from the partite set that consists of $v_j$ vertices. Now suppose that $d$ is even. Let $0\leq i_0\leq\max\{n_1,n_2\}-1$. Pick a specific set of $i_0$ vertices out of the labeled vertices in one of the partite sets, as well as another vertex, say $v$, in the opposite partite set. We will let $C_b(n,n_j,r,s,d-1,i_0)$ denote the number of graphs in $A$ such that there is no path from any of the $i_0$ vertices to vertex $v$ that consists of at most $d-1$ edges where vertex $v$ comes from the partite set that consists of $v_j$ vertices. If $d$ is even, we can derive the recursive formula
\begin{align}
C_b(n,n_j,r,s,d+1,i_0)&=(s-r)^{i_0n_j}s^{n_1n_2-i_0n_j}\nonumber\\
&\quad+\sum_{i_1=1}^{n_j-1}\binom{n_j-1}{i_1}\left(s^{i_0}-(s-r)^{i_0}\right)^{i_1}(s-r)^{i_0(n_j-i_1)}C_b(n-i_0,n_j,r,s,d,i_1)\label{eqn10c}
\end{align} 
valid for all $0\leq i_0\leq n-n_j$ and even $d\geq 2$, which can be simplified to
\begin{equation}\label{eqn10a}
C_b(n,n_j,r,s,d+1,i_0)=\sum_{i_1=0}^{n_j-1}\binom{n_j-1}{i_1}\left(s^{i_0}-(s-r)^{i_0}\right)^{i_1}(s-r)^{i_0(n_j-i_1)}C_b(n-i_0,n_j,r,s,d,i_1)
\end{equation}
if we assume that $i_0>0$. On the other hand, if $d$ is odd, we can derive the recursive formula
\begin{align}
C_b(n,n_j,r,s,d+1,i_0)&=(s-r)^{i_0(n-n_j)}s^{n_1n_2-i_0(n-n_j)}\nonumber\\
&\quad+\sum_{i_1=1}^{n-n_j}\binom{n-n_j}{i_1}\left(s^{i_0}-(s-r)^{i_0}\right)^{i_1}(s-r)^{i_0(n-n_j-i_1)}C_b(n-i_0,n_j-i_0,r,s,d,i_1)\label{eqn10d}
\end{align} 
valid for all $0\leq i_0\leq n_j-1$ and $d\geq 1$, which can be simplified to
\begin{equation}\label{eqn10b}
C_b(n,n_j,r,s,d+1,i_0)=\sum_{i_1=0}^{n-n_j}\binom{n-n_j}{i_1}\left(s^{i_0}-(s-r)^{i_0}\right)^{i_1}(s-r)^{i_0(n-n_j-i_1)}C_b(n-i_0,n_j-i_0,r,s,d,i_1).
\end{equation}
As well,
\begin{equation*}
C_b(n,n_j,r,s,1,i_0)=(s-r)^{i_0}s^{n_1n_2-i_0}
\end{equation*}
for all $0\leq i_0\leq n-n_j$, completing the formula. Then we can deduce that $C_b(n,n_j,r,s,d-1,1)=\deg b$ if we are working with diameter $d$. Let $D_b(n,n_j,p,d-1,i_0)=\frac{C_b(n,n_j,r,s,d-1,i_0)}{s^{n_1n_2}}$ so that $D_b(n,n_j,p,d-1,i_0)$ is the probability that the edge distance between $v$ and any of the $i_0$ vertices is greater than $d-1$. We will prove that if $0\leq i_0\leq n-n_j$, $0<p<1$, $d\geq 1$ is odd, then
\begin{equation}\label{lowerdegbbipartiteodd}
D_b(n,n_j,p,d,i_0)\geq(1-p^d)^{i_0\left((n_1n_2)^{\frac{d-1}{2}}+(n_1n_2)^{\frac{d-3}{2}}(p^{-1}+p^{-2})+(n_1n_2)^{\frac{d-5}{2}}(p^{-3}+p^{-4})+\ldots+(p^{-(d-2)}+p^{-(d-1)})\right)}
\end{equation}
and that under the additional constraints $i_0\leq\frac{n_j-1}{4np+(4np)^3+\ldots+(4np)^{d'}}$ and $i_0\leq\frac{n-n_j}{1+(4np)^2+\ldots+(4np)^{d'-1}}$ if $d'$ is odd, $i_0\leq\frac{n_j-1}{4np+(4np)^3+\ldots+(4np)^{d'-1}}$, $i_0\leq\frac{n-n_j}{1+(4np)^2+\ldots+(4np)^{d'}}$ if $d'$ is even and is at least $2$, or $i_0\leq n-n_j$ if  $d'=0$, then we also have
\begin{equation}
D_b(n,n_j,p,d,i_0)<h(n,p,d,i_0)\left(1-f(n,p,d,i_0)\right)^{i_0g_b(n,n_j,p,d,d',i_0)}.\label{upperdegbbipartite}
\end{equation}
Also, we will prove that if $0\leq i_0\leq n_j-1$, $0<p<1$, $d\geq 2$ is even, then
\begin{equation}\label{lowerdegbbipartiteeven}
D_b(n,n_j,p,d,i_0)\geq(1-p^d)^{i_0(n-n_j)\left((n_1n_2)^{\frac{d-2}{2}}+(n_1n_2)^{\frac{d-4}{2}}(p^{-1}+p^{-2})+(n_1n_2)^{\frac{d-6}{2}}(p^{-3}+p^{-4})+\ldots+(p^{-(d-3)}+p^{-(d-2)})\right)}
\end{equation}
and that under the additional constraints $i_0\leq\frac{n-n_j}{4np+(4np)^3+\ldots+(4np)^{d'}}$ and $i_0\leq\frac{n_j-1}{1+(4np)^2+\ldots+(4np)^{d'-1}}$ if $d'$ is odd, $i_0\leq\frac{n-n_j}{4np+(4np)^3+\ldots+(4np)^{d'-1}}$, $i_0\leq\frac{n_j-1}{1+(4np)^2+\ldots+(4np)^{d'}}$ if $d'$ is even and is at least $2$, or $i_0\leq n_j-1$ if  $d'=0$, then we also have \eqref{upperdegbbipartite}. For $d=1$, we can see that \eqref{lowerdegbbipartiteodd} holds. Suppose $d$ is odd and \eqref{lowerdegbbipartiteodd} holds for all $0\leq i_0\leq n-n_j$ and $0<p<1$ . We will prove that \eqref{lowerdegbbipartiteeven} holds for $d+1$. First, we can verify that \eqref{lowerdegbbipartiteeven} holds if $i_0=0$. For what follows let
\begin{equation*}
C_b(n_1,n_2,p,d):=(n_1n_2)^{\frac{d-1}{2}}+(n_1n_2)^{\frac{d-3}{2}}(p^{-1}+p^{-2})+(n_1n_2)^{\frac{d-5}{2}}(p^{-3}+p^{-4})+\ldots+(p^{-(d-2)}+p^{-(d-1)}).
\end{equation*}
By \eqref{eqn10b}, we have
\begin{align*}
&\quad D_b(n,n_j,p,d+1,i_0)\\
&=(1-p)^{i_0(n-n_j)}\sum_{i_1=0}^{n-n_j}\binom{n-n_j}{i_1}((1-p)^{-i_0}-1)^{i_1}D_b(n-i_0,n_j-i_0,p,d,i_1)\\
&>(1-p)^{i_0(n-n_j)}\\
&\quad\cdot\sum_{i_1=0}^{n-n_j}\binom{n-n_j}{i_1}((1-p)^{-i_0}-1)^{i_1}(1-p^d)^{i_1C_b(n_1,n_2,p,d)}\\
&=(1-p)^{i_0(n-n_j)}\left(1+\left((1-p)^{-i_0}-1\right)(1-p^d)^{C_b(n_1,n_2,p,d)}\right)^{n-n_j}\\
&=\left((1-p)^{i_0}+\left(1-(1-p)^{i_0}\right)(1-p^d)^{C_b(n_1,n_2,p,d)}\right)^{n-n_j}.
\end{align*}
Using Lemma \ref{lem1} we thus have
\begin{align*}
D_b(n,n_j,p,d+1,i_0)&>(1-p^{d+1})^{i_0(n-n_j)C_b(n_1,n_2,p,d)}.
\end{align*}
Suppose $d$ is even and \eqref{lowerdegbbipartiteeven} holds for all $0\leq i_0\leq n_j-1$ and $0<p<1$ . We will prove that \eqref{lowerdegbbipartiteodd} holds for $d+1$. First, we can verify that \eqref{lowerdegbbipartiteodd} holds if $i_0=0$. By \eqref{eqn10a}, we have
\begin{align*}
&\quad D_b(n,n_j,p,d+1,i_0)\\
&=(1-p)^{i_0n_j}\sum_{i_1=0}^{n_j-1}\binom{n_j-1}{i_1}((1-p)^{-i_0}-1)^{i_1}D_b(n-i_0,n_j,p,d,i_1)\\
&>(1-p)^{i_0n_j}\sum_{i_1=0}^{n_j-1}\binom{n_j-1}{i_1}((1-p)^{-i_0}-1)^{i_1}(1-p^d)^{i_1(n-n_j)C_b(n_1,n_2,p,d-1)}\\
&=(1-p)^{i_0n_j}\left(1+\left((1-p)^{-i_0}-1\right)(1-p^d)^{(n-n_j)C_b(n_1,n_2,p,d-1)}\right)^{n_j-1}\\
&=(1-p)^{i_0}\left((1-p)^{i_0}+\left(1-(1-p)^{i_0}\right)(1-p^d)^{(n-n_j)C_b(n_1,n_2,p,d-1)}\right)^{n_j-1}.
\end{align*}
Using Lemma \ref{lem1} we thus have
\begin{align*}
D_b(n,n_j,p,d+1,i_0)&>(1-p)^{i_0}(1-p^{d+1})^{i_0(n-n_j)C_b(n_1n_2,p,d-1)(n_j-1)}\\
&=\frac{(1-p^{d+1})^{i_0}}{(1+p+p^2+\ldots+p^d)^{i_0}}\cdot(1-p^{d+1})^{i_0(n-n_j)C_b(n_1n_2,p,d-1)(n_j-1)}\\
&>(1-p^{d+1})^{i_0}e^{(-p-p^2-\ldots-p^d)i_0}(1-p^{d+1})^{i_0(n-n_j)C_b(n_1n_2,p,d-1)(n_j-1)}\\
&>(1-p^{d+1})^{(1+p^{d-1}+p^{-2}+^{-3}+\ldots+p^{-d})i_0}(1-p^{d+1})^{i_0(n-n_j)C_b(n_1n_2,p,d-1)(n_j-1)}\\
&>(1-p^{d+1})^{i_0C_b(n_1,n_2,p,d+1)}.
\end{align*}
Thus \eqref{lowerdegbbipartiteodd} and \eqref{lowerdegbbipartiteeven} are proved. Next we prove \eqref{upperdegbbipartite} again by induction on $d$. For $d=2$, applying Lemma \ref{lem1} we have
\begin{align*}
D_b(n,n_j,p,2,i_0)&=(1-p)^{i_0(n-n_j)}\sum_{i_1=0}^{n-n_j}\binom{n-n_j}{i_1}((1-p)^{-i_0}-1)^{i_1}(1-p)^{i_1}\\
&=(1-p)^{i_0(n-n_j)}(p+(1-p)^{1-i_0})^{n-n_j}\\
&=(1-p+p(1-p)^{i_0})^{n-n_j}\\
&<\left(1-\frac{p(1-(1-p)^{i_0})}{i_0}\right)^{i_0(n-n_j)}.
\end{align*}
Suppose for some odd $d\geq 3$ \eqref{upperdegbbipartite} holds for all $i_0$ in the stated ranges, and $0<p<1$. We will prove \eqref{upperdegbbipartite} holds for $d+1$. We have
\begin{equation*}
D_b(n,n_j,p,d+1,i_0)=(1-p)^{i_0(n-n_j)}\sum_{i_1=0}^{n-n_j}\binom{n-n_j}{i_1}((1-p)^{-i_0}-1)^{i_1}D_b(n-i_0,n_j-i_0,p,d,i_1).
\end{equation*}
We divide into three cases.
\setcounter{case}{0}
\begin{case}{$\frac{n-n_j}{4np}<i_0\leq n_j-1$.}
\normalfont
\newline
\newline
We have the following:
\begin{align*}
&\quad D_b(n,n_j,p,d+1,i_0)\\
&=(1-p)^{i_0(n-n_j)}\sum_{i_1=0}^{n-n_j}\binom{n-n_j}{i_1}((1-p)^{-i_0}-1)^{i_1}D_b(n-i_0,n_j-i_0,p,d,i_1)\\
&=(1-p)^{i_0(n-n_j)}\left(1+\sum_{i_1=1}^{n-n_j}\binom{n-n_j}{i_1}((1-p)^{-i_0}-1)^{i_1}\right.\\
&\qquad\qquad\qquad\qquad\qquad\qquad\qquad\left.\left(1-f(n-i_0,p,d,i_1)\right)^{i_1g_b(n-i_0,n_j-i_0,p,d,0,i_1)}h(n-i_0,p,d,i_1)\right).
\end{align*}
We can deduce that $h(n-i_0,p,d,i_1)\leq h(n-i_0,p,d,4npi_0)\leq h(n,p,d+1,i_0)$ and from Lemma \ref{lem3}, we can deduce that $f(n,p,d,4npi_0)<f(n-i_0,p,d,i_1)$. As well, $g_b(n-i_0,n_j-i_0,p,d,n-n_j-1)\leq g_b(n-i_0,n_j-i_0,p,d,i_1)$. Thus we have
\begin{align*}
&\quad D_b(n,n_j,p,d+1,i_0)\\
&<h(n,p,d+1,i_0)(1-p)^{i_0(n-n_j)}\left(1+\sum_{i_1=1}^{n-n_j}\binom{n-n_j}{i_1}((1-p)^{-i_0}-1)^{i_1}\right.\\
&\qquad\qquad\qquad\qquad\qquad\qquad\qquad\qquad\qquad\qquad\qquad\left.\left(1-f(n,p,d,4npi_0)\right)^{i_1g_b(n-i_0,n_j-i_0,p,d,0,n-n_j)}\right)\\
&=h(n,p,d+1,i_0)(1-p)^{i_0(n-n_j)}\\
&\quad\cdot\left(1+((1-p)^{-i_0}-1)\left(1-f(n,p,d,4npi_0)\right)^{g_b(n-i_0,n_j-i_0,p,d,0,n-n_j)}\right)^{n-n_j}\\
&<h(n,p,d+1,i_0)\left((1-p)^{i_0}+(1-(1-p)^{i_0})\left(1-f(n,p,d,4npi_0)\right)^{g_b(n-i_0,n_j-i_0,p,d,0,n-n_j)}\right)^{n-n_j}.
\end{align*}
We note that $g_b(n-i_0,n_j,i_0,p,d,0,n-n_j)<n^{d-1}$ and so, using Lemma \ref{lem1}, we thus have
\begin{align*}
&\quad D_b(n,n_j,p,d+1,i_0)\\
&<h(n,p,d+1,i_0)\\
&\quad\cdot\left(1-p\left(\frac{1-(1-p)^{i_0}}{pi_0}\right)\left(\frac{1-(1-p^d)^{n^{d-1}}}{n^{d-1}p^d}\right)f(n,p,d,4npi_0)\right)^{i_0(n-n_j)g_b(n-i_0,n_j-i_0,p,d,0,n-n_j)}\\
&=h(n,p,d+1,i_0)\left(1-f(n,p,d+1,i_0)\right)^{i_0(n-n_j)g_b(n-i_0,n_j-i_0,p,d,0,n-n_j)}.
\end{align*}
We can deduce that $(n-n_j)g_b(n-i_0,n_j-i_0,p,d,0,n-n_j)>g_b(n,n_j,p,d+1,0,i_0)$ and so we have \eqref{upperdegb}. 
\end{case}
\begin{case}{$i_0\leq\frac{n-n_j}{4np}$}
\normalfont
\newline
\newline
Given a subset of $i_1$ vertices from a set of $n_1$ vertices and vertex, say $v$, from a set of $n_2$ vertices in a graph from $G(n_1,n_2,p)$, we know that $D_b(n,n_2,p,d,i_1)$ is the probability that the edge distance between $v$ and any of the $i_1$ vertices is greater than $d$ where $n=n_1+n_2$. By adding one more vertex to our set of $i_1$ vertices, it therefore follows that $D_b(n,n_2,p,d,i_1+1)\leq D_b(n,n_2,p,d,i_1)$. Thus, by Lemma \ref{4nplemma}, we have
\begin{align*}
&\quad D_b(n,n_j,p,d+1,i_0)\\
&<\left(1-\frac{4}{5}\left(\frac{e}{3}\right)^{4npi_0}\right)^{-1}(1-p)^{i_0(n-n_j)}\sum_{i_1=0}^{4npi_0}\binom{n-n_j}{i_1}((1-p)^{-i_0}-1)^{i_1}D_b(n-i_0,n_j-i_0,p,d,i_1)\\
&<\left(1-\frac{4}{5}\left(\frac{e}{3}\right)^{4npi_0}\right)^{-1}(1-p)^{i_0(n-n_j)}\\
&\quad\cdot\left(1+\sum_{i_1=1}^{4npi_0}\binom{n-n_j}{i_1}((1-p)^{-i_0}-1)^{i_1}\right.\\
&\qquad\qquad\qquad\qquad\qquad\cdot\left.\left(1-f(n-i_0,p,d,i_1)\right)^{i_1g_b(n-i_0,n_j-i_0,p,d,0,i_1)}h(n-i_0,p,d,i_1)\right).
\end{align*}
We can deduce that $h(n-i_0,p,d,i_1)\leq h(n-i_0,p,d,4npi_0)$ and from Lemma \ref{lem3}, we can deduce that $f(n,p,d,4npi_0)<f(n-i_0,p,d,i_1)$. As well, $g_b(n-i_0,n_j-i_0,p,d,0,n-n_j)\leq g_b(n-i_0,n_j-i_0,p,d,0,i_1)$. Thus we have
\begin{align*}
&\quad D_b(n,n_j,p,d+1,i_0)\\
&<\left(1-\frac{4}{5}\left(\frac{e}{3}\right)^{4npi_0}\right)^{-1}(1-p)^{i_0(n-n_j)}\\
&\quad\cdot\left(1+\sum_{i_1=1}^{4npi_0}\binom{n-n_j}{i_1}((1-p)^{-i_0}-1)^{i_1}\right.\\
&\left.\qquad\qquad\qquad\qquad\qquad\cdot\left(1-f(n,p,d,4npi_0)\right)^{i_1g_b(n-i_0,n_j-i_0,p,d,0,n-n_j)}h(n,p,d,4npi_0)\right)\\
&<\left(1-\frac{4}{5}\left(\frac{e}{3}\right)^{4npi_0}\right)^{-1}h(n-i_0,p,d,4npi_0)(1-p)^{i_0(n-n_j)}\\
&\quad\cdot\left(1+\sum_{i_1=1}^{n-n_j}\binom{n-n_j}{i_1}((1-p)^{-i_0}-1)^{i_1}\left(1-f(n,p,d,4npi_0)\right)^{i_1g_b(n-i_0,n_j-i_0,p,d,0,n-n_j)}\right)\\
&<h(n,p,d+1,i_0)(1-p)^{i_0(n-n_j)}\\
&\quad\cdot\left(1+\left((1-p)^{-i_0}-1\right)\left(1-f(n,p,d,4npi_0)\right)^{g_b(n-i_0,n_j-i_0,p,d,0,n-n_j)}\right)^{n-n_j}\\
&<h(n,p,d+1,i_0)\left((1-p)^{i_0}+\left(1-(1-p)^{i_0}\right)\left(1-f(n,p,d,4npi_0)\right)^{g_b(n-i_0,n_j-i_0,p,d,0,n-n_j)}\right)^{n-n_j}.
\end{align*}
We note that $g_b(n-i_0,n_j-i_0,p,d,0,n-n_j)<n^{d-1}$ and so, using Lemma \ref{lem1}, we thus have
\begin{align*}
&\quad D_b(n,n_j,p,d+1,i_0)\\
&<h(n,p,d+1,i_0)\\
&\quad\cdot\left(1-p\left(\frac{1-(1-p)^{i_0}}{pi_0}\right)\left(\frac{1-(1-p^d)^{n^{d-1}}}{n^{d-1}p^d}\right)f(n,p,d,4npi_0)\right)^{i_0(n-n_j)g_b(n-i_0,n_j-i_0,p,d,0,n-n_j)}\\
&=h(n,p,d+1,i_0)\left(1-f(n,p,d+1,i_0)\right)^{i_0(n-n_j)g_b(n-i_0,n_j-i_0,p,d,0,n-n_j)}.
\end{align*}
We can deduce that $(n-n_j)g_b(n-i_0,n_j-i_0,p,d,0,n-n_j)>g_b(n,n_j,p,d+1,0,i_0)$ and so we have \eqref{upperdegb}.
\end{case}
\begin{case}{$i_0\leq\frac{n-n_j}{4np+(4np)^3+\ldots+(4np)^{d'}}$, $i_0\leq\frac{n_j-1}{1+(4np)^2+\ldots+(4np)^{d'-1}}$ OR $i_0\leq\frac{n-n_j}{4np+(4np)^3+\ldots+(4np)^{d'-1}}$, $i_0\leq\frac{n_j-1}{1+(4np)^2+\ldots+(4np)^{d'}}$}
\normalfont
\newline
\newline
Given a subset of $i_1$ vertices from a set of $n_1$ vertices and vertex, say $v$, from a set of $n_2$ vertices in a graph from $G(n_1,n_2,p)$, we know that $D_b(n,n_2,p,d,i_1)$ is the probability that the edge distance between $v$ and any of the $i_1$ vertices is greater than $d$ where $n=n_1+n_2$. By adding one more vertex to our set of $i_1$ vertices, it therefore follows that $D_b(n,n_2,p,d,i_1+1)\leq D_b(n,n_2,p,d,i_1)$. Thus, by Lemma \ref{4nplemma}, we have
\begin{align*}
&\quad D_b(n,n_j,p,d+1,i_0)\\
&<\left(1-\frac{4}{5}\left(\frac{e}{3}\right)^{4npi_0}\right)^{-1}(1-p)^{i_0(n-n_j)}\sum_{i_1=0}^{4npi_0}\binom{n-n_j}{i_1}((1-p)^{-i_0}-1)^{i_1}D_b(n-i_0,n_j-i_0,p,d,i_1)\\
&<\left(1-\frac{4}{5}\left(\frac{e}{3}\right)^{4npi_0}\right)^{-1}(1-p)^{i_0(n-n_j)}\\
&\quad\cdot\left(1+\sum_{i_1=1}^{4npi_0}\binom{n-n_j}{i_1}((1-p)^{-i_0}-1)^{i_1}\right.\\
&\qquad\qquad\qquad\qquad\qquad\cdot\left.\left(1-f(n-i_0,p,d,i_1)\right)^{i_1g_b(n-i_0,n_j-i_0,p,d,d',i_1)}h(n-i_0,p,d,i_1)\right).
\end{align*}
We can deduce that $h(n-i_0,p,d,i_1)\leq h(n-i_0,p,d,4npi_0)$ and from Lemma \ref{lem3}, we can deduce that $f(n,p,d,4npi_0)<f(n-i_0,p,d,i_1)$. As well, $g_b(n-i_0,n_j-i_0,p,d,d',4npi_0)\leq g_b(n-i_0,n_j-i_0,p,d',d',i_1)$. Thus we have
\begin{align*}
&\quad D_b(n,n_j,p,d+1,i_0)\\
&<\left(1-\frac{4}{5}\left(\frac{e}{3}\right)^{4npi_0}\right)^{-1}(1-p)^{i_0(n-n_j)}\\
&\quad\cdot\left(1+\sum_{i_1=1}^{4npi_0}\binom{n-n_j}{i_1}((1-p)^{-i_0}-1)^{i_1}\right.\\
&\left.\qquad\qquad\qquad\qquad\qquad\cdot\left(1-f(n,p,d,4npi_0)\right)^{i_1g_b(n-i_0,n_j-i_0,p,d,d',4npi_0)}h(n,p,d,4npi_0)\right)\\
&<\left(1-\frac{4}{5}\left(\frac{e}{3}\right)^{4npi_0}\right)^{-1}h(n-i_0,p,d,4npi_0)(1-p)^{i_0(n-n_j)}\\
&\quad\cdot\left(1+\sum_{i_1=1}^{n-n_j}\binom{n-n_j}{i_1}((1-p)^{-i_0}-1)^{i_1}\left(1-f(n,p,d,4npi_0)\right)^{i_1g_b(n-i_0,n_j-i_0,p,d,d',4npi_0)}\right)\\
&<h(n,p,d+1,i_0)(1-p)^{i_0(n-n_j)}\\
&\quad\cdot\left(1+\left((1-p)^{-i_0}-1\right)\left(1-f(n,p,d,4npi_0)\right)^{g_b(n-i_0,n_j-i_0,p,d,d',4npi_0)}\right)^{n-n_j}\\
&<h(n,p,d+1,i_0)\left((1-p)^{i_0}+\left(1-(1-p)^{i_0}\right)\left(1-f(n,p,d,4npi_0)\right)^{g_b(n-i_0,n_j-i_0,p,d,d',4npi_0)}\right)^{n-n_j}.
\end{align*}
We note that $g_b(n-i_0,n_j-i_0,p,d,0,n-n_j)<n^{d-1}$ and so, using Lemma \ref{lem1}, we thus have
\begin{align*}
&\quad D_b(n,n_j,p,d+1,i_0)\\
&<h(n,p,d+1,i_0)\\
&\quad\cdot\left(1-p\left(\frac{1-(1-p)^{i_0}}{pi_0}\right)\left(\frac{1-(1-p^d)^{n^{d-1}}}{n^{d-1}p^d}\right)f(n,p,d,4npi_0)\right)^{i_0(n-n_j)g_b(n-i_0,n_j-i_0,p,d,d',4npi_0)}\\
&=h(n,p,d+1,i_0)\left(1-f(n,p,d+1,i_0)\right)^{i_0(n-n_j)g_b(n-i_0,n_j-i_0,p,d,d',4npi_0)}.
\end{align*}
We can deduce that $(n-n_j)g_b(n-i_0,n_j-i_0,p,d,d',4npi_0)>g_b(n,n_j,p,d+1,d'+1,i_0)$ and so we have \eqref{upperdegb}.
\end{case}
Suppose for some even $d\geq 2$ \eqref{upperdegbbipartite} holds for all $i_0$ in the stated ranges, and $0<p<1$. We will prove \eqref{upperdegbbipartite} holds for $d+1$. We have
\begin{equation*}
D_b(n,n_j,p,d+1,i_0)=(1-p)^{i_0n_j}\sum_{i_1=0}^{n_j-1}\binom{n_j-1}{i_1}((1-p)^{-i_0}-1)^{i_1}D_b(n-i_0,n_j,p,d,i_1).
\end{equation*}
We divide into three cases.
\setcounter{case}{0}
\begin{case}{$\frac{n_j-1}{4np}<i_0\leq n-n_j$.}
\normalfont
\newline
\newline
We have the following:
\begin{align*}
&\quad D_b(n,n_j,p,d+1,i_0)\\
&=(1-p)^{i_0n_j}\sum_{i_1=0}^{n_j-1}\binom{n_j-1}{i_1}((1-p)^{-i_0}-1)^{i_1}D_b(n-i_0,n_j,p,d,i_1)\\
&=(1-p)^{i_0n_j}\left(1+\sum_{i_1=1}^{n_j-1}\binom{n_j-1}{i_1}((1-p)^{-i_0}-1)^{i_1}\right.\\
&\qquad\qquad\qquad\qquad\qquad\qquad\qquad\left.\left(1-f(n-i_0,p,d,i_1)\right)^{i_1g_b(n-i_0,n_j,p,d,0,i_1)}h(n-i_0,p,d,i_1)\right).
\end{align*}
We can deduce that $h(n-i_0,p,d,i_1)\leq h(n-i_0,p,d,4npi_0)<h(n-i_0,p,d+1,i_0)$ and from Lemma \ref{lem3}, we can deduce that $f(n,p,d,4npi_0)<f(n-i_0,p,d,i_1)$. As well, $g_b(n-i_0,n_j,p,d,n_j-1)\leq g_b(n-i_0,n_j,p,d,i_1)$. Thus we have
\begin{align*}
&\quad D_b(n,n_j,p,d+1,i_0)\\
&<h(n,p,d+1,i_0)(1-p)^{i_0n_j}\left(1+\sum_{i_1=1}^{n_j-1}\binom{n-n_j}{i_1}((1-p)^{-i_0}-1)^{i_1}\right.\\
&\qquad\qquad\qquad\qquad\qquad\qquad\qquad\qquad\qquad\qquad\qquad\left.\left(1-f(n,p,d,4npi_0)\right)^{i_1g_b(n-i_0,n_j,p,d,0,n_j-1)}\right)\\
&=h(n,p,d+1,i_0)(1-p)^{i_0n_j}\\
&\quad\cdot\left(1+((1-p)^{-i_0}-1)\left(1-f(n,p,d,4npi_0)\right)^{g_b(n-i_0,n_j,p,d,0,n_j-1)}\right)^{n_j-1}\\
&<h(n,p,d+1,i_0)(1-p)^{i_0}\\
&\quad\cdot\left((1-p)^{i_0}+(1-(1-p)^{i_0})\left(1-f(n,p,d,4npi_0)\right)^{g_b(n-i_0,n_j,p,d,0,n_j-1)}\right)^{n_j-1}.
\end{align*}
We note that $g_b(n-i_0,n_j,i_0,p,d,0,n_j-1)<n^{d-1}$ and so, using Lemma \ref{lem1}, we thus have
\begin{align*}
&\quad D_b(n,n_j,p,d+1,i_0)\\
&<h(n,p,d+1,i_0)\\
&\quad\cdot\left(1-p\left(\frac{1-(1-p)^{i_0}}{pi_0}\right)\left(\frac{1-(1-p^d)^{n^{d-1}}}{n^{d-1}p^d}\right)f(n,p,d,4npi_0)\right)^{i_0(n_j-1)g_b(n-i_0,n_j,p,d,0,n_j-1)+i_0}\\
&=h(n,p,d+1,i_0)\left(1-f(n,p,d+1,i_0)\right)^{i_0(n_j-1)g_b(n-i_0,n_j,p,d,0,n_j-1)+i_0}.
\end{align*}
We can deduce that $(n_j-1)g_b(n-i_0,n_j,p,d,0,n_j-1)+1>g_b(n,n_j,p,d+1,0,i_0)$ and so we have \eqref{upperdegb}. 
\end{case}
\begin{case}{$i_0\leq\frac{n_j-1}{4np}$}
\normalfont
\newline
\newline
Given a set of $i_1$ vertices and one additional vertex, say $v$, in a graph from $G(n-i_0,p)$, we know that $D_b(n-i_0,p,d,i_1)$ is the probability that the edge distance between $v$ and any of the $i_1$ vertices is greater than $d$. By adding one more vertex to our set of $i_1$ vertices, it therefore follows that $D_b(n-i_0,p,d,i_1+1)\leq D_b(n-i_0,p,d,i_1)$. Thus, by Lemma \ref{4nplemma}, we have
\begin{align*}
&\quad D_b(n,n_j,p,d+1,i_0)\\
&<\left(1-\frac{4}{5}\left(\frac{e}{3}\right)^{4npi_0}\right)^{-1}(1-p)^{i_0n_j}\sum_{i_1=0}^{4npi_0}\binom{n_j-1}{i_1}((1-p)^{-i_0}-1)^{i_1}D_b(n-i_0,n_j-i_0,p,d,i_1)\\
&<\left(1-\frac{4}{5}\left(\frac{e}{3}\right)^{4npi_0}\right)^{-1}(1-p)^{i_0n_j}\\
&\quad\cdot\left(1+\sum_{i_1=1}^{4npi_0}\binom{n_j-1}{i_1}((1-p)^{-i_0}-1)^{i_1}\right.\\
&\qquad\qquad\qquad\qquad\qquad\cdot\left.\left(1-f(n-i_0,p,d,i_1)\right)^{i_1g_b(n-i_0,n_j,p,d,0,i_1)}h(n-i_0,p,d,i_1)\right).
\end{align*}
We can deduce that $h(n-i_0,p,d,i_1)\leq h(n-i_0,p,d,4npi_0)$ and from Lemma \ref{lem3}, we can deduce that $f(n,p,d,4npi_0)<f(n-i_0,p,d,i_1)$. As well, $g_b(n-i_0,n_j-i_0,p,d,0,n-n_j)\leq g_b(n-i_0,n_j-i_0,p,d,0,i_1)$. Thus we have
\begin{align*}
&\quad D_b(n,n_j,p,d+1,i_0)\\
&<\left(1-\frac{4}{5}\left(\frac{e}{3}\right)^{4npi_0}\right)^{-1}h(n-i_0,p,d,4npi_0)(1-p)^{i_0n_j}\\
&\quad\cdot\left(1+\sum_{i_1=1}^{n_j-1}\binom{n_j-1}{i_1}((1-p)^{-i_0}-1)^{i_1}\left(1-f(n,p,d,4npi_0)\right)^{i_1g_b(n-i_0,n_j,p,d,0,n_j-1)}\right)\\
&<h(n,p,d+1,i_0)(1-p)^{i_0n_j}\\
&\quad\cdot\left(1+\left((1-p)^{-i_0}-1\right)\left(1-f(n,p,d,4npi_0)\right)^{g_b(n-i_0,n_j,p,d,0,n_j-1)}\right)^{n_j-1}\\
&<h(n,p,d+1,i_0)(1-p)^{i_0}\\
&\quad\cdot\left((1-p)^{i_0}+\left(1-(1-p)^{i_0}\right)\left(1-f(n,p,d,4npi_0)\right)^{g_b(n-i_0,n_j,p,d,0,n_j-1)}\right)^{n_j-1}.
\end{align*}
We note that $g_b(n-i_0,n_j,p,d,0,n_j-1)<n^{d-1}$ and so, using Lemma \ref{lem1}, we thus have
\begin{align*}
&\quad D_b(n,n_j,p,d+1,i_0)\\
&<h(n,p,d+1,i_0)\\
&\quad\cdot\left(1-p\left(\frac{1-(1-p)^{i_0}}{pi_0}\right)\left(\frac{1-(1-p^d)^{n^{d-1}}}{n^{d-1}p^d}\right)f(n,p,d,4npi_0)\right)^{i_0n_jg_b(n-i_0,n_j,p,d,0,n_j-1)+i_0}\\
&=h(n,p,d+1,i_0)\left(1-f(n,p,d+1,i_0)\right)^{i_0n_jg_b(n-i_0,n_j,p,d,0,n_j-1)}.
\end{align*}
We can deduce that $(n_j-1)g_b(n-i_0,n_j,p,d,0,n_j-1)+1>g_b(n,n_j,p,d+1,0,i_0)$ and so we have \eqref{upperdegb}.
\end{case}
\begin{case}{$i_0\leq\frac{n_j-1}{4np+(4np)^3+\ldots+(4np)^{d'}}$, $i_0\leq\frac{n-n_j}{1+(4np)^2+\ldots+(4np)^{d'-1}}$ OR $i_0\leq\frac{n_j-1}{4np+(4np)^3+\ldots+(4np)^{d'-1}}$, $i_0\leq\frac{n-n_j}{1+(4np)^2+\ldots+(4np)^{d'}}$}
\normalfont
\newline
\newline
Given a set of $i_1$ vertices and one additional vertex, say $v$, in a graph from $G(n-i_0,p)$, we know that $D_b(n-i_0,p,d,i_1)$ is the probability that the edge distance between $v$ and any of the $i_1$ vertices is greater than $d$. By adding one more vertex to our set of $i_1$ vertices, it therefore follows that $D_b(n-i_0,p,d,i_1+1)\leq D_b(n-i_0,p,d,i_1)$. Thus, by Lemma \ref{4nplemma}, we have
\begin{align*}
&\quad D_b(n,n_j,p,d+1,i_0)\\
&<\left(1-\frac{4}{5}\left(\frac{e}{3}\right)^{4npi_0}\right)^{-1}(1-p)^{i_0n_j}\sum_{i_1=0}^{4npi_0}\binom{n_j-1}{i_1}((1-p)^{-i_0}-1)^{i_1}D_b(n-i_0,n_j-i_0,p,d,i_1)\\
&<\left(1-\frac{4}{5}\left(\frac{e}{3}\right)^{4npi_0}\right)^{-1}(1-p)^{i_0n_j}\\
&\quad\cdot\left(1+\sum_{i_1=1}^{4npi_0}\binom{n_j-1}{i_1}((1-p)^{-i_0}-1)^{i_1}\right.\\
&\qquad\qquad\qquad\qquad\qquad\cdot\left.\left(1-f(n-i_0,p,d,i_1)\right)^{i_1g_b(n-i_0,n_j,p,d,d',i_1)}h(n-i_0,p,d,i_1)\right).
\end{align*}
We can deduce that $h(n-i_0,p,d,i_1)\leq h(n-i_0,p,d,4npi_0)$ and from Lemma \ref{lem3}, we can deduce that $f(n,p,d,4npi_0)<f(n-i_0,p,d,i_1)$. As well, $g_b(n-i_0,n_j-i_0,p,d,d',4npi_0)\leq g_b(n-i_0,n_j-i_0,p,d,d',i_1)$. Thus we have
\begin{align*}
&\quad D_b(n,n_j,p,d+1,i_0)\\
&<\left(1-\frac{4}{5}\left(\frac{e}{3}\right)^{4npi_0}\right)^{-1}h(n-i_0,p,d,4npi_0)(1-p)^{i_0n_j}\\
&\quad\cdot\left(1+\sum_{i_1=1}^{n_j-1}\binom{n_j-1}{i_1}((1-p)^{-i_0}-1)^{i_1}\left(1-f(n,p,d,4npi_0)\right)^{i_1g_b(n-i_0,n_j,p,d,d',4npi_0)}\right)\\
&<h(n,p,d+1,i_0)(1-p)^{i_0n_j}\\
&\quad\cdot\left(1+\left((1-p)^{-i_0}-1\right)\left(1-f(n,p,d,4npi_0)\right)^{g_b(n-i_0,n_j,p,d,d',4npi_0)}\right)^{n_j-1}\\
&<h(n,p,d+1,i_0)(1-p)^{i_0}\\
&\quad\cdot\left((1-p)^{i_0}+\left(1-(1-p)^{i_0}\right)\left(1-f(n,p,d,4npi_0)\right)^{g_b(n-i_0,n_j,p,d,d',4npi_0)}\right)^{n_j-1}.
\end{align*}
We note that $g_b(n-i_0,n_j,p,d,0,n_j-1)<n^{d-1}$ and so, using Lemma \ref{lem1}, we thus have
\begin{align*}
&\quad D_b(n,n_j,p,d+1,i_0)\\
&<h(n,p,d+1,i_0)\\
&\quad\cdot\left(1-p\left(\frac{1-(1-p)^{i_0}}{pi_0}\right)\left(\frac{1-(1-p^d)^{n^{d-1}}}{n^{d-1}p^d}\right)f(n,p,d,4npi_0)\right)^{i_0n_jg_b(n-i_0,n_j,p,d,d',4npi_0)}\\
&=h(n,p,d+1,i_0)\left(1-f(n,p,d+1,i_0)\right)^{i_0n_jg_b(n-i_0,n_j,p,d,d',4npi_0)+i_0}.
\end{align*}
We can deduce that $(n_j-1)g_b(n-i_0,n_j,p,d,d',4npi_0)+1>g_b(n,n_j,p,d+1,d'+1,i_0)$ and so we have \eqref{upperdegb}.
\end{case}
By \eqref{upperdegb}, we have
\begin{equation*}
\sum_{b\in B}\deg b<s^{n_1n_2}\sum_{j=1}^{2}\binom{n_j}{2}h(n,p,d,1)\left(1-f(n,p,d,1)\right)^{g_b(n,n_j,p,d,d',1)}
\end{equation*}
if $d$ is even, and
\begin{align*}
\sum_{b\in B}\deg b&<s^{n_1n_2}n_1n_2h(n,p,d,1)\left(1-f(n,p,d,1)\right)^{g_b(n,n_1,p,d,d',1)}
\end{align*}
if $d$ is odd.
Hence, by the simple sieve, we have
\begin{align*}
P(G(n_1,n_2,p),d)&>1-\sum_{j=1}^{2}\binom{n_j}{2}h(n,p,d,1)\left(1-f(n,p,d,1)\right)^{g_b(n,n_j,p,d,d',1)}
\end{align*}
if $d$ is even, and
\begin{align*}
P(G(n_1,n_2,p),d)&>1-n_1n_2h(n,p,d,1)\left(1-f(n,p,d,1)\right)^{g_b(n,n_1,p,d,d',1)}
\end{align*}
if $d$ is odd.
\newline
\newline
We now calculate $n(b_1,b_2)$ to get an upperbound for $\sum_{i=1}^{d}P(G(n_1,n_2,p),i)$ using the Tur\'an sieve. If the two pairs of vertices $b_1$ and $b_2$ are the same, then we just have $n(b_1,b_2)=\deg b$. If $b_1$ and $b_2$ have exactly one vertex in common, then we can see that $n(b_1,b_2)=C_b(n,n_j,r,s,d,2)$ and use \eqref{upperdegbbipartite}. Hence the only question is when the two pairs of vertices are disjoint.
\newline
\newline
As in our calculations for $\deg b$, to help calculate $n(b_1,b_2)$ in this case, we will calculate a generalised notion of $n(b_1,b_2)$ as follows. Suppose that $d$ is even. Let $0\leq i_0\leq n_j-2$ and $0\leq i_0'\leq n_j-2$ where $i_0+i_0'\leq n_j-2$. Pick two disjoint sets of vertices having $i_0$ and $i_0'$ vertices out of the $n_j$ labeled vertices in one of the partite sets, as well as two other vertices, say $v$ and $v'$, out of the same set. Suppose that $d$ is odd. Let $0\leq i_0\leq n-n_j$ and $0\leq i_0'\leq n-n_j$ where $i_0+i_0'\leq n-n_j$. Pick two disjoint sets of vertices having $i_0$ and $i_0'$ vertices out of the $n-n_j$ labeled vertices in one of the partite sets, as well as two other vertices, say $v$ and $v'$, out of the opposite partite set consisting of $n_j$ vertices. In both cases, we will let $C_b'(n,n_j,r,s,d,i_0,i_0')$ denote the number of graphs in $A$ such that there is no path from any of the $i_0$ vertices to vertex $v$ that consists of at most $d$ edges, as well as the requirement that there is no path from any of the $i_0'$ vertices to the vertex $v'$ that consists of at most $d$ edges (note that this does not generalise the construction where $d$ is even and $b_1$ is a pair of vertices from the set of $n_1$ (or $n_2$) vertices and $b_2$ is a pair of vertices from the set of $n_2$ (or $n_1$ respectively) vertices, we will return to this case later). If $i_0=0$, then we have $C_b'(n,n_j,r,s,d,i_0,i_0')=C_b(n,n_j,r,s,d,i_0')$ and if $i_0'=0$, then we have $C_b'(n,n_j,r,s,d,i_0,i_0')=C_b(n,n_j,r,s,d,i_0)$. So suppose that $i_0,i_0'>0$. If $d$ is odd, then we have
\begin{align}
C_b'(n,n_j,r,s,d+1,i_0,i_0')&<\sum_{i_1=0}^{n-n_j}\binom{n-n_j}{i_1}\left(s^{i_0}-(s-r)^{i_0}\right)^{i_1}(s-r)^{i_0(n-n_j-i_1)}\nonumber\\
&\quad\cdot\sum_{i_1'=0}^{n-n_j-i_1}\binom{n-n_j-i_1}{i_1'}\left(s^{i_0'}-(s-r)^{i_0'}\right)^{i_1'}(s-r)^{i_0'(n-n_j-i_1-i_1')}s^{i_1i_0'}\nonumber\\
&\quad\cdot C_b'(n-i_0-i_0',n_j-i_0-i_0',r,s,d,i_1,i_1')\label{eqn27}
\end{align} 
valid for all $1\leq i_0,i_0'\leq n_j-3$ with $i_0+i_0'\leq n_j-2$. If $d$ is even, then we have
\begin{align}
C_b'(n,n_j,r,s,d+1,i_0,i_0')&<\sum_{i_1=0}^{n_j-2}\binom{n_j-2}{i_1}\left(s^{i_0}-(s-r)^{i_0}\right)^{i_1}(s-r)^{i_0(n_j-i_1-1)}s^{i_0}\nonumber\\
&\quad\cdot\sum_{i_1'=0}^{n_j-2-i_1}\binom{n_j-2-i_1}{i_1'}\left(s^{i_0'}-(s-r)^{i_0'}\right)^{i_1'}(s-r)^{i_0'(n_j-i_1-i_1'-1)}s^{i_1i_0'+i_0'}\nonumber\\
&\quad\cdot C_b'(n-i_0-i_0',n_j,r,s,d,i_1,i_1')\label{eqn28}
\end{align} 
valid for all $1\leq i_0,i_0'\leq n-n_j-1$ with $i_0+i_0'\leq n-n_j$. As well,
\begin{equation*}
C_b'(n,n_j,r,s,1,i_0,i_0')=(s-r)^{i_0+i_0'}s^{n_1n_2-i_0-i_0'}
\end{equation*}
for all $0\leq i_0,i_0'\leq n-n_j$ with $i_0+i_0'\leq n-n_j$, completing the formula. Then we can deduce that $C_b(n,n_j,r,s,d,1,1)=n(b_1,b_2)$ if we are working with diameter $d$. Let $D_b'(n,p,d,i_0,i_0')=\frac{C_b'(n,r,s,1,i_0,i_0')}{s^{n_1n_2}}$ so that $D_b'(n,n_j,p,d,i_0,i_0')$ is the probability that the edge distance between $v$ and any of the $i_0$ vertices is greater than $d$ and that the edge distance between $v'$ and any of the $i_0'$ vertices is greater than $d$. We will prove that
\begin{equation}\label{D'Dbipartite}
D_b'(n,n_j,p,d,i_0,i_0')\leq D_b(n-1,n_j-1,p,d,i_0+i_0')
\end{equation}
for $0<p<1$, $0\leq i_0,i_0'\leq n_j-2$, $i_0+i_0'\leq n_j-2$ if we assume that $d\geq 2$ is even, and for $0<p<1$, $0\leq i_0,i_0'\leq n-n_j$, $i_0+i_0'\leq n-n_j$ if we assume that $d\geq 1$ is odd. For $d=1$, we have
\begin{align*}
D_b'(n,n_j,p,1,i_0,i_0')=(1-p)^{i_0+i_0'}=D_b(n-1,n_j-1,p,1,i_0+i_0')
\end{align*}
so \eqref{D'Dbipartite} holds for $d=1$. Suppose for some odd $d\geq 1$ \eqref{D'Dbipartite} holds for all $n\in\mathbb{N}$, $0\leq i_0,i_0'\leq n-n_j$, $i_0+i_0'\leq n-n_j$, $0<p<1$. We can see that \eqref{D'Dbipartite} holds if $i_0=0$ or $i_0'=0$. So assume that $0<i_0,i_0'\leq n_j-3$ with $i_0+i_0'\leq n_j-2$. First we have
\begin{align*}
D_b'(n,n_j,p,d+1,i_0,i_0')&<\sum_{i_1=0}^{n-n_j}\binom{n-n_j}{i_1}\left(1-(1-p)^{i_0}\right)^{i_1}(1-p)^{i_0(n-n_j-i_1)}\\
&\quad\cdot\sum_{i_1'=0}^{n-n_j-i_1}\binom{n-n_j-i_1}{i_1'}\left(1-(1-p)^{i_0'}\right)^{i_1'}(1-p)^{i_0'(n-n_j-i_1-i_1')}\\
&\quad\cdot D_b'(n-i_0-i_0',n_j-i_0-i_0',p,d,i_1,i_1')\\
&<(1-p)^{i_0(n-n_j)}\sum_{i_1=0}^{n-n_j}\binom{n-n_j}{i_1}\left((1-p)^{-i_0}-1\right)^{i_1}\\
&\quad\cdot(1-p)^{i_0'(n-n_j-i_1)}\sum_{i_1'=0}^{n-n_j-i_1}\binom{n-n_j-i_1}{i_1'}\left((1-p)^{-i_0'}-1\right)^{i_1'}\\
&\quad\cdot D_b(n-1-i_0-i_0',n_j-1-i_0-i_0',p,d,i_1+i_1').
\end{align*}
Writing $k=i_1+i_1'$, we have
\begin{align*}
D_b'(n,p,d+1,i_0,i_0')&<(1-p)^{(i_0+i_0')(n-n_j)}\sum_{k=0}^{n-n_j}\binom{n-n_j}{k}D_b(n-1-i_0-i_0',n_j-1-i_0-i_0',p,d,k)\\
&\quad\cdot\left((1-p)^{-i_0'}-1\right)^{k}\sum_{i_1=0}^k\binom{k}{i_1}\left((1-p)^{-i_0}-1\right)^{i_1}(1-p)^{-i_1i_0'}\left((1-p)^{-i_0'}-1\right)^{-i_1}\\
&=(1-p)^{(i_0+i_0')(n-n_j)}\sum_{k=0}^{n-n_j}\binom{n-n_j}{k}D_b(n-1-i_0-i_0',n_j-1-i_0-i_0',p,d,k)\\
&\quad\cdot\left((1-p)^{-i_0'}-1\right)^{k}\sum_{i_1=0}^k\binom{k}{i_1}\left(\frac{(1-p)^{-i_0-i_0'}-(1-p)^{-i_0'}}{(1-p)^{-i_0'}-1}\right)^{i_1}\\
&=(1-p)^{(i_0+i_0')(n-n_j)}\sum_{k=0}^{n-n_j}\binom{n-n_j}{k}D_b(n-1-i_0-i_0',n_j-1-i_0-i_0',p,d,k)\\
&\quad\cdot\left((1-p)^{-i_0'}-1\right)^{k}\left(1+\frac{(1-p)^{-i_0-i_0'}-(1-p)^{-i_0'}}{(1-p)^{-i_0'}-1}\right)^{k}\\
&=(1-p)^{(i_0+i_0')(n-n_j)}\sum_{k=0}^{n-n_j}\binom{n-n_j}{k}D_b(n-1-i_0-i_0',n_j-1-i_0-i_0',p,d,k)\\
&\quad\cdot\left((1-p)^{-i_0-i_0'}-1\right)^{k}\\
&=D_b(n-1,n_j-1,p,d+1,i_0+i_0').
\end{align*}
Thus \eqref{D'Dbipartite} holds for $d+1$. Suppose for some even $d\geq 2$ \eqref{D'Dbipartite} holds for all $n\in\mathbb{N}$, $0\leq i_0,i_0'\leq n_j-2$, $i_0+i_0'\leq n-n_j-1$, $0<p<1$. Assume that $0<i_0,i_0'\leq n-n_j-1$ with $i_0+i_0'\leq n-n_j$. First we have
\begin{align*}
D_b'(n,n_j,p,d+1,i_0,i_0')&<\sum_{i_1=0}^{n_j-2}\binom{n_j-2}{i_1}\left(1-(1-p)^{i_0}\right)^{i_1}(1-p)^{i_0(n_j-i_1-1)}\\
&\quad\cdot\sum_{i_1'=0}^{n_j-2-i_1}\binom{n_j-2-i_1}{i_1'}\left(1-(1-p)^{i_0'}\right)^{i_1'}(1-p)^{i_0'(n_j-i_1-i_1'-1)}\\
&\quad\cdot D_b'(n-i_0-i_0',n_j,p,d,i_1,i_1')\\
&<(1-p)^{i_0(n_j-1)}\sum_{i_1=0}^{n_j-2}\binom{n_j-2}{i_1}\left((1-p)^{-i_0}-1\right)^{i_1}\\
&\quad\cdot(1-p)^{i_0'(n_j-i_1-1)}\sum_{i_1'=0}^{n_j-2-i_1}\binom{n_j-2-i_1}{i_1'}\left((1-p)^{-i_0'}-1\right)^{i_1'}\\
&\quad\cdot D_b(n-1-i_0-i_0',n_j-1,p,d,i_1+i_1').
\end{align*}
Writing $k=i_1+i_1'$, we have
\begin{align*}
D_b'(n,n_j,p,d+1,i_0,i_0')&<(1-p)^{(i_0+i_0')(n_j-1)}\sum_{k=0}^{n_j-2}\binom{n_j-2}{k}D_b(n-1-i_0-i_0',n_j-1,p,d,k)\\
&\quad\cdot\left((1-p)^{-i_0'}-1\right)^{k}\sum_{i_1=0}^k\binom{k}{i_1}\left((1-p)^{-i_0}-1\right)^{i_1}(1-p)^{-i_1i_0'}\left((1-p)^{-i_0'}-1\right)^{-i_1}\\
&=(1-p)^{(i_0+i_0')(n_j-1)}\sum_{k=0}^{n_j-2}\binom{n_j-2}{k}D_b(n-1-i_0-i_0',n_j-1,p,d,k)\\
&\quad\cdot\left((1-p)^{-i_0'}-1\right)^{k}\sum_{i_1=0}^k\binom{k}{i_1}\left(\frac{(1-p)^{-i_0-i_0'}-(1-p)^{-i_0'}}{(1-p)^{-i_0'}-1}\right)^{i_1}\\
&=(1-p)^{(i_0+i_0')(n_j-1)}\sum_{k=0}^{n_j-2}\binom{n_j-2}{k}D_b(n-1-i_0-i_0',n_j-1,p,d,k)\\
&\quad\cdot\left((1-p)^{-i_0'}-1\right)^{k}\left(1+\frac{(1-p)^{-i_0-i_0'}-(1-p)^{-i_0'}}{(1-p)^{-i_0'}-1}\right)^{k}\\
&=(1-p)^{(i_0+i_0')(n_j-1)}\sum_{k=0}^{n_j-2}\binom{n_j-2}{k}D_b(n-1-i_0-i_0',n_j-1,p,d,k)\\
&\quad\cdot\left((1-p)^{-i_0-i_0'}-1\right)^{k}\\
&<D_b(n-1,n_j-1,p,d+1,i_0+i_0').
\end{align*}
Thus \eqref{D'Dbipartite} holds for $d+1$.
\newline
\newline
We now generalise the construction for $n(b_1,b_2)$ where $d$ is even and $b_1$ is a pair of vertices from the set of $n_1$ (or $n_2$) vertices and $b_2$ is a pair of vertices from the set of $n_2$ (or $n_1$ respectively). Pick $i_0$ vertices out of the set of $n_1$ vertices if $d$ is even or out of the set of $n_2$ vertices if $d$ is odd. Also, pick $i_0'$ out of the set of $n_2$ vertices if $d$ is even or out of the set of $n_1$ of vertices if $d$ is odd. Also, pick another vertex $v$ out of the set of $n_1$ vertices and another vertex $'$ out of the set of $n_2$ vertices. Let $C_b''(n_1,n_2,r,s,d,i_0,i_0')$ denote the number of graphs in $A$ such that there is no path from any of the $i_0$ vertices to vertex $v$ that consists of at most $d$ dges, as well as fulfilling the requirement that there is no path from any of the $i_0'$ vertices to the vertex $v'$ that consists of at most $d$ edges. Suppose $i_0,i_0'>0$. If $d$ is odd, then we have
\begin{align}
C_b''(n_1,n_2,r,s,d+1,i_0,i_0')&<\sum_{i_1=0}^{n_2-i_0'-1}\binom{n_2-i_0'-1}{i_1}\left(s^{i_0}-(s-r)^{i_0}\right)^{i_1}(s-r)^{i_0(n_2-i_0'-i_1-1)}\nonumber\\
&\quad\cdot\sum_{i_1'=0}^{n_1-i_0-1}\binom{n_1-i_0-1}{i_1'}\left(s^{i_0'}-(s-r)^{i_0'}\right)^{i_1'}(s-r)^{i_0'(n_1-i_0-i_1'-1)}\nonumber\\
&\quad\cdot\sum_{l_1=0}^{i_0}\sum_{l_2=0}^{i_0'}\binom{i_0}{l_1}\binom{i_0'}{l_2}r^{l_1+l_2}(s-r)^{i_0-l_1+i_0'-l_2+l_1i_0'+l_2i_0-l_1l_2}s^{i_0i_0'-l_1i_0'-l_2i_0+l_1l_2}\nonumber\\
&\quad\cdot C_b''(n_1-i_0,n_2-i_0',r,s,d,i_1,i_1')\label{eqn29}
\end{align}
valid for all $1\leq i_0\leq n_1-1$, $1\leq i_0'\leq n_2-1$. If $d$ is even, then we have
\begin{align}
C_b''(n_1,n_2,r,s,d+1,i_0,i_0')&<\sum_{i_1=0}^{n_1-i_0'-1}\binom{n_1-i_0'-1}{i_1}\left(s^{i_0}-(s-r)^{i_0}\right)^{i_1}(s-r)^{i_0(n_1-i_0'-i_1)}s^{i_0i_0'}\nonumber\\
&\quad\cdot\sum_{i_1'=0}^{n_2-i_0-1}\binom{n_2-i_0-1}{i_1'}\left(s^{i_0'}-(s-r)^{i_0'}\right)^{i_1'}(s-r)^{i_0'(n_2-i_0-i_1')}\nonumber\\
&\quad\cdot C_b''(n_1-i_0',n_2-i_0,r,s,d,i_1,i_1')\label{eqn30}
\end{align}
valid for all $1\leq i_0\leq n_2-1$, $1\leq i_0'\leq n_1-1$. As well,
\begin{equation*}
C_b''(n_1,n_2,r,s,1,i_0,i_0')=(s-r)^{i_0+i_0'}s^{n_1n_2-i_0-i_0'}
\end{equation*}
for all $0\leq i_0\leq n_2-1$, $0\leq i_0'\leq n_1-1$, completing the formula. Then we can deduce that $C_b(n_1,n_2,r,s,d,1,1)=n(b_1,b_2)$ if we are working with diameter $d$. Let $D_b'(n_1,n_2,p,d,i_0,i_0')=\frac{C_b''(n_1n_2,r,s,d,i_0,i_0')}{s^{n_1n_2}}$ so that $D_b'(n_1,n_2,r,s,d,i_0,i_0')$ is the probability that the edge distance between $v$ and any of the $i_0$ vertices is greater than $d$ and that the edge  distance between $v'$ and any of the $i_0'$ vertices is greater than $d$. We claim that
\begin{align}
D_b''(n_1,n_2,p,d,i_0,i_0')&<h(n,p,d,i_0+i_0')^2\nonumber\\
&\quad\cdot\left(1-f(n,p,d,i_0+i_0')\right)^{i_0g(n_1,n_2,p,d,d',i_0,i_0')+i_0'g(n_2,n_1,p,d,d',i_0',i_0)}\label{upperdegbbipartite2}
\end{align}
if $d\geq 3$ is odd, $1\leq\frac{n_1-1}{i_0'+4npi_0+(4np)^2i_0'+(4np)^3i_0+\ldots+(4np)^{d'}i_0''}$, and $1\leq\frac{n_2-1}{i_0+4npi_0'+(4np)^2i_0+(4np)^3i_0'+\ldots+(4np)^{d'}i_0''}$ where $i_0''=i_0$ or $i_0''=i_0'$ depending on the parity of $d'$ and the inequality in question. Also,
\begin{align}
D_b''(n_1,n_2,p,d,i_0,i_0')&<h(n,p,d,i_0+i_0')^2\nonumber\\
&\quad\cdot\left(1-f(n,p,d,i_0+i_0')\right)^{i_0g(n_1,n_2p,d,d',i_0,i_0')+i_0'g(n_2,n_1,p,d,d',i_0',i_0)}\label{upperdegbbipartite3}
\end{align}
if $d\geq 2$ is even, $1\leq\frac{n_1-1}{i_0+4npi_0'+(4np)^2i_0+(4np)^3i_0'+\ldots+(4np)^{d'}i_0''}$, and $1\leq\frac{n_2-1}{i_0'+4npi_0+(4np)^2i_0'+(4np)^3i_0+\ldots+(4np)^{d'}i_0''}$ where $i_0''=i_0$ or $i_0''=i_0'$ depending on the parity of $d'$ and the inequality in question.
\newline
\newline
We prove by induction on $d$ in the same way that we proved \eqref{upperdegb} and \eqref{upperdegbbipartite}. For illustration purposes, we prove the base case $d=2$ and one of the cases for the induction step. First, two lemmas:
\begin{lemma}\label{lem4}
Let $x_1,x_2\geq 0$. Then we have
\begin{equation*}
\left(1-p+p(1-p)^{x_1}\right)\left(1-p+p(1-p)^{x_2}\right)\leq 1-p+p(1-p)^{x_1+x_2}.
\end{equation*}
\end{lemma}
\begin{proof}
Note that
\begin{equation*}
0<\left(1-(1-p)^{x_1}\right)\left(1-(1-p)^{x_2}\right)=1-(1-p)^{x_1}-(1-p)^{x_2}+(1-p)^{x_1+x_2}
\end{equation*}
so that
\begin{align*}
\left(1-p+p(1-p)^{x_1}\right)\left(1-p+p(1-p)^{x_2}\right)&=(1-p)^2+p(1-p)^{x_1+1}+p(1-p)^{x_2+1}+p^2(1-p)^{x_1+x_2}\\
&\leq(1-p)^2+p(1-p)\left((1-p)^{x_1+x_2}+1\right)+p^2(1-p)^{x_1+x_2}\\
&=1-p+p(1-p)^{x_1+x_2}.
\end{align*}
\end{proof}
\begin{lemma}\label{lem5}
Let $r_1,r_2\in\mathbb{N}$. Then
\begin{equation*}
\sum_{m_1=0}^{r_1}\sum_{m_2=0}^{r_2}\binom{r_1}{m_1}\binom{r_2}{m_2}p^{m_1+m_2}(1-p)^{r_1-m_1+r_2-m_2+m_1r_2+m_2r_1-m_1m_2}<1-p+p(1-p)^{r_1+r_2}.    
\end{equation*}
\end{lemma}
\begin{proof}
Wlog we may assume that $r_1\leq r_2$. We have
\begin{align*}
&\quad\sum_{m_1=0}^{r_1}\sum_{m_2=0}^{r_2}\binom{r_1}{m_1}\binom{r_2}{m_2}p^{m_1+m_2}(1-p)^{r_1-m_1+r_2-m_2+m_1r_2+m_2r_1-m_1m_2}\\
&=\sum_{m_1=0}^{r_1}\binom{r_1}{m_1}p^{m_1}(1-p)^{r_1-m_1+r_2+m_1r_2}\left(1+p(1-p)^{r_1-m_1-1}\right)^{r_2}.
\end{align*}
If $r_1=1$, then we have
\begin{align*}
&\quad\sum_{m_1=0}^{1}\binom{1}{m_1}p^{m_1}(1-p)^{1-m_1+r_2+m_1r_2}\left(1+p(1-p)^{1-m_1-1}\right)^{r_2}\\
&=(1-p)^{1+r_2}(1+p)+p(1-p)^{2r_2}\\
&<1-p+p(1-p)^{r_2+1}.
\end{align*}
If $r_1\geq 2$, then we have the following:
\begin{align*}
&\quad\sum_{m_1=0}^{r_1}\binom{r_1}{m_1}p^{m_1}(1-p)^{r_1-m_1+r_2+m_1r_2}\left(1+p(1-p)^{r_1-m_1-1}\right)^{r_2}\\
&\leq(1-p)^{r_1+r_2}\sum_{m_1=0}^{r_1}\binom{r_1}{m_1}p^{m_1}(1-p)^{(r_2-1)m_1}\\
&=(1-p)^{r_2}\left(1-p+p(1-p)^{r_2}\right)^{r_1}\\
&<\left(1-p+p(1-p)^{r_2}\right)^2\\
&\leq 1-p+p(1-p)^{2r_2}\\
&\leq 1-p+p(1-p)^{r_1+r_2}
\end{align*}
with the second last inequality following from Lemma \ref{lem4}.
\end{proof}
For $d=2$, we have
\begin{align*}
D_b''(n_1,n_2,p,2,i_0,i_0')&<(1-p)^{i_0(n_2-i_0'-1)}\sum_{i_1=0}^{n_2-i_0'-1}\binom{n_2-i_0'-1}{i_1}\left((1-p)^{-i_0}-1\right)^{i_1}\\
&\quad\cdot(1-p)^{i_0(n_1-i_0-1)}\sum_{i_1'=0}^{n_1-i_0-1}\binom{n_1-i_0-1}{i_1}\left((1-p)^{-i_0'}-1\right)^{i_1'}(1-p)^{i_1+i_1'}\\
&\quad\cdot\sum_{l_1=0}^{i_0}\sum_{l_2=0}^{i_0'}\binom{i_0}{l_1}\binom{i_0'}{l_2}p^{l_1+l_2}(1-p)^{i_0-l_1+i_0'-l_2+l_1i_0'+l_2i_0-l_1l_2}\\
&<\left(1-p+p(1-p)^{i_0}\right)^{n_2-i_0'-1}\left(1-p+p(1-p)^{i_0'}\right)^{n_1-i_0-1}\\
&\quad\cdot\left(1-p+p(1-p)^{i_0+i_0'}\right)\\
&<\left(1-f(n,p,2,i_0)^{i_0(n_2-i_0'-1)}f(n,p,2,i_0')\right)^{i_0'(n_1-i_0-1)}\left(1-f(n,p,2,i_0+i_0')^{i_0+i_0'}\right)\\
&<\left(1-f(n,p,2,i_0+i_0')\right)^{i_0(n_2-i_0')+i_0'(n_1-i_0)}
\end{align*}
with the second inequality following from Lemma \ref{lem5}. Suppose for some odd $d\geq 3$, \eqref{upperdegbbipartite2} holds for all $i_0,i_0'$ in the stated ranges, and $0<p<1$. We will prove \eqref{upperdegbbipartite3} holds for $d+1$. We have
\begin{align*}
D_b''(n_1,n_2,p,d+1,i_0,i_0')&<(1-p)^{i_0(n_2-i_0'-1)+i_0'(n_1-i_0-1)}\left(1-p+p(1-p)^{i_0+i_0'}\right)\\
&\quad\cdot\sum_{i_1=0}^{n_2-i_0'-1}\binom{n_2-i_0'-1}{i_1}\left((1-p)^{-i_0}-1\right)^{i_1}\\
&\quad\cdot\sum_{i_1'=0}^{n_1-i_0-1}\binom{n_1-i_0-1}{i_1'}\left((1-p)^{-i_0'}-1\right)^{i_1'}\\
&\quad\cdot D_b''(n_1-i_0,n_2-i_0',r,s,d,i_1,i_1').
\end{align*}
The case that we will prove follows.
\setcounter{case}{0}
\begin{case}{$\frac{n-1}{1+4np}<i_0\leq n-1$.}
\normalfont
\newline
\newline
We have the following:
\begin{align*}
D_b''(n_1,n_2,p,d+1,i_0,i_0')&<(1-p)^{i_0(n_2-i_0'-1)+i_0'(n_1-i_0-1)}\left(1-p+p(1-p)^{i_0+i_0'}\right)\\
&\quad\cdot\left(1+\sum_{i_1=1}^{n_2-i_0'-1}\binom{n_2-i_0'-1}{i_1}\left((1-p)^{-i_0}-1\right)^{i_1}\right.\\
&\qquad\qquad\cdot\left(1+\sum_{i_1'=1}^{n_1-i_0-1}\binom{n_1-i_0-1}{i_1'}\left((1-p)^{-i_0'}-1\right)^{i_1'}\right.\\
&\qquad\qquad\qquad\cdot h(n-i_0-i_0',p,d,i_1+i_1')h(n-i_0-i_0',p,d,i_1+i_1')\\
&\qquad\qquad\qquad\cdot\left(1-f(n-i_0-i_0',p,d,i_1+i_1')\right)^{i_1g(n_2-i_0',n_1-i_0,p,d,0,i_1,i_1')}\\
&\qquad\qquad\qquad\cdot\left(1-f(n-i_0-i_0',p,d,i_1+i_1')\right)^{i_1'g(n_1-i_0,n_2-i_0',p,d,0,i_1',i_1)}.
\end{align*}
We can deduce that $h(n-i_0-i_0',p,d,i_1+i_1')\leq h(n-i_0-i_0',p,d,4np(i_0+i_0'))$. Also, from Lemma \ref{lem3}, we can deduce that $f(n,p,d,4np(i_0+i_0'))<f(n-i_0-i_0',p,d,i_1+i_1')$. As well, $g(n_2-i_0',n_1-i_0,p,d,0,n_1-i_0-1,n_2-i_0'-1)\leq g(n_2-i_0',n_1-i_0,p,d,d',i_1',i_1)$ and $g(n_1-i_0,n_2-i_0',p,d,0,n_2-i_0'-1,n_1-i_0-1)\leq g(n_1-i_0,n_2-i_0',p,d,d',i_1,i_1')$. Thus we have
\begin{align*}
&\quad D_b''(n_1,n_2,p,d+1,i_0,i_0')\\
&<h(n-i_0-i_0',p,d,4np(i_0+i_0'))^2\\
&\quad\cdot(1-p)^{i_0(n_2-i_0'-1)+i_0'(n_1-i_0-1)}\left(1-p+p(1-p)^{i_0+i_0'}\right)\\
&\quad\cdot\left(1+\sum_{i_1=1}^{n_2-i_0'-1}\binom{n_2-i_0'-1}{i_1}\left((1-p)^{-i_0}-1\right)^{i_1}\right.\\
&\qquad\qquad\left.\cdot\left(1-f(n,p,d,i_1+i_1')\right)^{i_1g(n_2-i_0',n_1-i_0,p,d,0,n_2-i_0'-1,n_1-i_0-1)}\right)\\
&\quad\cdot\left(1+\sum_{i_1'=1}^{n_1-i_0-1}\binom{n_1-i_0-1}{i_1'}\left((1-p)^{-i_0'}-1\right)^{i_1'}\right.\\
&\qquad\qquad\left.\cdot\left(1-f(n,p,d,i_1+i_1')\right)^{i_1'g(n_1-i_0,n_2-i_0',p,d,0,n_1-i_0-1,n_2-i_0'-1)}\right)\\
&<h(n,p,d+1,i_0+i_0')^2\\
&\quad\cdot(1-p)^{i_0(n_2-i_0'-1)+i_0'(n_1-i_0-1)}\left(1-p+p(1-p)^{i_0+i_0'}\right)\\
&\quad\cdot\left(1+\left((1-p)^{-i_0}-1\right)\left(1-f(n,p,d,4np(i_0+i_0'))\right)^{g(n_2-i_0',n_1-i_0,p,d,0,n_2-i_0'-1,n_1-i_0-1)}\right)^{n_2-i_0'-1}\\
&\quad\cdot\left(1+\left((1-p)^{-i_0'}-1\right)\left(1-f(n,p,d,4np(i_0+i_0'))\right)^{g(n_1-i_0,n_2-i_0',p,d,0,n_1-i_0-1,n_2-i_0'-1)}\right)^{n_1-i_0-1}\\
&<h(n,p,d+1,i_0+i_0')^2\left(1-p+p(1-p)^{i_0+i_0'}\right)\\
&\quad\cdot\left((1-p)^{i_0}+\left(1-(1-p)^{i_0}\right)\left(1-f(n,p,d,4np(i_0+i_0'))\right)^{g(n_2-i_0',n_1-i_0,p,d,0,n_2-i_0'-1,n_1-i_0-1)}\right)^{n_2-i_0'-1}\\
&\quad\cdot\left((1-p)^{i_0'}+\left(1-(1-p)^{i_0'}\right)\left(1-f(n,p,d,4np(i_0+i_0'))\right)^{g(n_1-i_0,n_2-i_0',p,d,0,n_1-i_0-1,n_2-i_0'-1)}\right)^{n_1-i_0-1}.
\end{align*}
We note that $g(n_2-i_0',n_1-i_0,p,d,0,n_1-i_0-1,n_2-i_0'-1),g(n_1-i_0,n_2-i_0',p,d,0,n_2-i_0'-1,n_1-i_0-1)<n^{d-1}$ and so, using Lemma \ref{lem1}, we thus have
\begin{align*}
D_b''(n_1,n_2,p,d+1,i_0,i_0')&<h(n,p,d+1,i_0+i_0')^2\\
&\quad\cdot\left(1-f(n,p,d+1,i_0+i_0')\right)^{i_0(n_2-i_0'-1)g(n_2-i_0',n_1-i_0,p,d,0,n_2-i_0'-1,n_1-i_0-1)+i_0}\\
&\quad\cdot\left(1-f(n,p,d+1,i_0+i_0')\right)^{i_0'(n_1-i_0-1)g(n_1-i_0,n_2-i_0',p,d,0,n_1-i_0-1,n_2-i_0'-1)+i_0'}.
\end{align*}
We can deduce that $(n_2-i_0'-1)g(n_2-i_0',n_1-i_0,p,d,0,n_2-i_0'-1,n_1-i_0-1)+1\geq g(n_1,n_2,p,d+1,0,i_0,i_0')$ and $(n_1-i_0-1)g(n_1-i_0,n_2-i_0',p,d,0,n_1-i_0-1,n_2-i_0'-1)+1\geq g(n_2,n_1,p,d+1,0,i_0',i_0)$ and so we have \eqref{upperdegbbipartite3}.
\end{case}
The rest of the cases can be proved similarly.
\newline
\newline
We can deduce that $D_b(n,n_j,p,d,i_0)\leq D_b(n-1,n_j-1,p,d,i_0)$. Thus we have $n(b_1,b_2)\leq s^{n_1n_2}D_b(n-1,n_j-1,p,d,2)$ whenever $b_1$ and $b_2$ are not the same pair of vertices with each pair having at least one vertex in the partite set consisting of $n_j$ vertices. Also, when $d$ is even, and $b_1$ and $b_2$ are two pairs of vertices with one pair having its vertices in the partite set consisting of $n_1$ vertices and the other pairs having its vertices in the partite set consisting of $n_2$ vertices we have $n(b_1,b_2)\leq s^{n_1n_2}D_b''(n_1n_2,p,d,1,1)$. Hence we can use \eqref{upperdegbbipartite} and \eqref{upperdegbbipartite3} to get an upper bound. Thus, by the Tur\'an sieve, for even $d\geq2$, we have
\begin{align*}
&\quad P(G(n_1,n_2,p),d)\\
&<\frac{\sum_{j=1}^2\left(\binom{n_j}{2}^2D_b(n-1,n_j-1,p,d,2)+\binom{n_j}{2}D_b(n,n_j,p,d,1)\right)+2\binom{n_1}{2}\binom{n_2}{2}D''(n_1,n_2,p,d,1,1)}{\left(\sum_{j=1}^2\binom{n_j}{2}D_b(n,n_j,p,d,1)\right)^2}-1\\
&=\frac{\sum_{j=1}^2\binom{n_j}{2}^2D_b(n-1,n_j-1,p,d,2)+2\binom{n_1}{2}\binom{n_2}{2}D''(n_1,n_2,p,d,1,1)}{\left(\sum_{j=1}^2\binom{n_j}{2}D_b(n,n_j,p,d,1)\right)^2}\\
&\quad-1+\frac{1}{\sum_{j=1}^2\binom{n_j}{2}D_b(n,n_j,p,d,1)}\\
&<\left(\sum_{j=1}^2\binom{n_j}{2}(1-p^d)^{(n-n_j)\left((n_1n_2)^{\frac{d-2}{2}}+\sum_{l=1}^{\frac{d-2}{2}}(n_1n_2)^{\frac{d-2}{2}-l}\left(p^{1-2l}+p^{-2l}\right)\right)}\right)^{-2}\\
&\quad\cdot\left(\left(\sum_{j=1}^2\binom{n_j}{2}^2h(n-1,p,d,2)\left(1-f(n-1,p,d,2)\right)^{2\cdot g_b(n-1,n_j-1,p,d,d',2)}\right)\right.\\
&\qquad\quad\left.+2\binom{n_1}{2}\binom{n_2}{2}h(n,p,d,2)^2\left(1-f(n,p,d,2)\right)^{g_b'(n_2,n_1,p,d,d',1,1)+g_b'(n_1,n_2,p,d,d',1,1)}\right)\\
&\quad-1\\
&\quad+\frac{1}{\sum_{i=1}^2\binom{n_i}{2}(1-p^d)^{(n-n_j)\left((n_1n_2)^{\frac{d-2}{2}}+\sum_{j=1}^{\frac{d-2}{2}}(n_1n_2)^{\frac{d-2}{2}-j}\left(p^{1-2j}+p^{-2j}\right)\right)}}
\end{align*}
and for odd $d\geq 3$, we have
\begin{align*}
P(G(n_1,n_2,p),d)&<\frac{n_1^2n_2^2D_b(n-1,n_1-1,p,d,2)+n_1n_2D_b(n,n_1,p,d,1)}{n_1^2n_2^2D_b(n,n_1,p,d,1)^2}-1\\
&=\frac{D_b(n-1,n_1-1,p,d,2)}{D_b(n,n_1,p,d,1)^2}-1+\frac{1}{n_1n_2D_b(n,n_1,p,d,1)}\\
&<(1-p^d)^{-2\left((n_1n_2)^{\frac{d-1}{2}}+\sum_{j=1}^{\frac{d-1}{2}}(n_1n_2)^{\frac{d-1}{2}-j}\left(p^{1-2j}+p^{-2j}\right)\right)}h(n-1,p,d,2)\\
&\quad\cdot f(n-1,p,d,2)^{2\cdot g_b(n-1,n_j-1,p,d,d',2)}\\
&\quad-1+\frac{1}{n_1n_2(1-p^d)^{\left((n_1n_2)^{\frac{d-1}{2}}+\sum_{j=1}^{\frac{d-1}{2}}(n_1n_2)^{\frac{d-1}{2}-j}\left(p^{1-2j}+p^{-2j}\right)\right)}}.
\end{align*}
\section{Restricted Results for Bipartite Graphs for diameter $d\geq 3$}
We impose further restrictions on $n_1,n_2$, and $p$ in Theorem \ref{bigthmbipartite} to make our result more clear and meaningful. The result is Corollary \ref{bigcor2}.
\begin{corollary}\label{bigcor2}
Let $d\geq 3$ be fixed. Suppose that \eqref{cond1} and \eqref{cond2} hold. Also suppose that $n_1\leq n_2$ and
\begin{equation}\label{cond3}
n^{1-\frac{2}{d}-\frac{1}{d^2}}\leq n_1.
\end{equation}
Suppose $d$ is even. Then we have
\begin{equation*}
P(G(n,p),d)>1-\binom{n_2}{2}\left(1+4^{d+1}dn^{\frac{-1}{2d^2}}\right)(1-p^d)^{n_1^{d/2}n_2^{d/2-1}}\left(1+\frac{n_1(n_1-1)(1-p^d)^{(n_1n_2)^{d/2-1}(n_2-n_1)}}{n_2(n_2-1)}\right).
\end{equation*}
and
\begin{align*}
P(G(n,p),d)&<\frac{2(1-p^d)^{-n_1^{d/2}n_2^{d/2-1}}\left(1+2n^{\frac{-1}{2d^2}}\right)}{n_2(n_2-1)}\left(1+\frac{n_1(n_1-1)(1-p^d)^{(n_1n_2)^{d/2-1}(n_2-n_1)}}{n_2(n_2-1)}\right)^{-1}\\
&\quad+4^{d+3}dn^{\frac{-1}{2d^2}}.
\end{align*}
Suppose $d$ is odd. Then we have
\begin{equation*}
P(G(n_1,n_2,p),d)>1-n_1n_2\left(1+4^{d+1}dn^{\frac{-1}{2d^2}}\right)(1-p^d)^{(n_1n_2)^{\frac{d-1}{2}}}.
\end{equation*}
and
\begin{equation*}
P(G(n_1,n_2,p),d)<\frac{(1-p^d)^{-(n_1n_2)^{\frac{d-1}{2}}}\left(1+2n^{\frac{-1}{2d^2}}\right)}{n_1n_2}+2\cdot 4^{d+2}dn^{\frac{-1}{2d^2}}.
\end{equation*}
\end{corollary}
We prove Corollary \ref{bigcor2}. Suppose \eqref{cond1}, \eqref{cond2}, and \eqref{cond2} all hold. As in the proof of Corollary \ref{bigcor}, we derive \eqref{eqn1} and \eqref{eqn3}. As well, we can derive that
\begin{equation*}
\frac{16(4np)^{d-3}}{7}\leq n_1-1.
\end{equation*}
Thus we can apply Theorem \ref{bigthmbipartite} for $d'=d-3$. Suppose $d$ is even. Then $d\geq 4$. From \eqref{cond1}, \eqref{cond2}, and \eqref{cond3}, we can derive
\begin{align}
g_b(n,n_j,p,d,d-3,1)&>\left(n-n_j-\frac{(4np)^{d-3}}{1-\frac{1}{16n^2p^2}}\right)^{d/2}\left(n_j-1-\frac{(4np)^{d-4}}{1-\frac{1}{16n^2p^2}}\right)^{(d-2)/2}\nonumber\\
&>(n-n_j)^{d/2}n_j^{d/2-1}\left(1-\frac{4^d(np)^{d-3}}{63(n-n_j)}-\frac{1}{n_j}-\frac{4^d(np)^{d-4}}{252n_j}\right)^{d/2}\nonumber\\
&>(n-n_j)^{d/2}n_j^{d/2-1}\left(1-\frac{4^d(np)^{d-3}}{63(n-n_j)}-\frac{4^d(np)^{d-3}}{252n_j}\right)^{d/2}\nonumber\\
&>(n-n_j)^{d/2}n_j^{d/2-1}\left(1-\frac{4^dn^{\frac{-1}{2d}-\frac{1}{2d^2}}}{63}-\frac{4^dn^{\frac{-1}{2d}-\frac{1}{2d^2}}}{252}\right)^{d/2}\nonumber\\
&>(n-n_j)^{d/2}n_j^{d/2-1}\left(1-4^{d-3}dn^{\frac{-1}{2d}-\frac{1}{2d^2}}\right)\label{eqn13},
\end{align}
\begin{align}
g_b(n-1,n_j-1,p,d,d-3,2)&>\left(n-n_j-\frac{2(4np)^{d-3}}{1-\frac{1}{16n^2p^2}}\right)^{d/2}\left(n_j-2-\frac{2(4np)^{d-4}}{1-\frac{1}{16n^2p^2}}\right)^{(d-2)/2}\nonumber\\
&>(n-n_j)^{d/2}n_j^{d/2-1}\left(1-2\cdot 4^{d-3}dn^{\frac{-1}{2d}-\frac{1}{2d^2}}\right)\label{eqn14},
\end{align}
\begin{align}
g_b'(n_2,n_1,p,d,d-3,1,1)&>\left(n_1-1-\frac{(4np)^{d-3}}{1-\frac{1}{4np}}\right)^{d/2}\left(n_2-1-\frac{(4np)^{d-3}}{1-\frac{1}{4np}}\right)^{d/2-1}\nonumber\\
&>n_1^{d/2}n_2^{d/2-1}\left(1-\frac{1}{n_1}-\frac{8(4np)^{d-3}}{7n_1}-\frac{1}{n_2}-\frac{8(4np)^{d-3}}{7n_2}\right)^{d/2}\nonumber\\
&>n_1^{d/2}n_2^{d/2-1}\left(1-\frac{9(4np)^{d-3}}{7n_1}-\frac{9(4np)^{d-3}}{7n_2}\right)^{d/2}\nonumber\\
&>n_1^{d/2}n_2^{d/2-1}\left(1-\frac{9\cdot 4^{d-3}n^{\frac{-1}{2d}-\frac{1}{2d^2}}}{7}-\frac{18n^{\frac{-5}{2d}-\frac{3}{2d^2}}}{7}\right)^{d/2}\nonumber\\
&>n_1^{d/2}n_2^{d/2-1}\left(1-\frac{9\cdot 4^{d-3}n^{\frac{-1}{2d}-\frac{1}{2d^2}}}{7}-\frac{18n^{\frac{-1}{2d}-\frac{1}{2d^2}}}{7\cdot 2^{16}}\right)^{d/2}\nonumber\\
&>n_1^{d/2}n_2^{d/2-1}\left(1-4^{d-3}dn^{\frac{-1}{2d}-\frac{1}{2d^2}}\right)\label{eqn18},
\end{align}
and
\begin{equation}
g_b'(n_1,n_2,p,d,d-3,1,1)>n_1^{d/2-1}n_2^{d/2}\left(1-4^{d-3}dn^{\frac{-1}{2d}-\frac{1}{2d^2}}\right).\label{eqn19}
\end{equation}
Substituting in \eqref{eqn1}, \eqref{eqn3}, and \eqref{eqn13} into the lower bound in Theorem \ref{bigthmbipartite}, we obtain
\begin{align}
P(G(n_1,n_2,p),d)&>1-\sum_{j=1}^{2}\binom{n_j}{2}\left(1+4^ddn^{\frac{-1}{2d^2}}\right)(1-p^d)^{(n-n_j)^{d/2}n_j^{d/2-1}\left(1-4^{d-1}dn^{\frac{-1}{2d}-\frac{1}{2d^2}}\right)\left(1-4^{d-2}dn^{\frac{-1}{2d}-\frac{1}{2d^2}}\right)}\nonumber\\
&>1-\sum_{j=1}^{2}\binom{n_j}{2}\left(1+4^ddn^{\frac{-1}{2d^2}}\right)(1-p^d)^{(n-n_j)^{d/2}n_j^{d/2-1}\left(1-2\cdot 4^{d-1}dn^{\frac{-1}{2d}-\frac{1}{2d^2}}\right)}\label{eqn15}.
\end{align}
As in the proof of Corollary \ref{bigcor}, we deduce
\begin{align*}
P(G(n_1,n_2,p),d)&>1-\sum_{j=1}^{2}\binom{n_j}{2}\left(1+4^{d+1}dn^{\frac{-1}{2d^2}}\right)(1-p^d)^{(n-n_j)^{d/2}n_j^{d/2-1}}\\
&=1-\binom{n_2}{2}\left(1+4^{d+1}dn^{\frac{-1}{2d^2}}\right)(1-p^d)^{n_1^{d/2}n_2^{d/2-1}}\left(1+\frac{n_1(n_1-1)(1-p^d)^{(n_1n_2)^{d/2-1}(n_2-n_1)}}{n_2(n_2-1)}\right).
\end{align*}
Also, from \eqref{cond2}, we have
\begin{align*}
&\quad(1-p^d)^{-2n\left(\sum_{l=1}^{\frac{d-2}{2}}(n_1n_2)^{\frac{d-2}{2}-l}\left(p^{1-2l}+p^{-2l}\right)\right)}\\
&<(1-p^d)^{-2n\left(\sum_{l=1}^{\infty}n^{d-2-2l}\left(p^{1-2l}+p^{-2l}\right)\right)}\\
&<(1-p^d)^{\frac{-16n^{d-3}p^{-2}}{3}}.
\end{align*}
Also, from \eqref{cond1} and \eqref{cond2}, we have
\begin{equation*}
\frac{16p^{d-2}n^{d-3}}{3}\leq\frac{16n^{\frac{-3}{2d}-\frac{1}{d^2}}}{3}<\frac{16n^{\frac{-1}{2d^2}}}{3}<\frac{1}{3\cdot 4^{d-2}d}\leq\frac{1}{192}.
\end{equation*}
Thus we can deduce
\begin{equation*}
(1-p^d)^{\frac{-16n^{d-3}p^{-2}}{3}}<1+\frac{192n^{\frac{-1}{2d^2}}}{191}.
\end{equation*}
Thus
\begin{align}
&\quad\left(\sum_{j=1}^2\binom{n_j}{2}(1-p^d)^{(n-n_j)\left((n_1n_2)^{\frac{d-2}{2}}+\sum_{l=1}^{\frac{d-2}{2}}(n_1n_2)^{\frac{d-2}{2}-l}\left(p^{1-2l}+p^{-2l}\right)\right)}\right)^{-2}\nonumber\\
&<\quad\left(\sum_{j=1}^2\binom{n_j}{2}(1-p^d)^{(n-n_j)(n_1n_2)^{\frac{d-2}{2}}}\right)^{-2}\left(1+\frac{192n^{\frac{-1}{2d^2}}}{191}\right)\label{eqn16}.
\end{align}
Similarly, we obtain
\begin{align}
&\quad\left(\sum_{j=1}^2\binom{n_j}{2}(1-p^d)^{(n-n_j)\left((n_1n_2)^{\frac{d-2}{2}}+\sum_{l=1}^{\frac{d-2}{2}}(n_1n_2)^{\frac{d-2}{2}-l}\left(p^{1-2l}+p^{-2l}\right)\right)}\right)^{-1}\nonumber\\
&<\quad\left(\sum_{j=1}^2\binom{n_j}{2}(1-p^d)^{(n-n_j)(n_1n_2)^{\frac{d-2}{2}}}\right)^{-1}\left(1+\frac{384n^{\frac{-1}{2d^2}}}{383}\right)\label{eqn20}.
\end{align}
Note that by \eqref{eqn1}, \eqref{eqn3}, \eqref{eqn9}, \eqref{eqn14}, \eqref{eqn18}, and \eqref{eqn19} we have
\begin{align}
&\quad h(n-1,p,d,2)\left(1-f(n-1,p,d,2)\right)^{2\cdot g_b(n-1,n_j-1,p,d,d-3,2)}\nonumber\\
&<\left(1+4^ddn^{\frac{-1}{2d^2}}\right)(1-p^d)^{2\left(1-4^{d-1}dn^{\frac{-1}{2d}-\frac{1}{2d^2}}\right)(n-n_j)^{d/2}n_j^{d/2-1}\left(1-2\cdot 4^{d-3}dn^{\frac{-1}{2d}-\frac{1}{2d^2}}\right)}\nonumber\\
&<\left(1+4^ddn^{\frac{-1}{2d^2}}\right)(1-p^d)^{2\left(1-2\cdot 4^{d-1}dn^{\frac{-1}{2d}-\frac{1}{2d^2}}\right)(n-n_j)^{d/2}n_j^{d/2-1}}\nonumber\\
&<\left(1+4^ddn^{\frac{-1}{2d^2}}\right)^3(1-p^d)^{2(n-n_j)^{d/2}n_j^{d/2-1}}\label{eqn17}
\end{align}
and
\begin{align}
&\quad h(n,p,d,2)^2\left(1-f(n,p,d,1)\right)^{g_b'(n_2,n_1,p,d,d-3,1,1)+g_b'(n_1,n_2,p,d,d-3,1,1)}\nonumber\\
&<\left(1+4^ddn^{\frac{-1}{2d^2}}\right)^2(1-p^d)^{2\left(1-4^{d-1}dn^{\frac{-1}{2d}-\frac{1}{2d^2}}\right)\left(n_1^{d/2}n_2^{d/2-1}+n_1^{d/2-1}n_2^{d/2}\right)\left(1-4^{d-3}dn^{\frac{-1}{2d}-\frac{1}{2d^2}}\right)}\nonumber\\
&<\left(1+4^ddn^{\frac{-1}{2d^2}}\right)^4(1-p^d)^{n_1^{d/2}n_2^{d/2-1}+n_1^{d/2-1}n_2^{d/2}}.\label{eqn21}
\end{align}
Substituting \eqref{eqn16}, \eqref{eqn20}, \eqref{eqn17}, and \eqref{eqn21} into the upper bound in Theorem \ref{bigthmbipartite}, we obtain
\begin{align*}
P(G(n_1,n_2,p),d)&<\left(1+4^ddn^{\frac{-1}{2d^2}}\right)^5-1\\
&\quad+\frac{2(1-p^d)^{-n_1^{d/2}n_2^{d/2-1}}\left(1+\frac{384n^{\frac{-1}{2d^2}}}{383}\right)}{n_2(n_2-1)}\left(1+\frac{n_1(n_1-1)(1-p^d)^{(n_1n_2)^{d/2-1}(n_2-n_1)}}{n_2(n_2-1)}\right)^{-1}\\
&<\frac{2(1-p^d)^{-n_1^{d/2}n_2^{d/2-1}}\left(1+\frac{384n^{\frac{-1}{2d^2}}}{383}\right)}{n_2(n_2-1)}\left(1+\frac{n_1(n_1-1)(1-p^d)^{(n_1n_2)^{d/2-1}(n_2-n_1)}}{n_2(n_2-1)}\right)^{-1}\\
&\quad+4^{d+3}dn^{\frac{-1}{2d^2}}.
\end{align*}
Suppose $d$ is odd. Similarly to how we derived \eqref{eqn13}, we derive
\begin{align}
g_b(n,n_j,p,d,d-3,1)&>\left(n_2-\frac{(4np)^{d-3}}{1-\frac{1}{16n^2p^2}}\right)^{\frac{d-1}{2}}\left(n_1-1-\frac{(4np)^{d-4}}{1-\frac{1}{16n^2p^2}}\right)^{(d-1)/2}\nonumber\\
&>(n_1n_2)^{\frac{d-1}{2}}\left(1-4^{d-3}dn^{\frac{-1}{2d}-\frac{1}{2d^2}}\right)\label{eqn22}
\end{align}
and
\begin{align}
g_b(n-1,n_1-1,p,d,d-3,2)&>\left(n_2-\frac{2(4np)^{d-3}}{1-\frac{1}{16n^2p^2}}\right)^{\frac{d-1}{2}}\left(n_1-2-\frac{2(4np)^{d-4}}{1-\frac{1}{16n^2p^2}}\right)^{\frac{d-1}{2}}\nonumber\\
&>(n_1n_2)^{\frac{d-1}{2}}\left(1-2\cdot 4^{d-3}dn^{\frac{-1}{2d}-\frac{1}{2d^2}}\right)\label{eqn23}.
\end{align}
Substituting in \eqref{eqn1}, \eqref{eqn3}, and \eqref{eqn22} into the lower bound in Theorem \ref{bigthmbipartite}, we obtain
\begin{align*}
P(G(n_1,n_2,p),d)&>1-n_1n_2\left(1+4^ddn^{\frac{-1}{2d^2}}\right)(1-p^d)^{(n_1n_2)^{\frac{d-1}{2}}\left(1-4^{d-1}dn^{\frac{-1}{2d}-\frac{1}{2d^2}}\right)\left(1-4^{d-3}dn^{\frac{-1}{2d}-\frac{1}{2d^2}}\right)}\\
&>1-n_1n_2\left(1+4^{d+1}dn^{\frac{-1}{2d^2}}\right)(1-p^d)^{(n_1n_2)^{\frac{d-1}{2}}}.
\end{align*}
Similarly to how we derived \eqref{eqn16}, \eqref{eqn20}, and \eqref{eqn17} we also have
\begin{align}
&\quad(1-p^d)^{-2\left((n_1n_2)^{\frac{d-1}{2}}+\sum_{l=1}^{\frac{d-1}{2}}(n_1n_2)^{\frac{d-1}{2}-l}\left(p^{1-2l}+p^{-2l}\right)\right)}\nonumber\\
&<(1-p^d)^{-2(n_1n_2)^{\frac{d-1}{2}}}\left(1+\frac{192n^{\frac{-1}{2d^2}}}{191}\right)\label{eqn24},
\end{align}
\begin{align}
&\quad(1-p^d)^{-\left((n_1n_2)^{\frac{d-1}{2}}+\sum_{l=1}^{\frac{d-1}{2}}(n_1n_2)^{\frac{d-1}{2}-l}\left(p^{1-2l}+p^{-2l}\right)\right)}\nonumber\\
&<(1-p^d)^{-(n_1n_2)^{\frac{d-1}{2}}}\left(1+\frac{384n^{\frac{-1}{2d^2}}}{383}\right)\label{eqn25},
\end{align}
and
\begin{align}
&\quad h(n-1,p,d,2)\left(1-f(n-1,p,d,2)\right)^{2\cdot g_b(n-1,n_j-1,p,d,d-3,2)}\nonumber\\
&<\left(1+4^ddn^{\frac{-1}{2d^2}}\right)^3(1-p^d)^{2(n_1n_2)^{\frac{d-1}{2}}}\label{eqn26}.
\end{align}
Substituting \eqref{eqn24}, \eqref{eqn25}, and \eqref{eqn26} into the upper bound in Theorem \ref{bigthmbipartite}, we obtain
\begin{align*}
P(G(n_1,n_2,p),d)&<\left(1+4^ddn^{\frac{-1}{2d^2}}\right)^4-1+\frac{(1-p^d)^{-(n_1n_2)^{\frac{d-1}{2}}}\left(1+\frac{384n^{\frac{-1}{2d^2}}}{383}\right)}{n_1n_2}\\
&<\frac{(1-p^d)^{-(n_1n_2)^{\frac{d-1}{2}}}\left(1+2n^{\frac{-1}{2d^2}}\right)}{n_1n_2}+4^{d+2}dn^{\frac{-1}{2d^2}}.
\end{align*}
\section{Directed Bipartite Graphs for diameter $d\geq 3$}
Using the above methods, we can obtain similar results about the probability of a random directed bipartite graph with $n_1$ and $n_2$ vertices in the partite sets having diameter $d$ where each directed edge is chosen independently with probability $p$. Furthermore, for any two vertices, say $v_1$ and $v_2$, the existence of the edge from $v_1$ to $v_2$ has probability $p$, while the existence of the edge from $v_2$ to $v_1$ also occurs with probability $p$, and these two edges occur independently. We proceed exactly as above the only changes being as follows. We multiply the second term in \eqref{eqn10c} by $s^{i_0n_j}$, multiply the right-hand side of \eqref{eqn10a} by $s^{i_0n_j}$, multiply the second term in \eqref{eqn10d} by $s^{i_0(n-n_j)}$, multiply the right-hand side of \eqref{eqn10b} by $s^{i_0(n-n_j)}$, replace the factor of $s^{i_1i_0'}$ with $s^{i_1i_0'+(i_0+i_0')(n-n_j)}$ in \eqref{eqn27}, replace the factor of $s^{i_1i_0'+i_0'}$ with $s^{i_1i_0'+i_0'+(i_0+i_0')n_j}$ in \eqref{eqn28}, replace $(s-r)^{i_0-l_1+i_0'-l_2+l_1i_0'+l_2i_0-l_1l_2}$ and $s^{i_0i_0'-l_1i_0'-l_2i_0+l_1l_2}$ with $(s-r)^{i_0-l_1+i_0'-l_2+l_1i_0'+l_2i_0}$ and $s^{i_0(n_2-l_2)+i_0'(n_1-l_1)}$ respectively in \eqref{eqn29}, replace $s^{i_0i_0'}$ with $s^{i_0'n_2+i_0n_1}$ in \eqref{eqn30}, and replace $n_1n_2$, $\binom{n_1}{2}$, $\binom{n_2}{2}$, and $\binom{n_j}{2}$ wherever they occur with $2n_1n_2$, $n_1(n_1-1)$, $n_2(n_2-1)$, and $n_j(n_j-1)$ respectively. The only other extra consideration is in our calculation for $n(b_1,b_2)$ where $d$ is odd and we may have both pairs of vertices each consisting of a vertex from each of the partite sets, but where the paths concerned start at vertices in opposite partite sets. To deal with this case, we would define $C_b'''\left(n,n_j,r,s,d+1,i_0,i_0'\right)$, which we define the same way as the function $C_b'\left(n,n_j,r,s,d+1,i_0,i_0'\right)$, except we consider directed paths from the $i_0$ vertices to vertex $v$, and directed paths from the vertex $v'$ to the $i_0'$ vertices and this case can be dealt with in exactly the same way as $C_b'\left(n,n_j,r,s,d+1,i_0,i_0'\right)$. Consequently, in Theorem \ref{bigthmbipartite} and Corollary \ref{bigcor2}, we multiply the second term of the lower bounds by $2$, divide the last term of the upper bounds in Theorem \ref{bigthmbipartite} by $2$, divide the first upper bound in Corollary \ref{bigcor2} by $2$ and divide the the first term in the second upper bound in Corollary \ref{bigcor2} by $2$ to get the analogous results for random directed bipartite graphs. Everything else is left unchanged.
\section{Acknowledgements}
The results in the paper are part of the Ph.D. thesis of the second author. He would like to thank his co-supervisor, Kevin Hare, and his thesis committee members, Karl Dilcher, David McKinnon, Jeffrey Shallit, and Cam Stewart for their helpful suggestions about this project. He would also like to thank the Azrieli Foundation for the award of an Azrieli International Postdoctoral Fellowship, as well as the University of Calgary also for the award of a Postdoctoral Fellowship.

\end{document}